\documentclass[11pt]{amsart}

\usepackage{amssymb,verbatim}
\usepackage{amsmath,amsfonts}
\usepackage[mathscr]{euscript} 
\usepackage{amsthm}
\usepackage{url}
\usepackage{graphicx} 
\usepackage{enumitem} 
\usepackage{hyperref}

\hypersetup{colorlinks} 
\usepackage{bbm} 
\usepackage{bm} 
\usepackage{mathrsfs} 
\usepackage{cancel} 

\usepackage[colorinlistoftodos]{todonotes}

\RequirePackage{cleveref}
\usepackage{hypcap}
\hypersetup{colorlinks=true, citecolor=darkblue, linkcolor=darkblue}
\definecolor{darkblue}{rgb}{0.0,0,0.7}
\newcommand{\darkblue}{\color{darkblue}}

\definecolor{darkred}{rgb}{0.68,0,0}
\newcommand{\darkred}{\color{darkred}}

\definecolor{darkgreen}{rgb}{0,.28,0}
\newcommand{\darkgreen}{\color{darkgreen}}

\definecolor{magenta}{rgb}{.36, 0, .36}
\newcommand{\magenta}{\color{magenta}}

\newcommand{\defn}[1]{\emph{\darkblue #1}}
\newcommand{\defna}[1]{\emph{\darkred #1}}
\newcommand{\defnb}[1]{\emph{\darkblue #1}}
\newcommand{\defng}[1]{\emph{\darkgreen #1}}
\newcommand{\defnm}[1]{\emph{\magenta #1}}

\setlist[enumerate]{
	label=\textnormal{({\roman*})},
	ref={\roman*}}

\makeatletter
\def\th@plain{%
	\thm@notefont{}
	\itshape 
}
\def\th@definition{%
	\thm@notefont{}
	\normalfont 
}
\makeatother

\newtheorem{thm}{Theorem}[section]
\newtheorem{lemma}[thm]{Lemma}

\newtheorem*{claim*}{Claim}

\newtheorem{conv}[thm]{Convention}

\theoremstyle{definition}
\newtheorem{ex}[thm]{Example}
\newtheorem{rem}[thm]{Remark}
\newtheorem{definition}[thm]{Definition}

\numberwithin{figure}{section}
\numberwithin{equation}{section}

\def\emp{\nothing}

\def\bu{\ast}
\def\sq{{\lozenge}}

\def\zz{\mathbb Z}
\def\nn{\mathbb N}
\def\cc{\mathbb C}

\def\rr{\mathbb R}

\def\ov{\overline}

\def\sm{\smallsetminus}

\def\Ga{\Gamma}

\def\la{\lambda}
\def\ga{\gamma}
\def\si{\sigma}
\def\de{\delta}
\def\ep{\ve}
\def\al{\alpha}
\def\be{\beta}
\def\om{\omega}

\def\ve{\varepsilon}

\def\ala{\al}
\def\beb{\be}
\def\gac{\ga}
\def\ded{\de}

\def\laone{\la^{(1)}}
\def\latwo{\la^{(2)}}
\def\muone{\mu^{(1)}}
\def\mutwo{\mu^{(2)}}

\def\cC{\mathcal C}

\def\cB{\mathcal B}

\def\cA{\mathcal A}

\def\cL{\mathcal L}

\def\cP{\mathcal P}
\def\cPo{\mathcal P}

\def\cQ{\mathcal Q}

\def\cY{\mathbb{Y}}
\def\cU{\mathcal U}
\def\cV{\mathcal V}

\def\bP{\mathbf{P}}

\def\<{\langle}
\def\>{\rangle}

\def\GL{ {\text {\rm GL} } }

\def\rS{{\text {\rm S} } }

\def\0{{\mathbf 0}}

\def\nothing{\varnothing}

\def\.{\hskip.06cm}
\def\ts{\hskip.03cm}

\def\lra{\leftrightarrow}

\def\nlra{\nleftrightarrow}

\def\bz{\textbf{\textit{z}}}

\def\zt{t}
\def\zq{q}

\def\La{\Lambda}

\def\ze{{\zeta}}

\newcommand{\maj}{\mathrm{maj}}

\newcommand{\LHS}{\mathrm{LHS}}
\newcommand{\RHS}{\mathrm{RHS}}

\newcommand{\SSYT}{\operatorname{SSYT}}
\newcommand{\SYT}{\operatorname{{\rm SYT}}}

\def\aJ{\textrm{J}}
\def\aJr{\textrm{\em J}}
\def\aK{\textrm{K}}

\DeclareMathOperator{\aL}{\textnormal{L}}
\def\aLr{\aL}

\def\.{\hskip.06cm}
\def\ts{\hskip.03cm}

\def\nin{\noindent}

\newcommand{\textsu}[1]{\textup{\textsf{#1}}}

\newcommand{\ComCla}[1]{\textup{\textsu{#1}}}

\newcommand{\sharpP}{\ComCla{\#P}}
\newcommand{\SP}{\ComCla{\#P}}

\newcommand{\PH}{\ComCla{PH}}

\def\SP{\sharpP}

\def\lE{{\ze}}

\DeclareMathOperator{\Ec}{\mathcal{E}} 

\DeclareMathOperator{\one}{\mathbf{1}} 
\DeclareMathOperator{\Rb}{\mathbb{R}} 
\DeclareMathOperator{\zero}{\mathbf{0}} 

\DeclareMathOperator{\cJ}{\mathcal{J}}
\DeclareMathOperator{\supp}{\textnormal{supp}\ts}

\DeclareMathOperator{\fa}{\textnormal{a}} 
\DeclareMathOperator{\fb}{\textnormal{b}} 
\DeclareMathOperator{\fc}{\textnormal{c}} 
\DeclareMathOperator{\fd}{\textnormal{d}} 
\DeclareMathOperator{\ff}{\textnormal{f}\ts} 

\DeclareMathOperator{\fF}{\textnormal{F}} 
\DeclareMathOperator{\fg}{\textnormal{g}\ts} 
\DeclareMathOperator{\fG}{\textnormal{G}} 
\DeclareMathOperator{\fh}{\textnormal{h}} 
\DeclareMathOperator{\fH}{\textnormal{H}} 
\DeclareMathOperator{\fm}{\textnormal{m}} 
\DeclareMathOperator{\fp}{\textnormal{p}} 
\DeclareMathOperator{\fq}{\textnormal{q}} 
\DeclareMathOperator{\fs}{\textnormal{s}} 
\DeclareMathOperator{\fA}{\textnormal{A}} 
\DeclareMathOperator{\fB}{\textnormal{B}} 
\DeclareMathOperator{\fC}{\textnormal{C}} 
\DeclareMathOperator{\fD}{\textnormal{D}} 

\newcommand{\blue}[1]{{\color{blue}{#1}}}
\newcommand{\red}[1]{{\color{red}{#1}}}

\title
[Equality conditions for correlation inequalities]
{Equality conditions for correlation inequalities}
\date{\today}

 \author{Swee Hong Chan}
 \address[Swee Hong Chan]{Department of Mathematics, Rutgers University,  Piscataway, NJ 08854.}
 \email{\texttt{sc2518@rutgers.edu}}

 \author[\ts Igor Pak]{Igor Pak}
 \address[Igor Pak]{Department of Mathematics, UCLA,  Los Angeles, CA 90095.}
 \email{\texttt{pak@math.ucla.edu}}

\begin{document}

\begin{abstract}
We prove equality conditions for the Ahlswede--Daykin (AD) inequality
and the Fortuin--Kasteleyn--Ginibre (FKG) inequality.  We then present
a number of applications and special cases of these equality conditions.
These include Bj\"orner's and Fishburn's inequalities for linear extensions
of finite posets, the Lam--Postnikov--Pylyavskyy (LPP) and the
Okounkov inequalities for Schur positivity of products of
Schur functions.  We conclude with equality conditions for the
Ahlswede--Daykin--Schur (ADS) inequality recently introduced in \cite{CCPS26},
which is an AD type extension of the LPP inequality.
\end{abstract}
	
\maketitle

\vskip-1.5cm


\section{Introduction}\label{s:intro}

\subsection{Foreword}\label{ss:intro-fore}
Correlation inequalities appear in different guises across the sciences,
going back to the classical Chebyshev inequality in analysis.  In the
cordillera of discrete correlation inequalities, two peaks tower over
all the others, largely due to their numerous and diverse applications.
These are the {\defna{Ahlswede--Daykin {\rm (AD)} inequality}} \cite{AD78},
also known as the \defng{four functions theorem}, and the
\defna{Fortuin--Kasteleyn--Ginibre {\rm (FKG)} inequality} \cite{FKG71}.

The importance of the AD and FKG inequalities in combinatorics and probability
can hardly be overstated.  Although somewhat technical to state (see below),
they are very natural and generalize a long sequence
of earlier correlation inequalities, notably  the
\defng{Harris--Kleitman {\rm (HK)} inequality} \cite{Har60,Kle66},
\defng{Seymour's inequality}  \cite{Sey73},
\defng{Holley's inequality}  \cite{Hol74}, and
\defng{Daykin's inequality} \cite{Day77}.  All of these can all be
summarized by a simple rule:  \.
 {\defnm{increasing events are positively correlated}}.

The AD and FKG inequalities are especially heavily studied in the areas
of extremal combinatorics, see e.g.\ \cite{And87,AB08,AS16,Bol86,Eng97}, and
percolation theory, see e.g.\ \cite{BR06,Gri99,Gri06,Kes82,Wer09}, and these
are just selected monographs.
In fact, even such basic notions as \emph{critical percolation} on general infinite
graphs require correlation inequalities to be well defined.\footnote{More precisely,
the HK inequality is used to show that the critical percolation \ts
$p_c(\Ga) := \sup\big\{p : \bP_p(x\lra \infty)=0\big\}$  \ts is
independent on the vertex $x$ in a connected graph~$\Ga$, see e.g.\ \cite{BR06,Gri99}.}

In this paper, we obtain the equality conditions for the AD and FKG inequalities
(Theorems~\ref{thm:AD-eq} and~\ref{thm:FKG-eq}), resolving two half-century-old
open problems in one swoop.  This is not a new direction, of course,
but a capstone on many earlier attempts in the area, going back to the
work of Daykin--Kleitman--West \cite{DKW79}.  We postpone further background
discussion until~$\S$\ref{ss:intro-prior}.

The equality conditions we establish have a curious property: the underlying
distributive lattice   must be a direct product of two, and the functions
in the assumptions factor into functions on these sublattices.
While structural rigidity is typical of equality cases,\footnote{In a
combinatorial setting, see, for example, rigid structures that emerge
in the equality cases of the \emph{Erd\H{o}s--Ko--Rado} and
\emph{Kruskal--Katona theorems}, the
\emph{Cauchy--Davenport} and \emph{Ruzsa triangle inequalities}, cf.~$\S$\ref{ss:finrem-why-eq}.}
such a plethora of possibilities is less common and makes it exceedingly
difficult to eliminate all other cases (see~$\S$\ref{ss:intro-why-hard}).

Our own motivation for this study lies in applications to order
theory and algebraic combinatorics.  Namely, we obtain equality
conditions for 
the \defng{Fishburn inequality} 
(Theorem~\ref{thm:Fis-eq}),
on the numbers of linear extensions of finite posets,
and for the \defng{Bj\"orner inequality} 
(Theorem~\ref{thm:Bjo-eq})
for the numbers of standard Young tableaux.
These results are celebrated applications of the FKG inequality,
and their equality conditions have been long sought in the literature,
cf.~\cite{CP-surv}. 

At this point let us caution the reader.  Although the derivations
of poset inequalities from the AD and FKG inequalities are now
completely streamlined, the same cannot be said about their equality
conditions, which cannot be used as a black box.  In fact, each case
requires both detailed understanding of the derivations and the proof
of equality conditions of the AD inequality, which does the
heavy lifting (cf.~$\S$\ref{ss:finrem-app}).

Next, we obtain the equality conditions for the remarkable
\defng{Lam--Postnikov--Pylyavskyy {\rm (LPP)} inequality} 
(Theorem~\ref{thm:LPP-eq}) and the \defng{Okounkov inequality} 
(Theorem~\ref{thm:Oko-eq}), which describe the
\defnm{log-supermodularity} \ts and \defnm{log-concavity} \ts of
products of Schur functions, respectively.
These two inequalities are especially prominent in algebraic combinatorics,
in connection with total positivity, recently giving rise to the most general
\defng{Ahlswede--Daykin--Schur {\rm (ADS)} inequality},  
for which we also obtain the equality conditions
(Theorem~\ref{thm:ADS-eq}).
%
We now proceed to state the main results,
followed by the prior work and further background sections.

\smallskip


\subsection{Ahlswede--Daykin inequality}\label{ss:intro-AD}
We start with one of most general and most consequential correlation inequality
in the area.

\smallskip
\begin{thm}[{\rm \defn{AD inequality}~\cite{AD78}}{}]\label{thm:AD}
Let \ts $\cL=(\aLr,\vee, \wedge)$  \ts be a finite distributive lattice.
Let \. $\fa,\fb,\fc,\fd: \aLr \to \rr_{\geq 0}$ \. be four nonnegative functions on $\cL$, such that
\begin{equation}\label{eq:AD-cd}\tag{AD-cond}
	\fa(x) \. \fb(y)  \ \leq \ \fc(x \vee y) \. \fd(x \wedge y) \qquad \forall \ x,y \in \aLr.
\end{equation}
Then
\begin{equation}\label{eq:AD}\tag{AD}
	\sum_{x \in \aLr}\. \fa(x) \,  \cdot \, \sum_{x \in \aLr}\. \fb(x)
\ \leq \  \sum_{x \in \aLr}\. \fc(x) \, \cdot \, \sum_{x \in \aLr}\. \fd(x)\..
\end{equation}
\end{thm}

\smallskip

To simplify the statement of equality conditions of the AD inequality, we need the following:

\smallskip

\begin{definition}[{\rm \defn{cross-factoring}\ts}{}]
\label{def:intro-factor}
We say that four functions \. $\fa,\fb,\fc,\fd: \aL \to \rr_{\geq 0}$ \.
\defn{cross-factor} \ts on a distributive lattice \. $\cL=(\aL,\vee, \wedge)$, if there exists

$\circ$ \ distributive lattices \. $\cL_1=(\aL_1,\vee', \wedge')$ \. and \. $\cL_2=(\aL_2,\vee'', \wedge'')$,

$\circ$ \ functions \. $\ff_1,\fg_1:\aL_1 \to \rr_{\geq 0}$ \. and \. $\ff_2,\fg_2:\aL_2 \to \rr_{\geq 0}\ts$, \. and

$\circ$ \ positive constants \. $\ala, \beb, \gac, \ded>0$, \. such that:

\nin
\begin{equation}\label{eq:AD-abstract-factor}
\cL \, \simeq \cL_1 \times \cL_2\,, \quad \ala\.\beb \, = \, \gac\. \ded\ts, \quad \text{and}
\end{equation}
\begin{equation}\label{eq:AD-abstract-2}
	\left\{ \ \begin{split}
			\fa(x_1,x_2) \ &= \ \ala \.  {\color{blue}\ff_1}(x_1) \. {\color{blue}\ff_2}(x_2)\ts, \quad
			&&\fb(x_1,x_2) \ = \  \beb \. {\color{red}\fg_1}(x_1) \. {\color{red}\fg_2}(x_2)\ts,\\
			\fc(x_1,x_2) \ &= \  \gac \. {\color{blue}\ff_1}(x_1) \. {\color{red}\fg_2}(x_2)\ts, \quad
			&&\fd(x_1,x_2) \ = \  \ded \. {\color{red}\fg_1}(x_1) \. {\color{blue}\ff_2}(x_2)\ts,
\end{split}\right.
	\end{equation}
for all \.  $(x_1,x_2)\in \aL_1\times \aL_2\ts$.
\end{definition}

\smallskip

%
%

Heuristically, these conditions define the notion of ``relative independence'' without
resorting to probabilistic language.  Note that when
four functions cross-factor on~$\cL$, we have an equality in~\eqref{eq:AD}.
Indeed, substituting \eqref{eq:AD-abstract-2} into \eqref{eq:AD} splits
each summation into a product of two.  The resulting products of four summations
on each side are identical, and the equality follows from \.
$\ala\ts\beb =\gac\ts\ded$ \ts in \eqref{eq:AD-abstract-factor}.


We are now ready to state the main result of this paper, which says that
cross-factoring of the four functions is the only way the equality in~\eqref{eq:AD}
can be achieved.  For a function \ts $f:\aL\to \rr_{\geq 0}$, denote by
\. $\supp(f) \. := \.  \{ \, x \in \aL \, : \, f(x) > 0\}$ \.
the \defnb{support} of~$f$.  For a subset $S \subseteq \aL$, the
\defnb{lattice closure} of $S$ is the minimal sublattice of $\cL$ that contains $S$.

\smallskip

\begin{thm}[{\rm \defn{Equality conditions for the AD inequality}\ts}{}]\label{thm:AD-eq}
	Let \ts $\cL=(\aL,\vee, \wedge)$ \ts be a finite distributive lattice.
	Let \. $\fa,\fb,\fc,\fd: \aLr \to \rr_{\geq 0}$ \. be four functions satisfying \eqref{eq:AD-cd},
and such that \ts $\cL$ \ts is the lattice closure of \. $\supp \ts (\fa+\fb+\fc+\fd)$.
	Then \ts \eqref{eq:AD} \ts is an equality~$:$
\begin{equation}\label{eq:AD-eq}\tag{AD-eq}
	\sum_{x \in \aLr}\. \fa(x) \,  \cdot \, \sum_{x \in \aLr}\. \fb(x)
\ = \  \sum_{x \in \aLr}\. \fc(x) \, \cdot \, \sum_{x \in \aLr}\. \fd(x)\..
\end{equation}
	\underline{\em if and only if} \, functions \ts $\fa,\fb,\fc,\fd$ \ts cross-factor on~$\cL$.
\end{thm}

\smallskip

We prove Theorem~\ref{thm:AD-eq} in Section~\ref{sec:AD-proof}.

\smallskip
{\small
\begin{rem}\label{rem:AD-lattice}
The assumption that \ts $\cL$ \ts is the lattice closure of \. $\rS:=\supp \ts (\fa+\fb+\fc+\fd)$ \.
cannot be omitted.  Indeed, without it, one could append a new maximum element $z$ to $\cL$,
creating an irreducible lattice $\cL'$, and extend the four functions by letting \.
$\fa(z)=\fb(z)=\fc(z)=\fd(z)=0$.  The conclusion of the theorem fails since $\cL'$ is not
a direct product of two lattices.  However, this assumption does not make the theorem
less general, as the functions can always be restricted to the lattice closure of~$\rS$.
\end{rem}
}
\smallskip


\subsection{Fortuin--Kasteleyn--Ginibre inequality}\label{ss:special-FKG}
The following inequality is ubiquitous in statistical physics and is
an easy consequence of the AD inequality.

 \smallskip

 \begin{thm}[{\rm \defn{FKG inequality}~\cite{FKG71}}{}]\label{thm:FKG}
Let \ts $\cL=(\aLr,\vee, \wedge)$  \ts be a finite distributive lattice, and
let \ts $\mu:\aLr\to \rr_{\geq 0}$ \ts be a log-supermodular function on~$\cL:$
\begin{equation}\label{eq:FKG-mu}
	\mu(x) \. \mu(y)  \ \leq \ \mu(x \vee y) \. \mu(x \wedge y) \qquad \forall \ x,\ts y \ts \in \ts \aLr\ts.
\end{equation}
Then, for all increasing functions \ts $\ff,\fg:\aLr\to\rr$\,
\begin{equation}\label{eq:FKG}\tag{FKG}
  \sum_{x \in \aLr} \. \ff(x)\.\mu(x)  \,\cdot \,\sum_{x \in \aLr} \.
  \fg(x) \. \mu(x)  \ \leq \ \sum_{x \in \aLr} \.  \ff(x) \. \fg(x) \. \mu(x) \, \cdot \, \sum_{x \in \aLr} \. \mu(x)\ts.
\end{equation}
 \end{thm}

 \smallskip

Here the function \ts $\ff:\aLr\to\rr$ \ts is called
\defn{increasing} \ts if \ts $f(x) \le f(y)$ \ts
for all \ts $x\preccurlyeq y$ \ts in the partial order given by~$\cL$.
The log-supermodular function $\mu$ is usually called a \defn{measure} \ts
in a probabilistic context, so each summation in \eqref{eq:FKG} can be viewed
as a \emph{discrete integration} \ts over measure~$\mu$.
The equality condition for~\eqref{eq:FKG} are given by the following result.

\smallskip
\begin{thm}[{\rm \defn{Equality conditions for the FKG inequality}\ts}{}]\label{thm:FKG-eq}
Let \ts $\cL=(\aLr,\vee, \wedge)$  \ts be a finite distributive lattice, and
let \ts $\mu:\aLr\to \rr_{> 0}$ \ts be a strictly positive log-supermodular function,
and let \ts $\ff,\fg:\aLr\to\rr$ \ts be increasing functions.
Then \eqref{eq:FKG} is an equality \,
	\underline{\em if and only if} \,  $\cL \simeq \cL_1 \times \cL_2$ \. for some
distributive lattices \. $\cL_1=(\aLr_1,\vee', \wedge')$ \. and \. $\cL_2=(\aLr_2,\vee'', \wedge'')$,
such that
	\begin{align}\label{eq:FKG-equal-1}
		\ff({\color{blue}x_1},{\color{red}x_2}) \ = \  \ff'({\color{blue}x_1}) \quad \text{ and } \quad \fg({\color{blue}x_1},{\color{red}x_2}) \ = \  \fg'({\color{red}x_2}) \quad \text{ for all } \quad ({\color{blue}x_1},{\color{red}x_2}) \in \aLr_1 \times \aLr_2
	\end{align}
and	
\begin{align}
		&\mu(x_1,x_2) \ = \  \mu_1(x_1) \. \mu_2(x_2) \qquad \forall \ (x_1,x_2) \in \aLr\ts,
	\end{align}
for some functions \. $\ff':\aLr_1 \to \rr$, \. $\fg':\aLr_2 \to \rr$, \.
$\mu_1:\aLr_1 \to \rr_{> 0}$ \.  and \. $\mu_2:\aLr_2 \to \rr_{> 0}\ts$.
\end{thm}
\smallskip

We prove Theorem~\ref{thm:FKG-eq} in Section~\ref{s:FKG}
as an easy corollary of Theorem~\ref{thm:AD-eq}.  Note that condition \eqref{eq:FKG-equal-1}
is easier than \eqref{eq:AD-abstract-2}, since \ts $\ff$ \ts
depends only on \ts $\aL_1$, and \ts $\fg$ \ts depends only on \ts $\aL_2$.

 \smallskip
{\small
\begin{rem}\label{rem:intro-FKG}
Without the assumption that function $\mu$ is be strictly positive, it follows from \eqref{eq:FKG-mu}
log-supermodularity that \ts $\cL':=\supp(\mu)$ \ts is a distributive sublattice of~$\cL$.  Thus, having
this assumption does not to be in fact make Theorem~\ref{thm:FKG-eq} less general, as we can always
restrict the inequality to~$\cL'$.
\end{rem}}

\smallskip


\subsection{Prior work}\label{ss:intro-prior}
The literature on correlation inequalities and their application is much too diverse
to allow for a quick overview.  We refer to \cite{AB08} for a historical review and
applications in extremal combinatorics, to \cite{Gri06,Gri18,HK86} for applications
in probability and statistical physics, and to \cite{CP-surv,FS00} for applications in
order theory.  A short summary of the role the AD and FKG inequalities
was given by Bollob\'as--Riordan:

\smallskip

\begin{center}\begin{minipage}{13.25cm}%
{{\em ``The FKG inequality is the most often quoted correlation inequality in
physics: even {\rm [the HK inequality]}  tends to be called the `FKG inequality'.
Ahlswede and Daykin (1978) extended the FKG inequality to a very
general correlation inequality on lattices. What is amazing is that
such an inequality \eqref{eq:AD} could be true: its proof, although far
from trivial, is not very difficult.''} \cite[p.~45]{BR06}.}
\end{minipage}\end{center}

\smallskip

The study of equality conditions for the  Ahlswede--Daykin inequality traces back to the work
of  Daykin--Kleitman--West \cite{DKW79}, motivated by combinatorial number theory.
There it is phrased as a minimization problem
for the size of \. $A\wedge B = \{a\wedge b \.:\. a,b\in \aL\}$ \. when the sizes of
\ts $|A|$ \ts and \ts $|B|$ \ts are fixed, where \ts $(\aL,\vee,\wedge)$ \ts
is a distributive lattice given as a product of chains.  Among other results,
the authors of \cite{DKW79} proved that the minimum is achieved only on products
of sublattices.

The problem of finding equality conditions of various correlation inequalities attracted
a great deal of attention.   See, e.g., \cite{Bec90,Bri90,AH93,MR94} for special cases and
applications of \eqref{eq:AD}.  Notably, McQuillan--Richter \cite{MR94} obtained the
equality conditions for the \defn{HK inequality}:
\begin{equation}\label{eq:HK}\tag{HK}
|\cU| \cdot |\cV| \, \le \, |\cU\cap \cV| \ts \cdot \ts 2^n\ts,
\end{equation}
for every two up-closed collections \ts $\cU,\cV\subseteq 2^{[n]}$.  Here a
collection \. $\cU\subseteq 2^{[n]}$ \.  of subsets of \ts $[n]:=\{1,\ldots,n\}$ \ts
is called \defn{up-closed}, if \ts $B \in \cU$ \ts for all \ts $B \subseteq A$, where \ts $A\in \cU$.

The main result by Ahlswede--Khachatrian in \cite[Thm~1]{AK95}, extending the work
in \cite{DKW79}, is a special case of Theorem~\ref{thm:FKG-eq}, when $\cL$ \ts
is a product of chains, \ts $f,g$ \ts are indicator functions of upper order
ideals of~$\cL$, and \ts $\mu$ \ts is the uniform measure on~$\ts\aL$.
Let us also mention Brightwell's equality conditions \cite{Bri90} for the GYY
and Fishburn's inequalities (see~$\S$\ref{ss:poset-Fish}), and Aharoni--Holzman \cite{AH93}
equality conditions for the remarkable \defn{Marica--Sch\"onheim inequality} \cite{MS69}:
\begin{equation}\label{eq:MS}\tag{MS}
|\{A - B \.:\.A,B\in \cU\}| \, \ge \, |\cU| \quad \text{for every} \quad
\cU\subseteq 2^{[n]}.
\end{equation}

Note that \eqref{eq:MS} is a one-line corollary from \eqref{eq:HK},
see e.g.\ \cite[Cor.~6.1.4]{AS16}.  Consequently, the equality conditions for \eqref{eq:MS}
follow easily from \cite{MR94}.  However, the original proof of these equality conditions
given in \cite{AH93} is rather technical, since the authors obtained it by a
deconstruction of a very different proof of \eqref{eq:MS} given in \cite{DL76}.

Soon after \cite{MR94}, Talagrand \cite{Tal96} independently rediscovered the
equality cases of \eqref{eq:HK}.  Remarkably, he gave an explicit lower bound
on the defect of \eqref{eq:HK} that is of \defn{Bonnesen type$:$} the
bound is zero only on equality cases.  See also a simpler proof in \cite{KKM16},
an alternative correlation lower bound in \cite{KMS14}, and \cite{CP-AF,Oss79}
for a discussion of ``Bonnesen type'' stability results.

Talagrand's result was greatly generalized over the years, most recently in
\cite[Thm~38]{DNS21}, where a general bound was obtained for the defect of the FKG
inequality for the product of chains (Theorem~\ref{thm:FKG-chain}).  Unfortunately, this
version is not Bonnesen type, and thus only gives necessary conditions for
equality cases in terms of Fourier coefficients.  Note that when $\cL$ is a 
products of chains, the difficulty is that the given $\mu$ is a priori not 
decomposable, and proving that is the missing piece in the literature.

In~\cite{YK}, Yang and Klein developed a formula for the defect of the FKG inequality
on Boolean lattices, when $\ff,\fg$ are indicator functions of upper order ideals of~$\cL$.
Unfortunately, their formula is much too cumbersome to allow an explicit description
of equality conditions.  In \cite{Win10}, Winkler proved a strict version of the remarkable
\defng{BK inequality}, which is a {negative correlation inequality} for disjoint events,
see e.g.\ \cite[$\S$4.3]{Gri18} for some background.

The equality cases have also
been studied in connection to Markov structures using the language of \emph{independence models},
see \cite[$\S$5]{F+17} for some combinatorial examples. Finally, in the context of
continuous strictly positive multivariate distributions, Perlman--Olkin
noted that the \eqref{eq:FKG} is strict under minor assumptions \cite[Prop.~2.4]{PO80}.

\smallskip

{\small
\begin{ex}\label{ex:Chebyshev}
For comparison, consider the following
\defn{Chebyshev inequality}, often viewed as the most basic example
of a correlation inequality \cite{Gra83},
see~$\S$\ref{ss:finrem-hist} for the background.
Let
\begin{equation}\label{eq:Cheb-cd}
a_1 \. \leq \. \ldots \. \leq \. a_n \,, \quad b_1 \. \leq \. \ldots \. \leq \. b_n \,
\quad \text{ and } \quad p_1,\ldots, p_n \. > \. 0\ts.
\end{equation}
Then:\footnote{To emphasize the relevance of the Chebyshev inequality, Graham writes:
``Basically, the FKG inequality represents a way of extending \eqref{eq:Cheb}
to the situation in which the underlying index set is only partially ordered,
as opposed to the totally ordered index set of integers occurring in \eqref{eq:Cheb}''
\cite{Gra83}. }
\begin{equation}\label{eq:Cheb}
\sum_{i=1}^n \. a_i \. p_i\,\cdot \, \sum_{i=1}^n \. b_i \. p_i
\ \le \ \sum_{i=1}^n \. a_i \. b_i \. p_i\,\cdot \,\sum_{i=1}^n \. p_i\.,
\end{equation}
and the equality holds \ \underline{if and only if} \
\begin{equation}\label{eq:Cheb-eq}
a_1\. = \. \ldots \. = \. a_n \quad \text{ or } \quad b_1\. = \. \ldots \. = \. b_n\..
\end{equation}

The inequality \eqref{eq:Cheb} follows immediately from \eqref{eq:FKG}. Indeed, take
a lattice \ts $\cL=([n],\vee,\wedge)$, where \ts $i \vee j:=\max\{i,j\}$ \ts
and \ts $i\wedge j:= \min\{i,j\}$, for all \ts $i,j\in [n]$.  Substituting \. $\ff(i) \gets a_i$,
\ts $\fg(i)\gets b_i$, \ts $\mu(i)\gets p_i$ \ts into \eqref{eq:FKG} gives \eqref{eq:Cheb}.
The equality conditions \eqref{eq:Cheb-eq} follow from Theorem~\ref{thm:FKG-eq},
where the two cases correspond to two ways of writing \ts $\cL$ \ts as a direct product:
\ts $\cL = \cL \times \one$ \ts and \ts $\cL = \one \times \cL$.

Compare this with the analytic proof, see e.g.~\cite[$\S$2.17]{HLP52}.
Write the defect \. $\de$ \.
of \eqref{eq:Cheb} as follows:
\begin{equation}\label{eq:Cheb-proof}
\de \ := \ \RHS \. - \. \LHS \ = \ \sum_{1\le i<j\le n} \. (a_j-a_i)(b_j-b_i) \. p_i \. p_j \ \ge \, 0.
\end{equation}
For the equality conditions, note that we must have \ts $(a_n-a_1)(b_n-b_1)=0$.
This and \eqref{eq:Cheb-cd} imply \eqref{eq:Cheb-eq}.
\end{ex}}


\smallskip

\subsection{Why these equality conditions are hard to prove}\label{ss:intro-why-hard}
In the introduction to their paper, Ahlswede and Khachatrian lamented:
``\emph{However, for \eqref{eq:AD} it seems to be difficult to classify
the cases of equality}'' \cite{AK95}.
This is in sharp contrast with the \eqref{eq:HK} and
Chebyshev inequality \eqref{eq:Cheb},
where the equality cases are straightforward.  In fact, there is
a fundamental difference between equality conditions for the AD
and FKG inequalities, versus those for other (simpler) inequalities.

Continuing with the Chebyshev inequality example, note that the
equality conditions \eqref{eq:Cheb-eq} take the form of
equalities for a subset of assumptions \eqref{eq:Cheb-cd}.
This is a natural byproduct of the proof of the
inequality, where the steps giving the proof can be examined to obtain
the equality conditions and allow for two possibilities at the end.

By contrast, the equality conditions for the AD inequality
in Theorem~\ref{thm:AD-eq} clearly do not follow this pattern,
as cross-factoring (Definition~\ref{def:intro-factor})
allows for numerous choices which do not naturally
arise in the standard proofs of \eqref{eq:AD}.
This is why a naive approach is likely to fail in this case.\footnote{For the
HK inequality, a different phenomenon happens:
the underlying measure function~$\mu$ is \emph{degenerate},
in a sense that it factorizes naturally, greatly simplifying the proof.}
Instead, our approach  gives a new elaborate proof of Theorem~\ref{thm:AD-eq},
which involves introductions of various structures and operations specifically
designed to obtain cross-factoring while preserving the assumptions \eqref{eq:AD-cd}.
This approach occupies much of the paper (Sections~\ref{sec:prelim}--\ref{sec:AD-proof}).

In other words, one way to think of the problem of finding equality conditions
of the AD and FKG inequalities, is to restrict the proofs to have only certain
tools which allow for a ``deconstruction''.  There is a formal complexity theoretic
framework for this given by Ikenmeyer and second author \cite{IP22}, where only
\emph{oblivious} \ts operations are allowed to describe the defect, i.e.\ operations
which are not allowed to make choices depending on the values of the functions.
In this restricted setting, negative results hold for the AD inequality
for $n\ge 2$ \cite[Prop.~2.5.1]{IP22}, and for the FKG inequality for $n\ge 3$
\cite[Ch.~6]{Gla-thesis}.

Back to the informal discussion, if there is a way to give a one-line explanation
for the difficulty of characterizing equality cases of certain analytic inequalities,
it is \defnm{the issue of zeros}.  Indeed, in many proofs, including standard proofs of
the AD and FKG inequalities (see e.g.~\cite[Ch.~6]{AS16}), when one needs to divide
by a variable, two cases are considered: when the variable is zero or not.  Combined
with the induction, these choices compound into a large unwieldy family of potential
equality cases.

Alternatively, various proofs of inequalities avoid the issue of zeros altogether by
having only strictly positive underlying variables, and using some kind of monotonicity
or coupling type arguments to prove the inequalities.  Then, applying elementary \defnm{limit
arguments}, the inequalities are extended to all nonnegative variables.  See, e.g.,
elegant proofs of the FKG type inequalities by Holley \cite{Hol74} and Bakry--Michel
\cite{BM92}.

When it comes to equality conditions,
the issues with these proofs is the limit argument itself,
as additional equality cases can, and often do, emerge in the limit.
For example, clearly, to prove the Chebyshev inequality \eqref{eq:Cheb},
it suffices to consider only the case of strict inequalities in \eqref{eq:Cheb-cd}.
But then there are no equality cases, as both cases in \eqref{eq:Cheb-eq} emerge only in the limit.

In summary, the AD and the FKG equality conditions are hard to prove because the
existing proofs do not easily extend to give equality conditions.  In contrast,
the equality cases of these inequalities for strictly positive functions
are really easy to describe: there are none.
However, the nonstrict version is essential for applications to many combinatorial
inequalities in algebraic combinatorics and order theory (see below).

\smallskip

\subsection{Other equality conditions are even harder}\label{ss:intro-why-other}
The difficulty of finding equality conditions go beyond the universe of
correlation inequalities, and extend to other areas.  A typical example
briefly mentioned in the Foreword ($\S$\ref{ss:intro-fore}) is the
\defng{Cauchy--Davenport inequality} \ts in additive combinatorics,
a classical result with a straightforward proof by induction.  There,
the equality conditions are given by a separate elaborate argument known
as \defng{Vosper's theorem}, see e.g.\ \cite[$\S$5.1]{TV06}.

Characterizing the equality cases of \defnm{geometric inequalities}
has proved to be a fascinating area of research in its own right.
Notably, finding equality conditions for the
\defna{Alexandrov--Fenchel {\rm (AF)} inequality} \ts remains a major
open problem even when convex bodies are full-dimensional, see e.g.\
\cite{HW20,Sch94,Sch14}, in part because such cases are ``numerous''
\cite[$\S$20.5]{BZ-book}.  While the area is over a century old,
there has been a flurry of activity in recent years, motivated in
part by Schneider's conjectural characterization of the equality cases
\cite{Sch85,Sch94}.

In a major development, Shenfeld and van~Handel obtained long sought equality
conditions for the \defng{Minkowski quadratic inequality} \cite{SvH-duke},
which is a special case of the AF inequality for three convex bodies,
and then for the AF inequality in the case of convex polytopes \cite{SvH-acta}.
Both advances required technology going beyond the classical mixed volume theory.

Briefly, when it comes to equality conditions,
the issue with Alexandrov's original proof of the AF~inequality
for polytopes \cite{Ale37}, is the limit argument at the end.
Polytopes are called \defn{strongly isomorphic} \ts if for every 
direction $u$, the exposed faces of the polytopes with normal 
direction $u$ have the same dimension.  In particular, they must 
have the same set of normals to facets.  This technical nondegeneracy 
condition turned out to be extremely helpful in the inductive argument.  
Additionally,
by \defng{Minkowski's existence theorem}, see e.g.\ \cite{BZ-book,Sch14},
one can obtain any family of convex polytopes as
the limit of strongly isomorphic polytopes.  Meanwhile, for strongly
isomorphic polytopes, the equality conditions of the AF inequality are
straightforward: the polytopes have to be homothetic, while the numerous
equality cases arise in the limit.\footnote{Going from polytopes to general
convex bodies also requires a limit argument, further complicating the problem.}

The limit argument similarly appears in other polytopal
proofs of the AF inequality, see e.g.\ \cite{BL26,CP-intro,SvH-pams}.\footnote{We are
omitting many analytic proofs here, since that technology was built on the limits, obviously.}
The same holds for algebro geometric proofs of many matroid inequalities by Huh and his
coauthors, see e.g.~\cite{Huh18} and references therein, even though the limits are
buried deep inside the algebraic geometry arguments.  The limits are more transparent in
our linear algebraic \defng{combinatorial atlas} \ts approach \cite{CP-intro,CP-atlas},
since we use induction on matrices with exactly one positive
eigenvalue, while in the limit an additional zero eigenvalue emerges.

Much of our recent work is aimed to understand and explain this phenomenon.
In the positive direction, by modifying the proof, we are able to strengthen
the inequalities while still obtaining the desired equality conditions.
Notably, we use the heavy but elementary combinatorial atlas technology
to obtain equality conditions for many matroid inequalities  \cite{CP-atlas}.
This work has been instrumental in our efforts, and motivated the approach
we took in this paper.

In the negative direction, we prove that explicit characterization of
equality cases is computationally intractable in several notable cases.
Formally, we show that characterizing equality cases of the following
inequalities is not in the polynomial hierarchy \ts $\PH$, i.e.\
\emph{hard} \ts from computational complexity point of view:

\quad $\circ$ \, \defng{AF for polytopes}, even when only four polytopes are allowed to be
distinct \cite{CP-AF},\footnote{This does not contradict Shenfeld and van~Handel's result,
since their characterization is {geometric} rather than combinatorial,
see a discussion in \cite{CP-AF,SvH-icbs}.}

\quad $\circ$ \, \defng{Stanley's inequality} \ts for the numbers of linear extensions \cite{CP-AF},

\quad $\circ$ \, \defng{Stanley--Yan inequality} \ts for numbers of certain bases in a matroid \cite{CP-SY}, and

\quad $\circ$ \, quadratic inequality for the permanent \cite{CP-perm}.

\smallskip


We view the equality conditions for the AD and FKG inequalities in a similar
spirit: it is a minor miracle that they can be described at all, albeit with
a major effort.  It would be interesting to see if our results can be
further extended  (see~$\S$\ref{ss:finrem-many}).
Fortunately, the equality conditions for the AD and FKG inequalities
are quite enough for our applications, leading to further minor miracles.


\smallskip





\subsection{Paper structure}\label{ss:intro-structure}
We start with standard definition and notation in algebraic combinatorics
and order theory in Section~\ref{s:def}, including a brief discussion of
Schur functions and distributive lattices.
In Section~\ref{s:Schur}, we state applications of the equality conditions
for \eqref{eq:AD} to Schur positive inequalities, followed by Section~\ref{s:poset}
with applications to inequalities in order theory.  Of these, the equality conditions
for the LPP inequality plays a central role and connects both sections.
After that, the paper is divided into two major parts.

The first part, consisting of Section~\ref{sec:prelim}--\ref{s:FKG},
focuses on the main results.  We first prove the equality of the AD
inequality for Boolean lattices (Theorem~\ref{thm:AD-boolean}),
followed by the general case (Theorem~\ref{thm:AD-eq}),
and the special case of product of chains (Theorem~\ref{thm:AD-chains}).
The proof is based on a number of technical results, including
\defng{consistency lemma} (Lemma~\ref{lem:cst}), the
\defng{identification lemma} (Lemma~\ref{lem:iden}), and
the \defng{support product lemma} (Lemma~\ref{lem:structure-partition}).
The latter is a lattice theoretic result allowing the reduction
to Boolean lattices.  We conclude the proof of
Theorem~\ref{thm:AD-eq} in Section~\ref{sec:AD-proof}, and
the proof of
Theorem~\ref{thm:FKG-eq} in Section~\ref{s:FKG}.

The second part, consisting of Sections \ref{sec:LPP}--\ref{s:ADS},
gives proofs of all applications of the AD and FKG equality conditions.
Of these, the LPP equality conditions are both the most technical and
the most useful for applications, so we start with that (Section~\ref{sec:LPP}).
The equality conditions of the ADS inequality give the most general result,
and come at the end (Section~\ref{s:ADS}), after all necessary ingredients and
technical tools introduced in earlier sections.
We conclude with final remarks and open problems in Section~\ref{sec:finrem}.

\medskip


\section{Definitions and notation}\label{s:def}

In this section we present standard definitions and notation in
algebraic combinatorics and order theory.  For further background in
algebraic combinatorics, see \cite{Mac,Sag01} and \cite[Ch.~7]{Sta-EC}.
For further background in order theory, see \cite{Gra98},
\cite[Ch.~3]{Sta-EC},
and \cite[Ch.~12]{West21}.

\subsection{Basic notation}\label{ss:def-basic}
We use \ts $\nn=\{0,1,2,\ldots\}$ \ts and \ts $\nn_{\ge 1} = \{1,2,\ldots\}$
\ts to denote the sets of nonnegative and positive integers.  Let
\ts $[n]=\{1,2,\ldots,n\}$ \ts and \ts $\rr_+ = \{x\ge 0\}$.
We use \ts $\zero$ \ts and \ts $\one$ \ts to denote the zero
and all-one vector.  When the dimension is relevant, we write
\ts $0^d=(0,\ldots,0)$ \ts and \ts $1^d=(1,\ldots,1)\in \rr^d$.
For a vector \. $\al = (\al_1, \dots, \al_d) \in \mathbb{R}^{d}$, denote
\[
\lfloor \al \rfloor \. := \. (\lfloor \al_1 \rfloor, \dots, \lfloor \al_d \rfloor)
\quad \text{and} \quad \lceil \al \rceil \. := \. (\lceil \al_1 \rceil, \dots, \lceil \al_d \rceil).
\]

We use \ts $2^X$ \ts to denote the set of subsets of~$X$.
The indicator function of a subset \ts $A$ \ts is denoted~$\mathbf{1}_A\ts$.
To simplify the notation, for an element \ts $a\in X$,  we use
\ts $X-a$ \ts to denote the subset \ts $X\sm \{a\}$.
Similarly,
we write \ts $X+b$ \ts to denote \ts $X\cup \{b\}$.
Next, for a subset \ts $Y\subseteq X$, we write \ts $X-Y$ \ts in place
of more standard \ts $X \sm Y$.
This way, both \ts $X-a+b$ \ts and \ts $X-Y-Z$ have a clear meaning.

For an inequality \ts $a \ge b$, the difference \ts $(a-b)$ \ts
is called the \defn{defect}.  For polynomials \ts $f,g \in \rr[x_1,x_2,\ldots]$,
we write \ts $f\geqslant g$ \ts if the defect \ts $(f-g)\ge 0$ \ts for all \ts
$x_1,x_2,\ldots \ge 0$.  We write \ts $f\geqslant_{\fm} g$ \ts if \ts
$(f-g)\in \rr_+[x_1,x_2,\ldots]$.  These properties are called the \defn{evaluation}
and the \defn{monomial positivity}, respectively.

The \defn{equality conditions} \ts of the inequality is an explicit description
of cases when the defect is zero.  For the remainder of the paper, we use
``AD equality" as shorthand for the ``equality conditions for the AD inequality'',
and the same for other inequalities.

We employ logical formulas to describe equality conditions.  For conditions $\cA$ and $\cB$, we use
{\small
$$
\left[ \. \aligned &\cA \\
& \cB \endaligned\right. \qquad \text{and} \qquad
\left\{ \. \aligned &\cA \\
& \cB \endaligned\right.
$$
}
to denote the \defn{disjunction} \ts $\cA\vee \cB$ \.
and \defn{conjunction} \ts $\cA\wedge \cB$, respectively.

An undirected graph \ts $G=(V,E)$ \ts is called \defn{simple}, if it has
no loops or multiple edges.  For vertices \ts $a,b\in V$, we write \ts
$a\lra b$ \ts if they are connected in~$G$, and \ts $a\nleftrightarrow b$ \ts
if they are not.  More generally, for subsets \ts $A,B\subset V$, we write
$A\lra B$ \ts if \ts $a\lra b$ \ts for some \ts $a\in A$, \ts $b\in B$, i.e.\
if subsets $A$ and~$B$  are connected in~$G$.  We write \ts $A\nleftrightarrow B$ \ts
otherwise.

 \smallskip

\subsection{Partitions}\label{ss:def-part}
%
An integer sequence  \ts $\la=(\la_1,\la_2,\ldots)$ \ts is a \defn{partition},
if \. $\la_1 \ge \la_2 \ge \ldots \ge \la_\ell>\la_{\ell+1} = 0$.  The \defn{size} \ts of~$\la$ \ts is the sum \ts
$|\la|:=\la_1+\ldots + \la_\ell$.  We write  \ts $\la \vdash n$ \ts for partitions with $|\la|=n$.
The \defnb{length} \ts  $\ell(\lambda)$ \ts is the number of nonzero parts of the partition.
We append a partition with zeros whenever convenient and refer to it as the same partition.
We use \ts $\cY$ \ts to denote the set of all partitions.

A \defn{Young diagram} \ts corresponding to partition~$\la$ \ts
 is the set of squares \. $\big\{(i,j)\in \nn^2 \. : \. 1\le j \leq  \la_i,
\ts 1\le i \le \ell\big\}$.  In a mild abuse of notation, we use \ts $\la$ \ts to also
denote the corresponding Young diagram, and refer to both as the \defn{straight shape}.
A \defn{conjugate partition} \ts $\la'=(\la_1',\la_2',\ldots)$ \ts is defined by
\ts $\la_j' = |\{i\.: \. \la_i\ge j\}|$.  The corresponding Young diagram is a
reflection of \ts $\la$ \ts across the \ts $x=y$ \ts line.

Let \ts $\mu=(\mu_1,\mu_2,\ldots)$ \ts be a partition such that \.
$\mu_i \le \la_i$ \. for all \ts $i\ge 1$.  We write \ts $\mu \subseteq \la$ \ts
in this case.  The difference of Young diagrams
is denoted by \. $\la/\mu$ \. and called the \defn{skew Young diagram}.
We use \defn{skew shape} \ts to refer to both pair of partitions as above,
and to the corresponding shew Young diagram. We use \ts $|\la/\mu|:=|\la|-|\mu|$ \ts
to denote the number of squares in~$\la/\mu$.

For partitions \ts $\al=(\al_1,\al_2,\ldots)$ \ts and \ts $\be=(\be_1,\be_2,\ldots)$,
denote by \. $\al+\be=(\al_1+\be_1,\al_2+\be_2,\ldots)$ \. and  \.
$\al\cup \be:=(\al'+\be')'$ \. the \defn{sum} \ts and \defn{union} \ts of parts,
respectively.  Denote by \ts $\al\vee \be$ \ts and \ts $\al\wedge \be$ \ts
the shapes given by the union and the intersection of the corresponding Young diagrams, respectively:
\begin{equation}\label{eq:LPP-union-int}\quad
\al\vee \be \. := \. \big(\max\{\al_1,\be_1\}, \ts \max\{\al_2,\be_2\}, \ts \ldots\big), \quad
\al\wedge \be \. := \. \big(\min\{\al_1,\be_1\}, \ts \min\{\al_2,\be_2\}, \ts \ldots\big).
\end{equation}
Note that \ts ${\mathcal{Y}}=(\cY,\vee,\wedge)$ \ts is a distributive lattice with
respect to these \defn{meet} \ts and \defn{join} \ts operations.
Finally, we write \ts $\la \vartriangleright \mu$ \ts for the \defn{dominance order} \ts
on partitions: \ts $\la_1 \ge \mu_1\ts$, \ts $\la_1+\la_2 \ge \mu_1 + \mu_2\ts$,
\ldots , \ts $|\la|=|\mu|$.

\smallskip

\subsection{Young tableaux} \label{ss:def-YT}
A \defn{standard Young tableau} \ts of shape $\la/\mu$ \ts
is a bijection \ts $T: \la/\mu\to [n]$ \ts which increases in rows and columns: \.
$T(i,j) < T(i+1,j)$ \ts and \ts $T(i,j) <  T(i,j+1)$ \ts whenever these are defined.
Denote by \ts $\SYT(\la/\mu)$ \ts the set of standard Young tableaux of shape~$\la/\mu$.
Famously, a \defng{hook-length formula} \ts gives a product formula for \ts $|\SYT(\la)|$.

Let \ts $A: \la/\mu \to \nn_{\ge 1}$ \ts be a function which weakly increases in rows and strictly
increases in columns.  We think of \ts $A$ \ts as a Young tableau with integers written
in squares of~$\la/\mu$.  Such~$A$ is called a \defn{semistandard Young tableau}.
The set of such tableaux is denoted \ts $\SSYT(\la/\mu)$.  We use \ts $\SSYT(\la/\mu,t)$ \ts
to denote semistandard Young tableaux with entries $\le t$.

The \defn{weight} \ts of a tableau \ts $A\in \SSYT(\la/\mu)$ \ts is a sequence \.
$\big(m_1(A),m_2(A),\ldots\big)$, where \ts $m_i(A):=|A^{-1}(i)|$ \ts is the number
of $i$'s in~$A$.  \defn{Kostka number} \ts $K_{\la,\mu}$ \ts is the number of
$A\in \SSYT(\la)$ \ts of weight~$\mu$.

\smallskip

\subsection{Symmetric polynomials and functions}\label{ss:def-sym}
For a partition $\la$,  a
\defn{monomial symmetric polynomial} \ts is given by
$$
\fm_\la(z_1,\ldots,z_n) \, := \, \sum_{\si} \. z^{\la_1}_{\si(1)}\. z^{\la_2}_{\si(2)} \. \cdots \,,
$$
where the summation is over a set of permutations \ts $\si\in S_n$ \ts representing the distinct rearrangements of $\lambda$,   to ensure each unique monomial terms appear
exactly once.  The
\defn{skew Schur polynomial} \ts is a symmetric polynomial associated with the skew shape \ts $\la/\mu$ \ts
and can be defined as
\begin{equation}\label{eq:Schur-def}
\fs_{\la/\mu}(z_1,\ldots,z_n) \, := \, \sum_{A \ts \in \ts \SSYT(\la/\mu, \ts n)} \. z_1^{m_1(A)} \ts \cdots \. z_n^{m_n(A)}\..
\end{equation}
Here we include only combinatorial definitions, rather than the (more standard)
determinantal definition.
Similarly, \defn{skew Schur functions} \ts are defined as
\begin{equation}\label{eq:Schur-def-inf}
\fs_{\la/\mu}(z_1,z_2,\ldots) \, := \, \sum_{A \ts \in \ts \SSYT(\la/\mu)} \. z_1^{m_1(A)} \. z_2^{m_2(A)} \ts \cdots
\end{equation}
They are the stable limits of Schur polynomials as \ts $n\to \infty$. Note that Schur functions
\ts $\fs_{\la/\mu}$ \ts are homogenous of degree \ts $|\la/\mu|$.

\defn{Schur functions} \ts $\fs_\la$ \ts correspond to straight shapes, i.e.\ when
\ts $\mu=\emp$.  They form a linear basis in the \defn{ring of symmetric functions} \.
$\La=\varprojlim \La_n$\ts, where \. $\La_n=\cc[x_1,\ldots,x_n]^{S_n}$.
For two symmetric functions \ts $\fp,\fq \in \La$, we write \. $\fp \leqslant_s \fq$ \.
if \. $(\fq-\fp)$ \. is \defn{Schur positive}, i.e.\ if
\. $(\fq-\fp)\in \rr_{\ge 0}\<s_\la\>$ \. is a nonnegative sum of Schur functions.

The \defn{Littlewood--Richardson coefficients} \ts are the structure constants
of multiplication of Schur functions:
$$
\fs_\mu \.\cdot\. \fs_\nu \ = \ \sum_{\la\in \cY} \. c^\la_{\mu,\nu} \. \fs_\la\ts.
$$
Taking the degrees shows that \ts $c^\la_{\mu,\nu}=0$ \ts unless \ts $|\la|=|\mu|\ts+\ts |\nu|$.

Recall the following \defn{principal evaluation identity} \ts due to Stanley:
\begin{equation}\label{eq:P-part}
	s_{\la/\mu}(1,q,q^2,\ldots) \ = \ \frac{1}{(1-q)(1-q^2)\cdots (1-q^{|\la/\mu|})} \,\. \sum_{T \ts \in \ts \SYT(\la/\mu)} \. q^{\maj(T)}\.,
\end{equation}
where \ts $\maj: \SYT(\la/\mu) \to \nn$ \ts is the \defng{major index} \ts of a tableau, see e.g.\ \cite[Thm~7.19.11]{Sta-EC}.

Finally, recall the following powerful result on equality of products
of Schur functions, which provides an important context for our equality conditions:


\begin{thm}[{\rm Rajan \cite{Raj}}{}]\label{t:Rajan}
Let \. $\la,\mu,\nu,\al,\be,\ga,\ldots \in \cY$.  Then:
$$
s_\la \cdot s_\mu \cdot s_\nu \. \cdots \ = \ s_\al \cdot s_\be \cdot s_\ga \. \cdots
$$
\underline{\em if and only if} \,  we have an  equality of multisets: \ts
$\{\la,\mu,\nu, \ldots\} \. = \. \{\al,\be,\ga, \ldots\}$.
\end{thm}

\smallskip

\subsection{Posets}\label{ss:def-posets}
Let \ts $\cPo=(X,\prec)$ \ts be a \defnb{partially ordered set}
on the ground set $X$ of size $|X|=n$, and with the partial order~``$\prec$''.
We use \ts $\cC_n$ \ts and \ts $\cA_n$ \ts to denote the chain and the antichain
on~$n$ elements.
A \defn{subposet} \ts is an induced poset \ts $\cP|_Y=(Y,\prec)$ \ts on the subset
\ts $Y\subseteq X$.

A \defn{linear extension} of \ts $\cPo$ \ts is a bijection \ts $\lE : X \to [n]$ \ts
that is order-preserving: \ts $x \prec y$ \ts implies \ts $\lE(x)<\lE(y)$, for all \ts
$x,y\in X$.
Denote by \ts $\Ec(\cPo)$ \ts the set of linear extensions of \ts $\cPo$, and let \ts
$e(\cPo):=|\Ec(\cPo)|$ \ts be the number of linear extensions.

A subset \ts $A \subseteq X$ \ts is an \defn{upper order ideal} \ts if
\ts $x \in A$ \ts and \ts $y \succ x$ \ts implies \ts $y \in A$.  Similarly,
a subset \ts $A \subseteq X$ \ts is a
\defn{lower order ideal} \ts if \. $x \in A$ \ts and \ts $y \prec x$ \ts implies \ts $y \in A$.

The \defn{comparability graph} \ts of $\cP$ is the graph \ts $\Ga(\cP)=(X,E)$,
where the edges \ts $(x,y)\in E$ \ts are pairs of comparable elements: \ts $x \prec y$ \ts
or \ts $x \succ y$.
For elements \ts $x,y \in X$, we write \. $x \. || \. y $ \. if \ts $x$ \ts is incomparable
to \ts $y$ \ts in~$\ts\cP$.

For posets \ts $\cP=(X,\prec')$ \ts and \ts $\cQ=(Y,\prec'')$,
the \defn{parallel sum} \. $\cP+\cQ=(Z,\prec)$ \. is the poset on the disjoint union
\ts $Z=X \sqcup Y$, where elements of $X$ retain the partial order of~$\cP$, elements of $Y$ retain
the partial order of~$\cQ$, and elements \ts $x \in X$ \ts and  \ts $y \in Y$ are incomparable.

Finally, for a skew shape \ts $\la/\mu$, denote by \ts $\cP_\la$ \ts a \defn{Young poset} \ts
on squares of the corresponding Young diagram:  \. $(i,j)\preccurlyeq (i',j')$ \. if \.
$i\le i'$ \. and \. $j\le j'$. Note that linear extensions of \ts $\cP_{\la/\mu}$ \ts are in
natural bijection with standard Young tableaux of shape~$\la/\mu$, so \ts $e(\cP_{\la/\mu})=|\SYT(\la/\mu)|$.

\smallskip

\subsection{Lattices}\label{ss:def-lattices}
A \defn{lattice} \ts $\cL=(\aL,\vee,\wedge)$ \ts is a set equipped with two commutative
and associative operations called \defn{join} and \defn{meet}, respectively,
 which satisfy \. $x\wedge x = x \vee x = x$, \ts
$x\wedge (y\vee x) = x$ \ts and \ts
$x\vee (y\wedge x) = x$, for all \ts $x,y\in \aL$.
A lattice \ts $\cL'=(\aL',\vee,\wedge)$ \ts is called a
\defn{sublattice} \ts of $\cL=(\aL,\vee,\wedge)$,
if \. $\aL'\subseteq \aL$, and meet/join operations coincide on~$\aL'$.
A lattice is \defn{distributive} \ts if we have in addition
$$
x \wedge (y\vee z) \. = \. (x \wedge y) \vee (x \wedge z) \quad \text{and}
\quad x \vee (y\wedge z) \. = \. (x \vee y) \wedge (x \vee z)\ts.
$$
Note that a lattice defines a poset structure: \ts $\cP_\cL = (\aL,\prec)$,
where \ts $x\preccurlyeq y$ \ts if and only if \ts $x\vee y = y$.

By a mild abuse of notation we use \ts $\cC_n$ \ts to denote a lattice
whose poset is an $n$-chain, as in the Example~\ref{ex:Chebyshev}.
We write \ts $\one=\cC_1$ \ts to denote a \defn{trivial lattice} \ts
with one element.
For distributive lattices \ts $\cL'=(\aL',\vee',\wedge')$ \ts and \ts
$\cL''=(\aL'',\vee'',\wedge'')$, their \defn{direct product} \ts
$\cL = \cL'\times \cL''$ \ts is a lattice on \ts $\aL'\times \aL''$ \ts
with
$$
(x',x'') \vee (y',y'') \, := \, (x'\vee' y',x''\vee'' y'') \quad  \ \
\text{and} \quad \ \
(x',x'') \wedge (y',y'') \, := \, (x'\wedge' y',x''\wedge'' y'').
$$
The direct product of distributive lattices is also distributive.
We use \. $\bigotimes_{i=1}^n \cL_i$ \. to denote a direct product of
$n$ lattices.
Lattice \ts $\cL$ \ts is called \defn{directly indecomposable} \ts if \ts
$\cL = \cL\times \one= \one \times \cL$ \ts are the only decompositions
as a direct product of two lattices.  A chain \ts $\cC_n$ \ts is
directly indecomposable.

The \defn{Boolean lattice} \ts $\cB_n=\big(2^{[n]},\cup,\cap)$ \ts
is the lattice of subsets of \ts $n$-set.  Clearly,
\ts $\cB_n=\bigotimes_{i=1}^n \cB_1 = \cC_2 \times \cdots \times \cC_2$ \ts ($n$ times).
The \defn{Young lattice} \. ${\mathcal{Y}}=(\cY,\vee,\wedge)$ \ts
defined above is the distributive lattice of partitions.
For a poset \ts $\cP=(X,\prec)$, the set of lower order ideals in~$\cP$
forms a distributive lattice. 
For example, partitions can be viewed as lower order ideals in \ts $\nn^2$.

\smallskip

\subsection{Birkhoff's theorem} \label{ss:def-Birk}
We say that finite distributive lattice \ts $\cL=(\aL,\vee,\wedge)$ \ts is
\defn{irreducible} \ts if it is not isomorphic to a
direct product of two nontrivial distributive lattices.
An element \ts $x \in \aL$ \ts is \defnb{join-irreducible} if $x$ is
neither the minimal element nor the join of any two smaller elements.
We denote by \ts $\aJ(\cL)$ \ts the set of join-irreducible elements of~$\ts\cL$.
Denote by \. $\cJ_{\!\cL}=(\{0,1\}^{\aJ(\cL)},\vee, \wedge)$ \. the Boolean lattice
on join irreducibles.

\smallskip

\begin{thm}[{\rm \defn{Birkhoff's representation theorem} \cite[Thm~17.3]{Bir}}{}]\label{thm:Birkhoff}
Every finite distributive lattice \ts $\cL=(\aL,\vee,\wedge)$ \ts
is isomorphic to a sublattice of  the Boolean lattice \. $\cJ_{\!\cL}$ \. under the lattice embedding
	\. $\phi: \aLr \to  \{0,1\}^{\aJr(\cL)}$ \.  given by
	\begin{align*}
		\phi(x) \ = \ (v_y)_{y \in \aJr(\cL)}\,, \quad \text{ where } \quad
		v_y  \, = \,
		\begin{cases}
			1 & \text{ if } \  y \preccurlyeq x \text{ \ in \ } \cP_\cL\., \\
			0 & \text{ otherwise}.
		\end{cases}
	\end{align*}
\end{thm}

\smallskip

The theorem says that every lattice can be obtained as a lattice of
lower order ideals of a poset.  It is also called the
\defng{fundamental theorem for finite distributive lattices}, 
see e.g.\ \cite[$\S$3.4]{Sta-EC}.


\medskip

\section{Schur positive inequalities}\label{s:Schur}

In this section, we present equality conditions for the LPP and the Okounkov
inequalities for products of skew Schur functions.  We then present equality
conditions for the ADS inequality which extends the LPP inequality for straight
shapes to the AD type inequality.

\smallskip

\subsection{Lam--Postnikov--Pylyavskyy inequality}\label{ss:Schur-LPP}
The following \defng{log-supermodularity} \ts of Schur functions is both remarkable
and unexpected:

\smallskip

\begin{thm}[{\rm \defn{LPP inequality} \cite[Thm~5]{LPP07}}{}]
	\label{thm:LPP}
	Let \ts $\lambda/\mu$ \ts and \ts $\nu/\rho$ \ts be skew shapes.
	Then:
	\begin{equation}\label{eq:LPP}\tag{LPP}
		 s_{\lambda/\mu} \. \cdot \. s_{\nu/\rho}
		 \ \leqslant_s \ s_{\la/\mu \. \vee \. \nu/\rho} \. \cdot \.  s_{\la/\mu \. \wedge \. \nu/\rho}\.,
	\end{equation}
	\vskip.18cm
	\nin
	where \. $\la/\mu \. \vee \. \nu/\rho \ts := \ts (\la\vee \nu)/(\mu\vee \rho)$ \. and \.
	$\la/\mu \. \wedge \. \nu/\rho \ts := \ts (\la\wedge \nu)/(\mu\wedge \rho)$.
\end{thm}

\smallskip

An injective proof of the monomial positive version of \eqref{eq:LPP} was given
in~\cite[Thm~4.5]{LP07}, see also~$\S$\ref{ss:finrem-inj}.
%
%
In~\cite[$\S$4.1]{CP-multi}, the authors deduced \eqref{eq:LPP}
from a multivariate generalization of~\eqref{eq:AD}.
This approach turned out to be crucial for this paper.


For straight shapes, 
the inequality \eqref{eq:LPP} states:
	\begin{equation}\label{eq:LPP-straight}
		 s_{\lambda} \. \cdot \. s_{\nu}
		 \ \leqslant_s \ s_{\la \vee \nu} \. \cdot \.  s_{\la\wedge\nu}\..
	\end{equation}
The equality conditions of \eqref{eq:LPP-straight} are straightforward in this case:
\begin{equation}\label{eq:LPP-straight-eq}
\la \subseteq \nu \qquad \text{or} \qquad \nu \subseteq \la\ts.
\end{equation}
This dichotomy is an immediate consequence of Rajan's Theorem~\ref{t:Rajan}.

\smallskip
{\small
\begin{rem}\label{rem:LPP-LR}
Inequality \eqref{eq:LPP-straight} has a natural
interpretation in terms of LR coefficients: \.
$c^\tau_{\la,\nu}  \le  c^\tau_{\la \vee \nu, \ts\la \wedge \nu}\.$,
cf.~$\S$\ref{ss:finrem-more}.
In \cite{Pak-OPAC}, the second author conjectured that the defect of this
inequality is not in~$\SP$, a conjecture recently disproved in \cite{PS26}.
See also \cite[$\S$4.7]{PPY19} for implications of \eqref{eq:LPP-straight}
to the maximal LR coefficient.
\end{rem}
}

\smallskip

The following results show that the equality conditions for \eqref{eq:LPP} generalize
the dichotomy in~\eqref{eq:LPP-straight-eq}.  A skew shape is called \defn{connected}
\ts if the squares in the corresponding Young diagram are rook-connected.
\smallskip

\begin{thm}[{\rm \defn{Equality conditions for the LPP inequality}\ts}{}]\label{thm:LPP-eq}
		Let \ts ${\color{blue}\la}/{\color{blue}\mu}$ \ts and \ts ${\color{red}\nu}/{\color{red}\rho}$ \ts be skew shapes,
such that \. $(\blue{\la}\vee \red{\nu})/(\blue{\mu} \wedge \red{\rho})$ \.  is  connected.
Then \eqref{eq:LPP} is an equality \ \, \underline{\em if and only if} \   	
\begin{equation}\label{eq:LPP-eq-comb}\tag{LPP-eq-cond}
\text{either} \qquad 	
\left\{\. \aligned   {\color{blue}\la} &\subseteq {\color{red}\nu} \\
{\color{blue}\mu} & \subseteq {\color{red}\rho}
\endaligned \right.
\qquad \text{or }
\qquad
\left\{\. \aligned
{\color{red}\nu} & \subseteq  {\color{blue}\la} \\ 	
{\color{red}\rho} &\subseteq {\color{blue}\mu}
\endaligned\right.\hskip2.cm
\end{equation}
\end{thm}

\smallskip

We prove Theorem~\ref{thm:LPP-eq} in Section~\ref{sec:LPP} as
a corollary of Theorem~\ref{thm:AD-eq}.  Clearly, the equality conditions
for \eqref{eq:LPP} are the same as for monomial positive version of~\eqref{eq:LPP},
cf.~$\S$\ref{ss:finrem-inj}. Our proof is based on the technology we developed
in~\cite{CP-multi}.

\smallskip
{\small
\begin{rem}\label{rem:intro-LPP}
The connectivity assumption does not, in fact, make Theorem~\ref{thm:LPP-eq} less general,
as one can consider each connected component separately.  In that case, the inclusions in
\eqref{eq:LPP-eq-comb} can go in different directions for different connected components.
Note also that there is no known extension of Rajan's Theorem~\ref{t:Rajan}
to products of skew Schur function.  In fact, even characterizing coincidences
\. $s_{\la/\mu} = s_{\al/\be}$ \. among skew Schur functions, is a major open problem,
see \cite{BTW06,RSW09}.  Curiously, this problem can also be stated in terms of LR coefficients:
find all skew shapes \ts $\la/\mu$, $\al/\be$, such that
\begin{equation}\label{eq:LPP-LR-coinc}
c^\la_{\mu,\nu} \, = \,  c^\al_{\be,\nu}\quad
\text{for all} \quad \nu \vdash |\la|  - |\mu| = |\al|  - |\be|\ts.
\end{equation}
\end{rem}
}

\smallskip

\subsection{Okounkov's inequality}\label{ss:Schur-Oko}
The following Schur positive inequality was conjectured by Okounkov, who proved
the monomial positive version in the straight shape case \cite[p.~269]{Oko97}.
The conjecture was proved and generalized by Lam--Postnikov--Pylyavskyy \cite[Thms~5,\ts 11]{LPP07}.

\smallskip

\begin{thm}[{\rm \defn{Okounkov's inequality} \cite{Oko97,LPP07}}{}]\label{thm:Oko}
Let \ts $\lambda/\mu$ \ts and \ts $\nu/\rho$ \ts be skew shapes.  Then:
\begin{equation}\label{eq:Oko}\tag{Ok}
	  s_{\lambda/\mu} \.\cdot \. s_{\nu/\rho} \ \leqslant_{s} \ s_{\lceil \al/\be\rceil}
\. \cdot\.
s_{\lfloor \al/\be\rfloor}\.,
\end{equation}
where \. $\al:=(\lambda+\nu)/2$, \. $\be:=(\mu+\rho)/2$, \.
$\lfloor \al/\be\rfloor:=\lfloor \al\rfloor/\lfloor \be\rfloor$ \.
and \. $\lceil \al/\be\rceil := \lceil \al\rceil/\lceil \be\rceil$.
\end{thm}

\smallskip

For straight shapes \ts $\la,\nu$ \ts with even parts, and \ts $\mu=\rho=\emp$, we get
a \defng{log-concave inequality}~$:$
\begin{equation}\label{eq:Oko-straight-even}
s_{\lambda} \cdot s_{\nu} \ \leqslant_{s} \ s_{(\la+\nu)/2}^2\..
\end{equation}
By Rajan's Theorem~\ref{t:Rajan}, the equality conditions are straightforward
in this case: $\la=\nu$. However, already for the general straight shapes (i.e.,
without the even part requirement), the equality conditions are more involved.
More precisely, the inequality
\begin{equation}\label{eq:Oko-straight-gen}
s_{\lambda} \cdot s_{\nu} \ = \ s_{\lceil(\la+\nu)/2\rceil} \cdot s_{\lfloor(\la+\nu)/2\rfloor}\..
\end{equation}
has equality cases coming from the ``rounding'':
\begin{equation}\label{eq:Oko-straight-eq}
\la \. - \. \nu \ = \ \pm (\ve_1,\ve_2,\ldots), \quad \text{where} \quad \ve_1,\ve_2,\ldots \in \{0,1\}.
\end{equation}
The following result gives complete equality conditions for \eqref{eq:Oko}, and generalizes \eqref{eq:Oko-straight-eq}.

\smallskip

\begin{thm}[{\rm \defn{Equality conditions for the Okounkov inequality}\ts}{}]\label{thm:Oko-eq}
Let \ts $\lambda/\mu$ \ts and \ts $\nu/\rho$ \ts be connected skew shapes.
Then \eqref{eq:Oko} is an equality \
	\underline{\em if and only if} \   	
\begin{align}\label{eq:Oko-eq-comb}
\la \ts - \ts \nu \, = \, a\ts (\ve_1,\ve_2,\ldots) \. + \. (b,b,\ldots), \quad
\mu \ts - \ts \rho \, = \, a\ts (\ve'_1,\ve'_2,\ldots) \. + \. (b,b,\ldots),
%
\end{align}
for some \. $a \in \{\pm 1\}$, \. $b \in \zz$, \. and \.
$\ve_1,\ve_1',\ve_2,\ve_2',\ldots \in \{0,1\}$.
\end{thm}

\smallskip

Here the parameter \ts $b$ \ts corresponds to the shift of skew shapes which leaves
skew Schur functions invariant: \ts $s_{\la/\mu} = s_{(\la+\one)/(\mu+\one)}\ts$, where \ts $\one=(1,1,\ldots)$.
We prove Theorem~\ref{thm:Oko-eq} in Section~\ref{sec:proof-Oko-eq}, by combining a finer version of
equality conditions for the LPP inequality (Theorem~\ref{thm:LPP-eq}),
with the framework of $L$-convexity, inspired by the recent work of Speyer~\cite{Spe26}.
In fact, we prove a slightly more general result, weakening the connectivity assumption
(Theorem~\ref{thm:Oko-eq-gen}).

\smallskip



\subsection{Ahlswede--Daykin--Schur inequality}\label{ss:Schur-ADS}
Most recently, Chen, Soskin and the authors introduced
the following \defn{Ahlswede--Daykin--Schur {\rm (ADS)} inequality},
which can be viewed as a Schur positive version of the AD inequality.

\smallskip

\begin{thm}[{\rm \defna{ADS inequality}~\cite[Thm~1.6]{CCPS26}}{}]\label{thm:ADS}
	 Let \. $\fa,\fb,\fc,\fd: \cY \to \rr_{\geq 0}$ \ts be functions satisfying
\begin{equation}\label{eq:ADS-cond}\tag{{\defn{\em ADS-cond}}}
 	\fa({\lambda}) \. \fb({\mu}) \ \leq \  \fc({\lambda \vee \mu}) \. \fd({\lambda \wedge \mu})
 \quad \text{ for all \ \ \. $\lambda,\mu \in \cY$.} \qquad
 \end{equation}
\nin
\ Then we have:
	 	\begin{equation}\label{eq:ADS}\tag{\defn{\em ADS}}
 \sum_{\lambda \in \cY} \. \fa({\lambda}) \. \fs_{\lambda} \,\cdot \,
 \sum_{\lambda \in \cY} \. \fb({\lambda}) \. \fs_{\lambda}
		 \ \leqslant_{\fs} \
		 \sum_{\lambda \in \cY} \. \fc({\lambda}) \. \fs_{\lambda}
        \,\cdot\,\sum_{\lambda \in \cY} \. \fd({\lambda}) \. \fs_{\lambda}
	 \end{equation}
\end{thm}

\smallskip

Note that the LPP inequality \eqref{eq:LPP-straight} for straight shapes follows
from \eqref{eq:ADS} by taking  \ts $\fa,\fb,\fc,\fd$ \ts to be the indicator
functions of \ts $\la,\mu,\la\vee \mu,\la\wedge\mu$, respectively.



\smallskip

\begin{thm}[{\rm \defn{equality conditions for the ADS inequality}\ts}{}]\label{thm:ADS-eq}
Let \. $\fa,\fb,\fc,\fd: \cY \to \rr_{\geq 0}$ \. be  functions satisfying \eqref{eq:AD-cd}.
Then \eqref{eq:ADS} is an equality \
	\underline{if and only if} \  one of the following two conditions holds:
\begin{equation}\label{eq:ADS-eq-1}\tag{ADS-eq1}
\left[ \. \aligned \fa \. & \equiv \. 0 \\
\fb \. & \equiv \. 0 \endaligned\right.
\qquad \text{and} \qquad
\left[ \. \aligned \fc \. & \equiv \. 0 \\
\fd \. & \equiv \. 0 \endaligned\right.
\end{equation}
or
\begin{equation}\label{eq:ADS-eq-2}\tag{ADS-eq2}
\left\{ \. \aligned \fa \. & = \. \theta \cdot \fc \\
\fb \. & =  \. \theta^{-1} \cdot \fd \endaligned\right.
\qquad \text{or} \qquad
\left\{ \. \aligned \fa \. &= \. \theta \cdot \fd \\
\fb \. & = \. \theta^{-1} \cdot\fc \endaligned\right. \ ,
\end{equation}
for some constant \. $\theta \ne 0$.
\end{thm}

\smallskip

This resolves an open problem in \cite[$\S$12.2]{CCPS26}.
In the case of \eqref{eq:LPP-straight} as above, the equality conditions \eqref{eq:LPP-straight-eq}
follow from \eqref{eq:ADS-eq-2}, with \ts $\theta=1$ \ts being the only possible choice.

\smallskip

{\small
\begin{rem}\label{rem:ADS-LPP}
In \cite[Thm~1.7]{CCPS26}, we give a generalization of Theorem~\ref{thm:ADS}
to skew shapes, called the \defng{skew ADS inequality}, which simultaneously
generalizes \eqref{eq:LPP}.   It would be really interesting and a
major challenge to obtain the equality condition for this inequality.
Unfortunately, the original proof of the skew ADS inequality is much too
involved for the approach in the proof of Theorem~\ref{thm:ADS-eq} to apply.
\end{rem}}

{\small
\begin{rem}\label{rem:ADS-Groth}
In \cite[Thms~1.8, \ts 1.10]{CCPS26},
the ADS inequality was used to prove log-supermodularity for the stable
and dual stable Grothendieck polynomials.   The resulting inequalities
generalize \eqref{eq:LPP} and coincide with it for lowest/highest degree terms.
Thus, in both cases, the equality conditions coincide with equality conditions
\eqref{eq:LPP-eq-comb} for the LPP inequality.
\end{rem}}

\medskip


\section{Poset inequalities}\label{s:poset}

In this section we present three more results on equality conditions
of combinatorial inequalities, all related to each other and the
LPP inequality.

\subsection{Bj\"{o}rner's inequality}\label{ss:poset-Bjo}
For a skew shape \ts $|\la/\mu|=n$, let
\begin{equation}\label{eq:Fis-f}
	\ff(\la/\mu) \, :=  \,  \frac{|\SYT(\la/\mu)|}{n!}\..
\end{equation}

The following inequality was originally proved by Bj\"{o}rner for straight shapes
\cite[Prop.~6.1]{Bjo11} via the hook-length formula.  He also noted that it is
a special case of Fishburn's inequality \eqref{eq:Fis}, see below.
In \cite[Cor~3.3]{CP-multi}, the authors recently extended
Bj\"{o}rner's inequality to skew shapes as follows:

\smallskip

\begin{thm}[{\rm \defn{Generalized Bj\"{o}rner inequality} \cite{Bjo11,CP-multi}}{}]\label{thm:Bjo}
	Let \ts $\la/\mu$ \ts and \ts $\nu/\rho$ \ts be skew shapes.  Then:
	\begin{equation}\label{eq:Bjo}\tag{Bj\"{o}}
		\ff(\la/\mu) \. \cdot \.  \ff(\nu/\rho) \ \leq \
        \ff(\la/\mu \ts \vee \ts \nu/\rho) \. \cdot \. \ff(\la/\mu \ts \wedge \ts  \nu/\rho),
	\end{equation}
where skew shapes on the RHS are defined as in Theorem~\ref{thm:LPP}.
\end{thm}

\smallskip

We note in \cite[$\S$4.1]{CP-multi}, that \eqref{eq:Bjo} follows easily
from \eqref{eq:LPP} by an asymptotic argument, cf.~$\S$\ref{sec:proof-Bjo-eq}.
We now give equality conditions for~\eqref{eq:Bjo}.

\smallskip

\begin{thm}[{\rm \defn{Equality conditions for the generalized Bj\"orner inequality}\ts}{}]\label{thm:Bjo-eq}
		Let \ts ${\color{blue}\la}/{\color{blue}\mu}$ \ts and \ts ${\color{red}\nu}/{\color{red}\rho}$ \ts be skew shapes,
such that \. $(\blue{\la}\vee \red{\nu})/(\blue{\mu} \wedge \red{\rho})$ \.  is  connected.
Then the generalized Bj\"orner inequality~\eqref{eq:Bjo} is an equality \
	\underline{if and only if } \ the equality condition \eqref{eq:LPP-eq-comb} holds.
\end{thm}

\smallskip

In other words, the inequality \eqref{eq:Bjo} is an equality if and only if \eqref{eq:LPP}
is an equality.  Since \eqref{eq:Bjo} follows from \eqref{eq:LPP}, it may seem
surprising the equality cases of \eqref{eq:Bjo} are limited only to those of \eqref{eq:LPP}.
Thus one can view Theorem~\ref{thm:Bjo-eq} as an extension of Theorem~\ref{thm:LPP-eq}.
Note also that Theorem~\ref{thm:Bjo-eq} gives a new result even in the straight shape case
\ts $\mu=\rho=\emp$, since Rajan's Theorem~\ref{t:Rajan} is no longer applicable.

We prove Theorem \ref{thm:Bjo-eq} in Section \ref{sec:proof-Bjo-eq}.
Rather than applying the Ahlswede--Daykin equality conditions (Theorem \ref{thm:AD-chains}) directly,
we instead present a proof that relies on the Schur positivity  in the LPP inequality.

\smallskip

\subsection{SSYT variation}\label{ss:poset-SSYT}
For a skew shape \ts $\la/\mu$, let
\begin{equation}\label{eq:GL-def}
	\fg(\la/\mu,n) \, :=  \,  |\SSYT(\la/\mu,n)|\ts.
\end{equation}
By definition \eqref{eq:Schur-def}, we have \ts $\fg(\la/\mu,n) = s_{\la/\mu}(1^n)$,
is an evaluation of the Schur polynomial.  Thus, the LPP inequality gives:
\begin{equation}\label{eq:SSYT}
		\fg(\la/\mu,n) \. \cdot \.  \fg(\nu/\rho,n) \ \leq \ \fg(\la/\mu \ts \vee \ts \nu/\rho, n) \. \cdot \. \fg(\la/\mu \ts \wedge \ts  \nu/\rho,n),
\end{equation}
for all \ts $n\ge 1$. One can think of \eqref{eq:Bjo} as a result about (scaled) dimensions
of $S_n$ modules, and of \eqref{eq:SSYT} as a similar result about dimensions
of $\GL(n,\cc)$ modules.  The following result gives equality conditions for \eqref{eq:SSYT},
for sufficiently large~$n$.

\smallskip

\begin{thm}\label{thm:SSYT-eq}
		Let \ts ${\color{blue}\la}/{\color{blue}\mu}$ \ts and
\ts ${\color{red}\nu}/{\color{red}\rho}$ \ts be skew shapes, such that
\. $(\blue{\la}\vee \red{\nu})/(\blue{\mu} \wedge \red{\rho})$ \.  is  connected.
Let \. $n>\max\{\ell(\lambda),\ell(\nu)\}$.  Then \ts \eqref{eq:SSYT} \ts  is an equality \
	\underline{if and only if } \ the equality condition \eqref{eq:LPP-eq-comb} holds.
\end{thm}

\smallskip

In other words, the inequality \eqref{eq:SSYT} is an equality if and only if \eqref{eq:LPP}
is an equality.  This is also quite surprising, as one can view Theorem~\ref{thm:SSYT-eq}
as yet another extension of Theorem~\ref{thm:LPP-eq}.  In $\S$\ref{ss:LPP-setup}, we prove
an even further extension (Theorem~\ref{thm:LPP-eq-eval}), which we use to prove
Theorem~\ref{thm:LPP-eq}.  Note again that Theorem~\ref{thm:SSYT-eq} is new even
for the straight shapes.



\smallskip

\subsection{Fishburn's inequality}\label{ss:poset-Fish}
For a poset \ts $\cP=(X,\prec)$ \ts and a subset \ts $A\subseteq X$, denote by
\. $e(A)$ \. the number of linear extensions of the subposet \ts $\cP|_A=(A,\prec)$.
Let
$$
\ff(A) \. := \. \frac{e(A)}{|A|!}
$$
denote the probability that a random total order of \ts $A$ \ts is a linear extension of~$\ts\cP|_A$.
For the Young poset \ts $\cP_\la$ \ts and a subset of squares $A$ corresponding to the
skew shape \ts $\la/\mu$, this notation is consistent: \ts $\ff(A)=\ff(\la/\mu)$.

The following inequality was originally proved by Fishburn~\cite{Fish84} in the case
\ts $C=D=\varnothing$, by using the AD inequality; it is called \defn{Fishburn's inequality}
\cite[$\S$7.2]{CP-surv}.
We proved the general version in \cite[Thm~3.4]{CP-multi}.\footnote{In
\cite[Thm~3.4]{CP-multi}, we have a weaker assumption \. $A\cap C =B \cap D =\varnothing$,
but at the cost of a somewhat more cumbersome statement of the inequality. Both versions are in fact equivalent, and we find the version
as in the theorem most convenient for our purposes.}

\smallskip

\begin{thm}[{\rm \defn{Generalized Fishburn inequality}~\cite{Fish84,CP-multi}}{}]\label{thm:Fis}
	Let \ts $\cPo=(X,\prec)$ \ts  be a finite poset. Let \ts $\blue{A},\red{B} \subseteq X$ \ts be lower order  ideals,
	and let \ts $\blue{C},\red{D} \subseteq X$ \ts be upper order ideals of~$\cPo$, such that \.
	\. $\blue{A},\red{B},\blue{C},\red{D}$ \. are pairwise disjoint.
	Then:
	\begin{equation}\tag{Fis}\label{eq:Fis}
		\ff(X-\blue{A}-\blue{C}) \.\cdot \. \ff(X-\red{B}-\red{D}) \ \leq \ \ff(X-\blue{C}-\red{D}) \.\cdot\. \ff(X-\blue{A}-\red{B}).
	\end{equation}
\end{thm}

\smallskip

For partition posets and two different skew shapes, this inequality gives \ts \eqref{eq:Bjo}.
In~\cite[Lemma~10]{Bri88}, Brightwell rederived Fishburn's inequality
from the \defng{Graham--Yao--Yao {\rm (GYY)} inequality} \cite{GYY80}.
The original paper proved the GYY inequality for posets of width two;
in full generality, GYY inequality is proved by Shepp in \cite{She80},
by an application of the FKG inequality.

In \cite{Bri90}, Brightwell
obtains equality conditions for both the GYY and Fishburn's inequalities,
i.e.\ for \eqref{eq:Fis} in the case \ts $C=D=\varnothing$.  At heart,
his proof uses an explicit injective argument.
In the next theorem we characterize the equality conditions for the
generalized Fishburn  inequality~\eqref{eq:LPP}, thus extending both
Brightwell's theorem and our Theorem~\ref{thm:Bjo-eq}.

\smallskip

\begin{thm}[{\rm \defn{Equality conditions for the generalized Fishburn inequality}\ts}{}]\label{thm:Fis-eq}
	Let \ts $\cPo=(X,\prec)$ \ts  be a finite poset. Let \ts $\blue{A},\red{B} \subseteq X$ \ts be lower  order ideals,
	and let \ts $\blue{C},\red{D} \subseteq X$ \ts be upper order ideals of~$\ts\cPo$, such that \.
	$\blue{A},\red{B},\blue{C},\red{D}$ \. are disjoint.
	Then equality occurs in the generalized Fishburn's inequality~\eqref{eq:Fis}:
	\begin{equation}\tag{Fis-eq}\label{eq:Fis-eq}
		\ff(X-\blue{A}-\blue{C}) \.\cdot \. \ff(X-\red{B}-\red{D}) \ = \ \ff(X-\blue{C}-\red{D}) \.\cdot\. \ff(X-\blue{A}-\red{B}).
	\end{equation}
	\underline{\em if and only if} \ the following four connectivity conditions hold in the comparability graph \. $\Ga(\cPo)$:
\begin{equation}\label{eq:it-Fis}
\blue{A} \. \nlra \. \red{B}, \quad \blue{C}\. \nlra \. \red{D}, \quad \blue{A}\. \nlra \. \blue{C}, \quad \red{B}\. \nlra \. \red{D}.
\end{equation}
\end{thm}

\smallskip
We prove Theorem~\ref{thm:Fis-eq} in Section~\ref{sec:proof-Fis}.
We should also note that Fishburn's original proof of \eqref{eq:Fis}
used the asymptotic limit of special instances of  \eqref{eq:AD}.
Similarly, Shepp's proof of GYY, later used by Brightwell~\cite{Bri88} to prove \eqref{eq:Fis}, used the
asymptotic limit of special instances of \eqref{eq:FKG}.
Instead, we employ an alternative argument that, while still based on the AD inequality, avoids the need for taking limits.
Furthermore, while Theorem~\ref{thm:Fis-eq} can be deduced from the AD equality (Theorem~\ref{thm:AD-eq}),
we provide a distinct proof that exploits different aspects of the AD inequality.


\medskip

\section{AD equality for Boolean lattices}\label{sec:prelim}

The next four sections are dedicated to proving equality conditions for the AD inequality on
Boolean lattices.  We prove various preliminary lemmas in the rest of this section.
We then prove the two main lemmas:
the \defng{consistency lemma} in Section~\ref{sec:cst}, and the \defng{identification lemma}
in Section~\ref{sec:iden}.  We conclude the proof of Theorem~\ref{thm:AD-boolean}
in~$\S$\ref{ss:AD-boolean-proof}.

\subsection{Setup} \label{ss:prelim-setup}
Let \ts $n\ge 1$. Observe that the \defn{Boolean lattice}
\. $\cB_n=\big(\aL_n,\vee,\wedge\big)$, where \. $\aL_n :=\{0,1\}^{n}$,
can be viewed as the lattice of $n$-dimensional vectors with entries in $\{0,1\}$,
where the join and meet operations are given by entrywise
maximum and minimum, respectively.
For a subset \ts $J\subseteq [n]$, we use \ts $\{0,1\}^J$ \ts to denote
vectors in \ts $\aL_n$ \ts whose support is in coordinates~$J$.
The following result is a special case of Theorem~\ref{thm:AD-eq} for \ts
$\cL=\cB_n\ts$.

\smallskip

\begin{thm}[{\rm \defn{AD equality for Boolean lattices}\ts}{}]\label{thm:AD-boolean}
	Let  \. $\fa,\fb,\fc,\fd: \aL_n \to \rr_{\geq 0}$ \. satisfy \eqref{eq:AD-cd}.
	Then the AD equality holds:
	\begin{equation}\label{eq:AD-eq-Boolean}
	\sum_{x \in \aLr_n}\.\fa(x) \. \cdot \.  \sum_{x \in \aLr_n}\. \fb(x)
    \ = \  \sum_{x \in \aLr_n}\fc(x) \. \cdot\.  \sum_{x \in \aLr_n}\.\fd(x)
	\end{equation}
	\underline{\em if and only if} \
	there exist a subset \ts $A \subseteq [n]$, such that
	\begin{equation}\label{eq:AD-equal-boolean}
		\fa({\color{blue}x_1},{\color{blue}x_2}) \. \fb({\color{red}y_1},{\color{red}y_2}) \  = \  \fc({\color{blue}x_1},{\color{red}y_2}) \. \fd({\color{red}y_1},{\color{blue}x_2}) \qquad \forall \ ({\color{blue}x_1},{\color{blue}x_2}), ({\color{red}y_1},{\color{red}y_2}) \in \{0,1\}^{A} \times \{0,1\}^{[n]- A}.
	\end{equation}
\end{thm}

\smallskip

\subsection{The case $n=1$}\label{sec:n1}
In this subsection we prove the special case of Theorem~\ref{thm:AD-boolean} for $n=1$, and which will be crucial in proving Theorem~\ref{thm:AD-boolean} in the full generality.

\smallskip

\begin{lemma}\label{lem:AD-boolean-1}
	Let \ts $\aLr:=\{0,1\}$ \ts 
and let  \. $\fa,\fb,\fc,\fd:  \aL \to \rr_{\geq 0}$ \. satisfy \eqref{eq:AD-cd}.
Define the following conditions:
\begin{alignat}{3}
		& \fa(0) \. \fb(0) \ = \  \fc(0) \. \fd(0) \quad &&\text{ and } \quad &&\fa(1) \. \fb(1) \ = \  \fc(1) \. \fd(1),  \label{eq:A0} \tag{A0} \\
		&   \fa(0) \. \fb(1) \ = \  \fc(0) \. \fd(1) \ &&\text{ and } \ &&\fa(1) \. \fb(0) \ = \  \fc(1) \. \fd(0), \label{eq:A1} \tag{A1} \\
		&   \fa(0) \. \fb(1) \ = \  \fc(1) \. \fd(0) \ &&\text{ and } \ &&\fa(1) \. \fb(0) \ = \  \fc(0) \. \fd(1). \label{eq:A2} \tag{A2}
	\end{alignat}
	Then \eqref{eq:AD-eq-Boolean} holds \, \underline{\em if and only if} \,
	\eqref{eq:A0} holds, and either \eqref{eq:A1} or \eqref{eq:A2} holds.
\end{lemma}

\smallskip

\begin{proof}
	The  $\ts\Leftarrow\ts$ direction is straightforward.
	We now present  the proof of $\ts\Rightarrow\ts$ direction.
	We will without loss of generality assume that
	\[ \sum_{x \in \aL}\.\fa(x)\. >  \. 0\., \quad \sum_{x \in \aL}\fb(x)\. >  \. 0\. ,
\quad \sum_{x \in \aL}\.\fc(x)\. >  \. 0 \., \quad \sum_{x \in \aL}\. \fd(x) \. >  \. 0. \]
Indeed, if say \. $ \sum \. \fa(x) =0$\., then \eqref{eq:AD-eq-Boolean} implies that either
\. $ \sum \. \fc(x) =0$ \. or  \. $ \sum \. \fd(x) =0$.
In either case, this implies that every term in
\eqref{eq:A0}, \eqref{eq:A1}, \eqref{eq:A2} is equal to~$0$,
and the conclusion follows.
	
\smallskip

First, let us prove \eqref{eq:A0}.
	Suppose to the contrary that \. $\fa(0) \fb(0) <   \fc(0)\fd(0)$.
	In particular, this implies that \ts $\fc(0), \fd(0)>0$\ts.
	Decrease  the value of \ts $\fc(0)$ \ts by sufficiently small \ts $\ep>0$ \ts
    while keeping the values of the other functions unchanged.
	That is, let \ts $\fc':\{0,1\} \to \rr_{\geq 0}$ \ts be the function
    given by  \ts $\fc'(0):= \fc(0)-\ep$ \ts and   \ts $\fc'(1):= \fc(1)$\ts.
	Notice that 	\ts $\fa,\fb,\fc',\fd$   \ts continue to  satisfy \eqref{eq:AD-cd}.
	It then follows from the AD inequality that
\begin{align*}
& \sum_{x \in \aL} \. \fa(x) \. \cdot \. \sum_{x \in \aL}\. \fb(x)
\ \ \leq \ \ \sum_{x \in \aL}\fc'(x) \.\cdot\.  \sum_{x \in \aL_n}\.\fd(x)    \\
& \qquad   \le \ \ \sum_{x \in \aL_n}\. \fc(x) \.\cdot\.  \sum_{x \in \aL}\fd(x)  \. - \.
\ep \. \sum_{x \in \aL}\. \fd(x) \ \ <  \ \ \sum_{x \in \aL}\fc(x) \.\cdot\.   \sum_{x \in \aL}\. \fd(x)\ts.
	\end{align*}
	This contradicts the assumption that equality occurs in \eqref{eq:AD}, as desired.
	By the same argument, we conclude that  \. $\fa(1) \fb(1) =   \fc(1)\fd(1)$\..
	This proves \eqref{eq:A0}.
	
	Note that \eqref{eq:AD-eq-Boolean} is given by
	\begin{align*}
		\big(\fa(0)+\fa(1)\big) \big(\fb(0)+\fb(1)\big) \ = \  \big(\fc(0)+\fc(1)\big) \big(\fd(0)+\fd(1)\big).
	\end{align*}
	It then follows from \eqref{eq:A0} that
	\begin{equation}\label{eq:AD-equal-1}
		\fa(0) \fb(1)  \. +  \.  \fa(1) \fb(0) \ = \ \fc(0) \fd(1)  \. + \. \fc(1) \fd(0).
	\end{equation}
	
Now, suppose that $\fc(1) \fd(0)=0$. Then we have \.
$\fa(1) \fb(0)= \fa(0) \fb(1)= 0$ \. because both quantities
are less than \ts $\fc(1) \fd(0)$ \ts by \eqref{eq:AD-cd}.
	It then follows from \eqref{eq:AD-equal-1} that \ts $\fc(0) \fd(1)=0$.  Thus,
	\[ \fa(0) \ts \fd(1) \, = \, \fc(0) \ts \fd(1) \. = \. 0, \qquad  \fa(1) \ts \fd(0) \, = \, \fc(1) \ts \fd(0) \. = \. 0, \]
	from which \eqref{eq:A1} follows.

Now assume that \ts $\fc(1) \fd(0) >0$\ts.
It follows from \eqref{eq:A0} that
	\begin{align*}
		\fc(0)\.\fd(1) \ = \   \frac{\fa(0) \.\fb(0) \.\fa(1) \.\fb(1)}{\fc(1) \.\fd(0)}\,.
	\end{align*}
	We can then rewrite  \eqref{eq:AD-equal-1} as
	\begin{equation*}
		\fa(0) \fb(1)  \. +  \.  \fa(1) \fb(0) \ = \  \frac{\fa(0) \. \fb(0) \.\fa(1) \.\fb(1)}{\fc(1) \.\fd(0)}  \. + \. \fc(1) \. \fd(0),
	\end{equation*}
	which is equivalent to
	\begin{align*}
		\frac{-1}{\fc(1) \fd(0)} \. \big(\fc(1) \fd(0) \. -  \.
\fa(1) \fb(0) \big) \. \big(\fc(1) \fd(0) \. - \. \fa(0) \fb(1)\big) \ = \ 0.
	\end{align*}
	This implies that, either \. $ \fc(1) \fd(0) = \fa(1) \fb(0)$ \. or \. $ \fc(1) \fd(0) = \fa(0) \fb(1)$.
	In the former case, it follows from \eqref{eq:AD-equal-1} that
	\[    \fa(0) \fb(1) \ = \  \fc(0) \fd(1) \quad \text{ and } \quad \fa(1) \fb(0) \ = \  \fc(1) \fd(0).  \]
	This gives \eqref{eq:A1}.
	In the latter case, it also follows from \eqref{eq:AD-equal-1} that
	\[    \fa(0) \fb(1) \ = \  \fc(1) \fd(0) \ \text{ and } \ \fa(1) \fb(0) \ = \  \fc(0) \fd(1), \]
	This gives \eqref{eq:A2} and completes the proof of the lemma.
\end{proof}

\smallskip

\subsection{Restriction operation}
We now introduce the  \defnb{restriction} operation.
Let $f: \{0,1\}^n \to \rr$ be a function, and let  $x \in \{0,1\}$.
The function  \. $f(x,\bu^{n-1}):\{0,1\}^{n-1} \to \rr$  \. is given by
\[ (y_1,\ldots, y_{n-1}) \quad \mapsto \quad f(x,y_1,\ldots, y_{n-1}),  \]
the function on $\rr^{n-1}$ obtained from $f$ by restricting the first coordinate input to $x$.
Here ($\bu$) represents the input variable to the function, and \. $\bu^{i} := \underbrace{(\bu, \ldots, \bu)}_i$ \. represent a string of  $i$ variables.
For each $i\in [n]$, the function \. $f(\bu^{i-1}, x,\bu^{n-i})$ \. is defined analogously.
The following result follows directly from  the definition of \eqref{eq:AD-cd}.

\smallskip

\begin{lemma}\label{lem:AD-cd-res}
	Let   \. $\fa,\fb,\fc,\fd: \{0,1\}^n \to \rr_{\geq 0}$ \. be four functions satisfying \eqref{eq:AD-cd}, and let $i\in [n]$.
Then, for all  $x, y \in \{0,1\}$,  the following four functions also satisfy \eqref{eq:AD-cd}~$:$
\[
\fa(\bu^{i-1}, x,\bu^{n-i})\., \quad \fb(\bu^{i-1}, y,\bu^{n-i})\., \quad
\fc(\bu^{i-1}, x\vee y,\bu^{n-i})\., \quad \fd(\bu^{i-1}, x\wedge y,\bu^{n-i})\ts.
\]
\end{lemma}


\smallskip

It is important to note that the lemma above does not prove \eqref{eq:AD-cd} for the functions \. $\fa(1,\bu), \fb(0,\bu), \fc(0,\bu), \fd(1, \bu)$ \. and \. $\fa(0,\bu), \fb(1,\bu), \fc(0,\bu), \fd(1, \bu)$\.,
because one  cannot have \. $x \vee y =0$ \. and \. $x \wedge y =1$ \. together.
Therefore  we cannot conclude  \eqref{eq:AD} for these functions.
Throughout this paper we  will deliberately avoid relying on \eqref{eq:AD} for these specific cases.

\medskip

\subsection{Averaging operation}
We now introduce the  \defnb{averaging} operation.
Let $f: \{0,1\}^n \to \rr$ be a function.
The function  \. $f(\sq,\bu^{n-1}):\{0,1\}^{n-1} \to \rr$  \. is given by
\[ (y_1,\ldots, y_{n-1}) \quad \mapsto \quad f(0,y_1,\ldots, y_{n-1}) \. + \. f(1,y_1,\ldots, y_{n-1}),  \]
the function on $\rr^{n-1}$ obtained from $f$ by summing over all possible values of  the first coordinate.
For each $i\in [n]$, the function \. $f(\bu^{i-1}, \sq ,\bu^{n-i})$ \. is defined analogously.
In particular, we have
\[
    f(\sq^{n}) \ = \   \sum_{x \in \aL_n} \. f(x),
\]
where \. $\sq^n:= \underbrace{(\sq,\ldots, \sq)}_{n}$ \. represents the string of $n$ circles.

\smallskip

\begin{lemma}\label{lem:AD-cd-square}
	Let   \. $\fa,\fb,\fc,\fd: \{0,1\}^n \to \rr_{\geq 0}$ \. be four functions satisfying \eqref{eq:AD-cd}, and let \ts $i\in [n]$.
	Then the following functions also satisfy \eqref{eq:AD-cd}~$:$
\[
\fa(\bu^{i-1}, \sq,\bu^{n-i})\., \quad \fb(\bu^{i-1}, \sq,\bu^{n-i})\.,
\quad \fc(\bu^{i-1}, \sq,\bu^{n-i})\., \quad \fd(\bu^{i-1}, \sq,\bu^{n-i})\ts.
\]
\end{lemma}

\smallskip

\begin{proof}
	We present only the proof of  the case $i=1$, as the proof of the other cases are analogous.
	Let \ts $x,y$ \ts  be any elements of $\{0,1\}^{n-1}$.
	It follows from consecutive applications of  Lemma~\ref{lem:AD-cd-res} that
	\. $\fa(\bu,x), \fb(\bu,y), \fc(\bu, x \vee y), \fd(\bu, x \wedge y)$ \.
	satisfy \eqref{eq:AD-cd}.
	Now note that
	\begin{alignat*}{3}
	\fa(\sq,x) \ &= \    \sum_{z \in \{0,1\}} \fa(z,x), \qquad  && \fb(\sq,y) \ &&= \    \sum_{z \in \{0,1\}} \fb(z,y) \\
	\fc(\sq,x \vee y) \ &= \    \sum_{z \in \{0,1\}} \fc(z,x \vee y), \qquad && \fd(\sq,x \wedge y) \ &&= \    \sum_{z \in \{0,1\}} \fb(z,x \wedge y).
	\end{alignat*}
	Applying  \eqref{eq:AD} to these four functions gives
	\begin{align*}
		\fa(\sq,x) \. \fb(\sq,y)  \ & \leq  \  \fc(\sq,x \vee y) \. \fd(\sq,x \wedge y)\ts,
	\end{align*}
	as desired.
\end{proof}

\smallskip

\smallskip

\begin{lemma}\label{lem:diag}
	Let   \. $\fa,\fb,\fc,\fd: \{0,1\}^n \to \rr_{\geq 0}$ \. be four functions satisfying \eqref{eq:AD-cd} and \eqref{eq:AD-eq}.
	Then, for
	all \ts $x \in \{0,1\}^n$, we have:
\[
\fa(x)\ts \fb(x) \ = \  \fc(x) \ts \fd(x)\ts.
\]
\end{lemma}

\smallskip

\begin{proof}
	We prove the claim by induction on $n$.
	The case $n=1$ follows from \eqref{eq:A0} in Lemma~\ref{lem:AD-boolean-1}.
For $n\geq 2$, assume the lemma holds for $n-1$.
Let  \. $x=(x_1,\ldots, x_n) \in\{0,1\}^n\.$.
	
	Since \. $\fa, \fb, \fc, \fd$ \. satisfies \eqref{eq:AD-eq},
	it follows that  the (one-dimensional) functions \. $\fa(\bu,\sq^{n-1})$, $\fb(\bu,\sq^{n-1})$, $\fc(\bu,\sq^{n-1})$, $\fd(\bu,\sq^{n-1})$ \.
	also satisfy \eqref{eq:AD-eq}. 	Note that the same four functions also satisfy
	\eqref{eq:AD-cd} by Lemma~\ref{lem:AD-cd-square}.
	Therefore, we can apply \eqref{eq:A0} in Lemma~\ref{lem:AD-boolean-1}, to get
	\begin{align*}
		\fa(x_1,\sq^{n-1}) \. \fb(x_1,\sq^{n-1})  \ = \  \fc(x_1,\sq^{n-1}) \. \fd(x_1,\sq^{n-1})\ts.
	\end{align*}
	This is equivalent to saying that \. $\fa(x_1,\bu^{n-1})$, \. $\fb(x_1,\bu^{n-1})$, \. $\fc(x_1,\bu^{n-1})$, \. $\fd(x_1,\bu^{n-1})$ \. satisfy \eqref{eq:AD-eq}.
	Since these four functions also satisfy \eqref{eq:AD-cd} by Lemma~\ref{lem:AD-cd-res} and are $(n-1)$-dimensional,
	it then follows from the induction that
	\begin{align*}
		\fa(x_1,x_2,\ldots, x_n) \.\cdot\. \fb(x_1,x_2,\ldots, x_n)  \ = \
\fc(x_1,x_2,\ldots, x_n) \. \cdot \. \fd(x_1,x_2,\ldots, x_n),
	\end{align*}
	as desired.
\end{proof}

\medskip

\section{Consistency lemma}\label{sec:cst}

\subsection{Setup} \label{ss:cst-setup}
The main result  of this section requires the following notations.
Recall binary addition on \ts $\{0,1\}$, given by \. $0+0 = 1 + 1 = 0$ \.
 and \. $0+1 =  1$.
This operation extends to the set \ts $\{0,1, \sq\}$ \ts as follows:
\begin{align*}
 0+ \sq    \, = \, 1+ \sq \, = \,    \sq + \sq \, = \, \sq\ts,
\end{align*}
where \ts $\sq$ \ts acts similar to $\infty$.
We denote \. $1^n := (1,\ldots,1)$,
$0^n := (0,\ldots,0) \in \{0,1\}^n$, and write \. $1^k0^{n-k}$ \. to mean
\. $(1^k,0^{n-k})$ \. when the context is clear.

In the next lemma we show that Theorem~\ref{thm:AD-boolean}  follows from a list
of conditions specifically chosen so that they can be verified by induction.

\smallskip

\begin{lemma}[{\rm \defn{Consistency lemma}\ts}{}]\label{lem:cst}
	Let   \. $\fa,\fb,\fc,\fd: \{0,1\}^n \to \rr_{\geq 0}$ \. be four functions
satisfying \eqref{eq:AD-cd} and \eqref{eq:AD-eq}, and let \. $k \in \{0,1,\ldots n\}$.
	Suppose that, for all sequences \. $z \in \{0,\sq\}^n$, we have
	\begin{align}\label{eq:cst-cd}
		\fa\big(1^k0^{n-k} +z\big) \. \fb\big(0^k1^{n-k} +z\big) \ = \   \fc\big(1^n +z\big) \. \fd\big(0^n +z\big).
	\end{align}
	Then,
	for all \. $(x_1,x_2), (y_1,y_2) \in \{0,1\}^{k} \times \{0,1\}^{n-k}$, we have:
	\begin{align}\label{eq:cst}
		\fa({\color{blue}x_1},{\color{blue}x_2}) \. \fb({\color{red}y_1},{\color{red}y_2}) \ = \
\fc({\color{blue}x_1},{\color{red}y_2}) \. \fd({\color{red}y_1},{\color{blue}x_2})\ts.
	\end{align}
\end{lemma}

\smallskip

Note that the conditions in \eqref{eq:cst-cd} do not include identities such as
\. $\fa(1,\sq^{n-1}) \fb(0,\sq^{n-1})=\fc(0,\sq^{n-1}) \fd(1,\sq^{n-1})$\..
This omission occurs because the functions
 \. $\fa(1,\bu^{n-1})$,  $\fb(0,\bu^{n-1})$, $\fc(0,\bu^{n-1})$,  $\fd(1,\bu^{n-1})$ \. do not necessarily satisfy \eqref{eq:AD-cd}.
 Consequently, inductive arguments are not applicable to these functions.

\smallskip

\subsection{Proof of Lemma~\ref{lem:cst}}\label{ss:cst-proof}
	We prove the lemma by induction on $n$.
	First let  $n=1$.
	Note that the case $k=0$ corresponds to Lemma~\ref{lem:AD-boolean-1} case \eqref{eq:A0} and \eqref{eq:A2},
	and the case $k=1$ corresponds to Lemma~\ref{lem:AD-boolean-1} case \eqref{eq:A0} and \eqref{eq:A1}.
	This completes the proof of the base case.
	For the rest of the proof, let $n\geq 2$, and suppose that the lemma holds for $n-1$.
	We split the proof into several different cases.
	
	\smallskip
	
\nin
	\textbf{Case 1.} {\em
		Let \. $(x_1,x_2)$, $(y_1,y_2)\in \{0,1\}^k \times \{0,1\}^{n-k}$ \. be given by
\begin{equation}\label{eq:case-lemma}
(x_1,x_2)= (x_{1,1},\ldots, x_{1,k}, x_{2,1},\ldots, x_{2,n-k})\., \quad
(y_1,y_2)= (y_{1,1},\ldots, y_{1,k}, y_{2,1},\ldots, y_{2,n-k}).
\end{equation}
Suppose that \.  $x_{1,i}  =   y_{1,i}$ \. for some \ts $i \in [k]$, or
\.  $x_{2,j}  =   y_{2,j}$ \. for some \ts $j \in [n-k]$.
Then  \eqref{eq:cst} holds for \. $(x_1,x_2)$ \. and \. $(y_1,y_2)$.}
		
\begin{proof}[Proof of Case 1]
We present only the proof for  $i=1$  and \.  $x_{1,i}  =   y_{1,i}=0$,
as the proof of other subcases is analogous.
Note that this subcase is equivalent to showing that, for  all
\. $(x_1',x_2')$, \. $(y_1',y_2') \in \{0,1\}^{k-1} \times \{0,1\}^{n-k}$, we have:
\[  		
\fa(0,x_1',x_2') \. \fb(0,y_1',y_2') \ = \  \fc(0,x_1',y_2') \. \fd'(0,x_2',y_1').
\]
		
	Let \. $\fa',\fb', \fc',\fd':\{0,1\}^{n-1} \to \rr_{\geq 0}$ \. be functions defined by
	\begin{align}\label{eq:cst-1}
\fa':=\fa(0,\bu^{n-1}),  \quad \fb':=\fb(0,\bu^{n-1}), \quad \fc':=\fc(0,\bu^{n-1}), \quad \fd':=\fa(0,\bu^{n-1}).
	\end{align}

By Lemma~\ref{lem:AD-cd-res}, these functions satisfy \eqref{eq:AD-cd}.
Note also that \eqref{eq:AD-eq} for these four functions is equivalent to
\[
\fa(0,\sq^{n-1}) \. \fb(0,\sq^{n-1}) \ = \  \fc(0,\sq^{n-1}) \. \fd(0,\sq^{n-1}),
\]
which holds  by applying \eqref{eq:A0} in Lemma~\ref{lem:AD-boolean-1} to functions
\. $\fa(\bu,\sq^{n-1})$, $\fb(\bu,\sq^{n-1})$, $\fc(\bu,\sq^{n-1})$, $\fd(\bu,\sq^{n-1})$.
By induction, it then suffices to show that, for all \. $z' \in \{0,\sq\}^{n-1}$, we have:
\begin{align*}
	\fa'\big(1^{k-1}0^{n-k} +z'\big) \. \fb'\big(0^{k-1}1^{n-k} +z'\big) \ = \
\fc'\big(1^{n-1} +z'\big) \. \fd'\big(0^{n-1} +z\big).
\end{align*}

Given the vector $z'$, we now define four function \. $\fa'',\fb'',\fc'', \fd'':\{0,1\}\to \rr_{\geq 0}$ \. by
\begin{equation}\label{eq:cst-2}
\begin{aligned}
	\fa'' \ &:= \ \fa\big(\bu,1^{k-1}0^{n-k} +z'\big), \qquad && \fb'' \ := \ \fb\big(\bu,0^{k-1},1^{n-k} +z'\big),\\
	\fc'' \ &:= \ \fc\big(\bu, 1^{n-1} +z'\big), \qquad
	&& \fd'' \ := \ \fd\big(\bu, 0^{n-1} +z'\big).
\end{aligned}
\end{equation}
It follows from Lemmas~\ref{lem:AD-cd-res} and~\ref{lem:AD-cd-square},
that these four functions satisfy \eqref{eq:AD-cd},
and from \eqref{eq:cst-cd} that these four functions satisfy \eqref{eq:AD-eq} (namely by substituting $z=(\sq,z')$).
It then follows from  \eqref{eq:A0} in Lemma~\ref{lem:AD-boolean-1}, that
\[  \fa''(0) \ts \fb''(0) \, = \,   \fc''(0) \ts \fd''(0)\ts.   \]
This is equivalent to
\begin{align*}
	\fa'(1^{k-1}0^{n-k} +z') \. \fb'(0^{k-1}1^{n-k} +z'\big) \ = \   \fc'(1^{n-1} +z') \. \fd'(0^{n-1} +z),
\end{align*}
as desired.
\end{proof}

\smallskip

\nin
\textbf{Case 2.} {\em
Let \. $(x_1,x_2)$, $(y_1,y_2) \in \{0,1\}^n$ \. be as given  in \eqref{eq:case-lemma}.
Suppose that \.  $x_{1,i}  =  1$ \. and \. $y_{1,i}=0$ \. for some \ts $i \in [k]$,
or \.  $x_{2,j}  = 0$ \. and \. $y_{2,j}=1$ \. for some \ts $j \in [n-k]$.
Then \eqref{eq:cst} holds for \. $(x_1,x_2), (y_1,y_2)$.}


\begin{proof}[Proof of Case 2.]
		As before, we present the proof only for $i=1$.
		This case is equivalent to showing that, for  all
\. $(x_1',x_2'), (y_1',y_2') \in \{0,1\}^{k-1} \times \{0,1\}^{n-k}$,
we have:
\[  		\fa(1,x_1',x_2') \. \fb(0,y_1',y_2') \ = \  \fc(1,x_1',y_2') \. \fd'(0,x_2',y_1').
\]
		
			Let \. $\fa',\fb', \fc',\fd':\{0,1\}^{n-1} \to \rr_{\geq 0}$ \. be four functions defined by
		\begin{align}\label{eq:cst-3}
			\fa':=\fa(1,\bu^{n-1}),  \quad \fb':=\fb(0,\bu^{n-1}), \quad \fc':=\fc(1,\bu^{n-1}), \quad \fd':=\fa(0,\bu^{n-1}).
		\end{align}
		
		Note that these functions satisfy \eqref{eq:AD-cd} by Lemma~\ref{lem:AD-cd-res}.
		Also note that  \eqref{eq:AD-eq} for these four functions is equivalent to
		\[  \fa(1,\sq^{n-1}) \. \fb(0,\sq^{n-1}) \ = \  \fc(1,\sq^{n-1}) \. \fd(0,\sq^{n-1}), \]
		which holds by \eqref{eq:cst-cd} (by substituting $z=(0,\sq^{n-1})$).
		By induction, it then suffices to show that, for all \. $z' \in \{0,\sq\}^{n-1}$,	we have:
		\begin{align*}
            \fa'\big(1^{k-1}0^{n-k} +z'\big) \. \fb'\big(0^{k-1}1^{n-k} +z'\big) \ = \
            \fc'\big(1^{n-1} +z'\big) \. \fd'\big(0^{n-1} +z'\big).
		\end{align*}
		This condition is equivalent to
		\begin{align*}
			\fa\big(1,1^{k-1}0^{n-k} +z'\big) \. \fb\big(0,0^{k-1}1^{n-k} +z'\big) \ = \
            \fc\big(1,1^{n-1} +z'\big) \. \fd\big(0,0^{n-1} +z'\big),
		\end{align*}
		which follows by applying \eqref{eq:cst-cd} for $z=(0,z')$.
		This completes the proof of this case.		
\end{proof}

\smallskip

	Combining Case 1 and Case 2 above, we have \eqref{eq:cst} is verified in all cases except for
	\[
        (x_1,x_2) \ = \  (0^k,1^{n-k}), \quad  (y_1,y_2) \ = \  (1^k,0^{n-k}).
    \]
	Note that the sum of \eqref{eq:cst} over all pairs \. $(x_1,x_2), (y_1,y_2) \in \{0,1\}^n$ \.
    equals to \eqref{eq:AD-cd}.
	Thus, it follows from \eqref{eq:AD-eq} that \eqref{eq:cst} is also satisfied for the pair
\. $(0^k,1^{n-k}), (1^k,0^{n-k})$. This completes the proof of Lemma~\ref{lem:cst}. \qed

\medskip

\section{Identification lemma}\label{sec:iden}

In this section we finish the proof of Theorem~\ref{thm:AD-boolean}.  We derive it
from the Identification Lemma~\ref{lem:iden}, a technical result which we prove first.

\subsection{Setup}\label{ss:iden-setup}
The  primary challenge in proving Theorem~\ref{thm:AD-boolean} lies in identifying
which \ts $i \in [n]$ \ts belong to the set~$A$.  This motivates the following definition.

\smallskip

\begin{definition}
Let   \. $\fa,\fb,\fc,\fd: \{0,1\}^n \to \rr_{\geq 0}$ \. be four functions satisfying \eqref{eq:AD-cd}.
Fix \ts $i\in [n]$.
We say that \ts $i$ \ts has type \eqref{eq:B1}, if
\begin{equation}\label{eq:B1}\tag{B1}
	\fa(\sq^{i-1}, 1, \sq^{n-i}) \. \fb(\sq^{i-1}, 0, \sq^{n-i}) \ = \  \fc(\sq^{i-1}, 1, \sq^{n-i}) \. \fd(\sq^{i-1}, 0, \sq^{n-i}).
\end{equation}
Similarly, we say that \ts $i$ \ts has type \eqref{eq:B2}, if
\begin{equation}\label{eq:B2}\tag{B2}
	\fa(\sq^{i-1}, 0, \sq^{n-i}) \. \fb(\sq^{i-1}, 1, \sq^{n-i}) \ = \  \fc(\sq^{i-1}, 1, \sq^{n-i}) \. \fd(\sq^{i-1}, 0, \sq^{n-i}).
\end{equation}
It follows from Lemma~\ref{lem:AD-boolean-1}, that each \. $i \in [n]$ \. must have at least
one of these two types.
Note that it is possible that \ts $i$ \ts has both types \eqref{eq:B1} and \eqref{eq:B2}
\end{definition}

\smallskip

Observe that the types are preserved under the averaging operation.
Formally, if \ts $i \in [n-1]$ \ts has type \eqref{eq:B1} with respect to the functions  \. $\fa,\fb,\fc,\fd$\.,
then $i$ also has type \eqref{eq:B1} with respect to the functions \. $\fa(\bu^{n-1},\sq)$, $\fb(\bu^{n-1},\sq)$, $\fc(\bu^{n-1},\sq)$, $\fd(\bu^{n-1},\sq)$.  The same holds for type \eqref{eq:B2}.

On the other hand, the types are \emph{not necessarily preserved} \ts
 under  the restriction operation, i.e., with respect to the functions
 \. $\fa(\bu^{n-1},x)$, $\fb(\bu^{n-1},y)$, $\fc(\bu^{n-1},x\vee y)$,
 $\fd(\bu^{n-1},x \wedge y)$ \. for \emph{arbitrary} \. $x,y \in \{0,1\}$.
Indeed, let \. $\fa,\fb,\fc,\fd:\{0,1\}^2 \to \rr_{\geq 0}$ \. be given by
\begin{align}\label{eq:counterexample-1}
	\fa(x) \ = \  \fb(x) \ = \  \fc(x) \ = \  \fd(x) \ = \
	\begin{cases}
		0 & \text{ if } \ x=(0,1),\\
		1&  \text{ otherwise}.
	\end{cases}
\end{align}
It is easy to check that $i=1$ is  of type  \eqref{eq:B2} with respect to the functions \. $\fa,\fb,\fc,\fd$\.,
but  is not of type \eqref{eq:B2} with respect to  the functions \. $\fa(\bu,1),\fb(\bu,0), \fc(\bu,1),\fd(\bu,0)$\..
This motivates the next lemma, which identifies elements \. $x,y \in \{0,1\}$ \. for which the restriction
operation  preserves the types.


\smallskip

\begin{lemma}[{\rm Identification lemma}{}]\label{lem:iden}
	Let   \. $\fa,\fb,\fc,\fd: \{0,1\}^n \to \rr_{\geq 0}$ \. be  functions satisfying condition \eqref{eq:AD-cd} and \eqref{eq:AD-eq}.
	Let \. $x=(x_1,\ldots, x_n)$, \. $y=(y_1,\ldots,y_n) \in \{0,1\}^n$ \. be  given by:
	\begin{align*}
		x_i \ &:= \ \begin{cases}
			1 & \text{ if \, $i$ \, is of type \eqref{eq:B1}},\\
			0 & \text{ if \, $i$ \, is of type \eqref{eq:B2} but not \eqref{eq:B1}},
		\end{cases}\\
		y_i \ &:= \  1+ x_i\..
	\end{align*}
	Then, for all \. $z \in \{0,\sq\}^n$, we have:
		\begin{align}\label{eq:iden}
		\fa\big(x +z\big)\. \fb\big(y +z\big) \ = \   \fc\big(1^n +z\big) \. \fd\big(0^n +z\big).
	\end{align}
	Similarly, let
		 \. $x'=(x_1',\ldots, x_n')$, \. $y'=(y_1',\ldots,y_n') \in \{0,1\}^n$ \. be  given by:
		 	\begin{align*}
		 	x_i' \ &:= \ \begin{cases}
		 		1 & \text{ if \, $i$ \, has type \eqref{eq:B1} but not \eqref{eq:B2}},\\
		 		0 & \text{ if \, $i$ \, has type \eqref{eq:B2}},
		 	\end{cases}\\
		 	y_i' \ &:= \  1+ x_i'\..
		 \end{align*}
	Then, for all $z \in \{0,\sq\}^n$,  we have:
\begin{align}\label{eq:iden-prime}
	\fa\big(x' +z\big) \. \fb\big(y' +z\big) \ = \   \fc\big(1^n +z\big) \. \fd\big(0^n +z\big).
\end{align}
\end{lemma}

\smallskip

{\small
\begin{rem}
Let us note a delicate issue when interpreting Lemma~\ref{lem:iden}.
Since the pairs \. $x,y$ \. and \. $x',y'$ \. differ only at entries indexed by $i\in [n]$ has  both type \eqref{eq:B1} and \eqref{eq:B2}, it might give the wrong impression  that the lemma would still hold even if we do not impose any conditions on entries indexed by such $i$'s.
However, when multiple distinct indices $i, j \in [n]$, have both \eqref{eq:B1} and \eqref{eq:B2}, it becomes necessary for $x_i$ and $x_j$ to be equal to ensure the lemma holds, as demonstrated in the following example.

Let \. $\fa,\fb,\fc,\fd:\{0,1\}^2 \to \rr_{\geq 0}$ \. be as in \eqref{eq:counterexample-1}.
It is easy to check that these four functions satisfy \eqref{eq:AD} and \eqref{eq:AD-cd}.
Furthermore, we have $i=1$ and $j=2$ has both types \eqref{eq:B1} and \eqref{eq:B2}.
However, if \eqref{eq:iden} were to hold in this case for \ts $x=(1,0)$, $y=(0,1)$, $z=(0,0)$
(note that \ts $x_i\neq x_j$ \ts with \ts $i=1$ \ts and \ts $j=2$), we would have
\[ \fa(1,0) \. \fb(0,1) \ = \  \fc(1,1) \. \fd(0,0).  \]
This gives a contradiction, since the \. LHS $=0$ \. while the \. RHS $=1$.
\end{rem}
}
\smallskip

\subsection{Technical lemmas}\label{ss:iden-tech-lemmas}
We now build toward the proof of Lemma~\ref{lem:iden}.
We will start with the following two lemmas for $n=2$.

\smallskip

\begin{lemma}\label{lem:n2-B1B1}
	Let $n=2$, and let   \. $\fa,\fb,\fc,\fd: \{0,1\}^n \to \rr_{\geq 0}$ \.
be four functions satisfying \eqref{eq:AD-cd} and \eqref{eq:AD-eq}.
	Suppose that for both \ts $i=1$ \ts and \ts $j=2$ \ts  have type \eqref{eq:B1}.
	Then we have:
	\[ \fa(1,1) \. \fb(0,0)  \ = \  \fc(1,1) \. \fd(0,0). \]
\end{lemma}

\begin{proof}
	First note that the lemma follows immediately if the
	case \eqref{eq:A1} holds for \. $\fa(1,\cdot)$, \. $\fb(0,\cdot)$, \. $\fc(1,\cdot)$, $\fd(0,\cdot)$\..
Suppose that case \eqref{eq:A2} holds. This implies:
	\begin{align}
		\label{eq:HP-1}	\fa(1,0)\. \fb(0,1) \ &= \  \fc(1,1)\. \fd(0,0)\ts,\\
		\label{eq:HP-2}	\fa(1,1)\. \fb(0,0) \ &= \  \fc(1,0)\. \fd(0,1)\ts.
	\end{align}
	Applying the same argument to
	\. 	$\fa(\cdot,1)$, \. $\fb(\cdot,0)$, \. $\fc(\cdot,1)$, $\fd(\cdot,0)$ \. gives
	\begin{align}
		\label{eq:HP-3}	\fa(0,1) \.\fb(1,0) \ &= \  \fc(1,1) \. \fd(0,0)\ts,\\
		\label{eq:HP-4} 	\fa(1,1)\. \fb(0,0) \ &= \  \fc(0,1) \fd(1,0)\ts.
	\end{align}
	Combining \eqref{eq:HP-1} and \eqref{eq:HP-3} gives
	\begin{align*}
		\fa(1,0) \. \fb(1,0) \.\fa(0,1) \. \fb(0,1) \ = \ \big(\fc(1,1) \. \fd(0,0)\big)^2.
	\end{align*}
	Applying Lemma~\ref{lem:diag} to the LHS gives
	\begin{align*}
		\fc(1,0) \.\fd(1,0) \.\fc(0,1)\. \fd(0,1) \ = \ \big(\fc(1,1) \. \fd(0,0)\big)^2.
	\end{align*}
	Applying \eqref{eq:HP-2} and \eqref{eq:HP-4} to the LHS again gives
	\begin{align*}
		\big(\fa(1,1) \. \fb(0,0)\big)^2 \ = \ \big(\fc(1,1) \. \fd(0,0)\big)^2,
	\end{align*}
	as desired.
\end{proof}

\smallskip

\begin{lemma}\label{lem:n2-B1B2}
	Let $n=2$, and let   \. $\fa,\fb,\fc,\fd: \{0,1\}^n \to \rr_{\geq 0}$ \. be four functions satisfying \eqref{eq:AD-cd} and \eqref{eq:AD-eq}.
	Suppose that $i=1$  has type \eqref{eq:B1}, and $j=2$ has type \eqref{eq:B2}.
	Suppose also that
\begin{equation}\label{eq:lemma-B1B2}
\fa(1,0) \. \fb(0,1)  \ \neq \  \fc(1,1) \.\fd(0,0)\ts.
\end{equation}	
	Then  $i=1$  also has type \eqref{eq:B2}, and $j=2$ also has type \eqref{eq:B1}.
\end{lemma}

\begin{proof}
	We will only show that  $j=2$  satisfies \eqref{eq:B1}, as the proof of the other case is analogous.
	
	Consider the functions \. 	$\fa(1,\bu)$, \. $\fb(0,\bu)$, \. $\fc(1,\bu)$, $\fd(0,\bu)$\..
	These four functions satisfy \eqref{eq:AD-cd} by Lemma~\ref{lem:AD-cd-res}, and satisfies \eqref{eq:AD-eq} since
	$i=1$  satisfies \eqref{eq:B1}.
	Since \. $\fa(1,0) \fb(0,1)  \ \neq \  \fc(1,1) \fd(0,0)$ \. by assumption,
	it then follows from Lemma~\ref{lem:AD-boolean-1} that case \eqref{eq:A1} holds.  This implies
	\begin{align}
		\fa(1,0)\. \fb(0,0) \ &= \ \fc(1,0) \. \fd(0,0)\ts, \label{eq:march-1}\\
		\fa(1,0)\. \fb(0,1) \ &= \ \fc(1,0) \. \fd(0,1)\ts, \label{eq:march-3}\\
		\fa(1,1)\. \fb(0,0) \ &= \ \fc(1,1) \. \fd(0,0)\ts, \label{eq:march-0}\\
		\fa(1,1)\. \fb(0,1) \ &= \ \fc(1,1) \. \fd(0,1)\ts. \label{eq:march-2}
	\end{align}	
	Applying an analogous argument  to
	\. 	$\fa(\bu,0)$, \. $\fb(\bu,1)$, \. $\fc(\bu,1)$, $\fd(\bu,0)$ \.
	gives us
	\begin{align}
	\fa(0,0) \. \fb(0,1) \ &= \ \fc(0,1)\. \fd(0,0)\ts, \label{eq:march-2.5} \\
	\fa(0,0) \. \fb(1,1) \ &= \ \fc(1,1)\.\fd(0,0)\ts, \label{eq:march-0.5} \\
	\fa(1,0) \. \fb(0,1) \ &= \ \fc(0,1)\.\fd(1,0)\ts, \label{eq:march-0.1} \\
	\fa(1,0) \. \fb(1,1) \ &= \ \fc(1,1)\. \fd(1,0)\ts. \label{eq:march-1.5}
	\end{align}

	We now consider the functions \. 	$\fa(\bu,1)$, \. $\fb(\bu,1)$, \. $\fc(\bu,1)$, $\fd(\bu,1)$\..
		These four functions satisfy \eqref{eq:AD-cd} by Lemma~\ref{lem:AD-cd-res}.
		Also note that these four functions satisfy \eqref{eq:AD-eq} as a result of applying
\eqref{eq:A0} in Lemma~\ref{lem:AD-boolean-1}  to the functions  \. 	$\fa(\sq,\bu)$, \. $\fb(\sq,\bu)$, \. $\fc(\sq,\bu)$, $\fd(\sq,\bu)$\..
		Also note that \. 	$\fa(\bu,1)$, \. $\fb(\bu,1)$, \. $\fc(\bu,1)$, $\fd(\bu,1)$\. satisfies \eqref{eq:A1} because of \eqref{eq:march-2}, so it follows from Lemma~\ref{lem:AD-boolean-1} that
		\begin{align}
	\fa(0,1) \.\fb(1,1) \ &= \ \fc(0,1) \.\fd(1,1)\ts, \label{eq:march-4} \\
	\fa(1,1) \. \fb(0,1) \ &= \ \fc(1,1)\. \fd(0,1)\ts.  \label{eq:march-4.1}
		\end{align}

		We now consider the functions \. 	$\fa(\bu,0)$, \. $\fb(\bu,0)$, \. $\fc(\bu,0)$, $\fd(\bu,0)$\..
	These four functions satisfy \eqref{eq:AD-cd} by Lemma~\ref{lem:AD-cd-res}.
	Also note that these four functions satisfy \eqref{eq:AD-eq} as a result of applying Lemma~\ref{lem:AD-boolean-1} \eqref{eq:A0} to the functions  \. 	$\fa(\sq,\bu)$, \. $\fb(\sq,\bu)$, \. $\fc(\sq,\bu)$, $\fd(\sq,\bu)$\..
	Also note that \. 	$\fa(\bu,0)$, \. $\fb(\bu,0)$, \. $\fc(\bu,0)$, $\fd(\bu,0)$\. satisfies \eqref{eq:A1} because of \eqref{eq:march-1}, so it follows from Lemma~\ref{lem:AD-boolean-1} that
	\begin{align}
		\fa(0,0) \fb(1,0) \ &= \ \fc(0,0)\fd(1,0), \label{eq:march-5} \\
		\fa(1,0) \fb(0,0) \ &= \ \fc(1,0)\fd(0,0).  \notag
	\end{align}

	Note that we also have
\[ \fc(1,1)\, > \,  0 \qquad \text{and} \qquad \fd(0,0) \, > \,  0\ts.
\]
	Indeed, it follows from \eqref{eq:AD-cd} that
	\[   \fa(1,0) \fb(0,1) \ \leq \ \fc(1,1) \fd(0,0). \]
By assumption \eqref{eq:lemma-B1B2}, this implies that
	\. $\fc(1,1)\ts \fd(0,0)  >   0$.
Together with \eqref{eq:march-0} and \eqref{eq:march-0.5}, this implies that
	\[ \fa(1,1) \. \fb(0,0)  \ = \  \fa(0,0) \. \fb(1,1)  \ > \  0.\]
On the other hand, we have from Lemma~\ref{lem:diag} that
    \[ \fa(0,0) \. \fb(0,0)  \ = \  \fc(0,0) \. \fd(0,0), \quad
       \fa(1,1) \. \fb(1,1)  \ = \  \fc(1,1)\. \fd(1,1).
    \]
This gives:
\[  \fc(0,0) \, > \, 0 \qquad \text{and} \qquad \fd(1,1)  \, > \, 0. \]
Putting everything together, we conclude:
\begin{equation}\label{eq:iden-nonzero}
\fa(x,x), \. \fb(x,x), \. \fc(x,x), \. \fd(x,x) \  > \ 0 \quad \text{ for all } \ x \in \{0,1\}.
\end{equation}

	Now, combining \eqref{eq:march-0} and \eqref{eq:march-4}, we get
\[
    \fa(0,1) \. \fa(1,1) \. \fb(1,1) \. \fb(0,0) \ = \
    \fc(0,1) \. \fc(1,1) \. \fd(1,1) \. \fd(0,0).
\]
	On the other hand, we have  \. $\fa(1,1)\ts\fb(1,1)=\fc(1,1) \ts \fd(1,1)$ \.
    by  Lemma~\ref{lem:diag}, and all these terms are nonzero by \eqref{eq:iden-nonzero}.
	Hence we can cancel these terms from the previous equation to obtain
	\begin{align}\label{eq:march-6}
			  \fa(0,1)  \fb(0,0) \ = \ \fc(0,1) \fd(0,0).
	\end{align}

		Now, combining \eqref{eq:march-0} and \eqref{eq:march-5}, we get
\[  \fa(1,1) \. \fa(0,0) \. \fb(0,0) \. \fb(1,0) \ = \ \fc(1,1) \. \fc(0,0) \. \fd(0,0) \. \fd(1,0).
\]
	By the reasoning as in the previous paragraph,
	we can cancel the terms from \. $\fa(0,0)\ts\fb(0,0)=\fc(0,0) \ts\fd(0,0)$ \. to obtain
		\begin{align}\label{eq:march-7}
		\fa(1,1) \ts \fb(1,0) \ = \ \fc(1,1)\ts \fd(1,0).
	\end{align}
	
	Let us now show that
	\begin{equation}\label{eq:march-8}
		\fa(0,1) \ts \fb(1,0) \ = \ \fc(0,1) \ts \fd(1,0).
	\end{equation}
	We split the proof into three subcases.
First suppose that \. $\fa(1,0) \ts \fb(0,1)>0$.
Note that we have from Lemma~\ref{lem:diag} that
\begin{equation}\label{eq:eq-above}
\aligned
			\fa(0,1)  \fb(0,1) \ &= \ \fc(0,1) \fd(0,1), \qquad
			\fa(1,0)  \fb(1,0) \ = \ \fc(1,0) \fd(1,0).
\endaligned
\end{equation}
Since we have \. $\fa(1,0) \fb(0,1) =  \fc(1,0)\fd(0,1)$ \. from \eqref{eq:march-3},
and these products are nonzero by assumption, we can cancel them
from \eqref{eq:eq-above} to conclude \eqref{eq:march-8} in this case.

For the second case, suppose that $\fa(1,0)=0$. Since \. $\fb(1,1), \fc(1,1)>0$ \.
from \eqref{eq:iden-nonzero}, it then follows from \eqref{eq:march-1.5} that $\fd(1,0)=0$.
Similarly, since \. $\fb(0,0), \fd(0,0)>0$ \. from \eqref{eq:iden-nonzero},
it then follows from \eqref{eq:march-1} that \ts $\fc(1,0)=0$.
Finally, since \. $\fa(1,1), \fc(1,1)>0$ \. from \eqref{eq:iden-nonzero},
it then follows from  \eqref{eq:march-7} that \ts $\fb(1,0)=0$.
Putting these observations implies that both sides of \eqref{eq:march-8}
are equal to~$0$, as desired.

For the third case, suppose that \ts $\fb(0,1)=0$.  Since \. $\fa(1,1), \fc(1,1)>0$ \.
from \eqref{eq:iden-nonzero}, it then follows from \eqref{eq:march-2} that \ts $\fd(0,1)=0$.
Similarly, since \. $\fa(0,0), \fd(0,0)>0$ \. from \eqref{eq:iden-nonzero},
it then follows from \eqref{eq:march-2.5} that $\fc(0,1)=0$.
Finally, since \. $\fb(0,0), \fd(0,0)>0$ \. from \eqref{eq:iden-nonzero},
it then follows from \eqref{eq:march-6}   that \ts $\fa(0,1)=0$.
Putting these observations implies that both sides of \eqref{eq:march-8}
are equal to~$0$, as desired.  This completes the proof of \eqref{eq:march-8}.

Finally, note that \ts \eqref{eq:march-0}, \eqref{eq:march-6}, \eqref{eq:march-7},  \eqref{eq:march-8}
imply that \.  $\fa(\sq,1)\ts \fb(\sq,0)=\fc(\sq,1) \ts \fd(\sq,0)$.
This shows that $j=2$ satisfies \eqref{eq:B1}, and completes the proof.
\end{proof}

\smallskip

\subsection{Proof of Lemma~\ref{lem:iden}}\label{ss:iden-proof}
We prove only the first equation \eqref{eq:iden} in the lemma
as the proof of \eqref{eq:iden-prime} is analogous.
We use  induction on $n$.  For $n=1$ there is nothing to prove.

\smallskip

For $n=2$, we split the proof into two cases.  First, suppose that \ts $x_1=x_2$.
We will without loss of generality assume that \ts $x_1=x_2=1$,
as the subcase \ts $x_1=x_2=0$ \ts is similar.  It then follows
from the definition that both \ts $i=1$ \ts and \ts $j=2$ \ts has type  \eqref{eq:B1}.
Then \eqref{eq:iden} follows from 
 	\[ \fa(1,1)\. \fb(0,0)  \ = \  \fc(1,1) \. \fd(0,0). \]
This in turn follows directly from Lemma~\ref{lem:n2-B1B1}.

Second, suppose that \ts $x_1\neq x_2$.
Without loss of generality, assume that \ts $x_1=1$ \ts and \ts $x_2=0$.
It then follows from the definition that \ts  $i=1$ \ts has type  \eqref{eq:B1},
while \ts $j=2$ \ts satisfies \eqref{eq:B2} and does not have type  \eqref{eq:B1}.
Then \eqref{eq:iden} follows from
\[
    \fa(1,0) \. \fb(0,1)  \ = \  \fc(1,1) \. \fd(0,0).
\]
Suppose to the contrary, that this equation does not hold.
It then follows from Lemma~\ref{lem:n2-B1B2}, that \ts $j=2$ \ts has type \eqref{eq:B1},
which gives a contradiction.  This completes the proof for the case $n=2$.

\smallskip

For \ts $n\geq 3$, suppose that the lemma holds for the case \ts $n-1$.
Let $z=(z_1,\ldots, z_n)$ be an element in \ts $\{0,\sq\}^n$, and
suppose that \. $z_k=\sq$ \.  for some \ts $k \in [n]$.
Consider the functions \. $\fa',\fb',\fc',\fd':\{0,1\}^{n-1} \to \rr_{\geq 0}$  given by
	\begin{align*}
	&\fa'\. := \. \fa(\bu^{k-1},\sq, \bu^{n-k})\ts, \ \ \quad \fb'\. := \. \fb(\bu^{k-1},\sq, \bu^{n-k})\ts,\\
	&\fc'\. := \. \fc(\bu^{k-1},\sq, \bu^{n-k})\ts,  \ \ \quad \fd'\. := \. \fd(\bu^{k-1},\sq, \bu^{n-k})\ts.
\end{align*}
Note that these four functions satisfy \eqref{eq:AD-cd} by Lemma~\ref{lem:AD-cd-square}.
Additionally, they satisfy \eqref{eq:AD-eq} by the assumption that \ts $\fa,\fb,\fc,\fd$ \ts
satisfy \eqref{eq:AD-eq}.

Let \. $x',y'\in \{0,1\}^{n-1}$  \. and \. $z' \in \{0,\sq\}^{n-1}$ \. be given by
\begin{align*}
& x' \, := \, (x_1,\ldots, x_{k-1}, x_{k+1},\ldots, x_n),\\
& y' \, := \, (y_1,\ldots, y_{k-1}, y_{k+1},\ldots, y_n),\\
& z' \, := \, (z_1,\ldots, z_{k-1}, z_{k+1},\ldots, z_n).
\end{align*}
It then follows from the induction assumption that
		\begin{align*}
	\fa'\big(x' +z'\big) \. \fb'\big(y' +z'\big) \ = \   \fc'\big(1^{n-1} +z'\big) \. \fd'\big(0^{n-1} +z'\big).
\end{align*}
On the other hand, it follows from the definition that
\begin{alignat*}{3}
\fa'\big(x' +z'\big)  \ &= \ \fa\big(x+z\big), \qquad \qquad \quad \ \ \fb'\big(y' +z'\big)  \ &&= \ \fb\big(y+z\big), \\
\fc'\big(1^{n-1} +z'\big)  \ &= \ \fc\big(1^n+z\big), \qquad \qquad\fd'\big(0^{n-1} +z'\big)  \ &&= \ \fd\big(0^n+z\big).
\end{alignat*}
It then follows that \eqref{eq:iden} holds for all $z \in \{0,\sq\}^n$ that is not equal to $(0,\ldots,0)$.

We now prove that \eqref{eq:iden} holds for  $z=(0,\ldots,0)$.
This is equivalent to
\[ 		\fa\big(x\big)  \fb\big(y\big) \ = \   \fc\big(1,\ldots,1 \big)  \fd\big(0,\ldots,0 \big).  \]
Since \ts $n\geq 3$, there exists distinct \ts $i,j \in [n]$ \ts such that \ts $x_i=x_j$.
Without loss of generality assume that \ts $i=1$ and $j=2$, and that $x_1=x_2=1$.
Then consider the functions \. $\fa',\fb',\fc',\fd':\{0,1\}^{2} \to \rr_{\geq 0}$  \. given by
\begin{alignat*}{4}
	&\fa' \ &&:= \ \fa(\bu,\bu,x_{3},\ldots, x_k),  \quad && \fb' \ &&:= \ \fb(\bu,\bu,y_{3},\ldots, y_k)\\
	&\fc' \  &&:= \ \fc(\bu,\bu,1,\ldots,1), \quad && \fd'\ &&:= \ \fd(\bu,\bu,0,\ldots, 0).
\end{alignat*}
	It follows from Lemma~\ref{lem:AD-cd-res}, that these four functions satisfy \eqref{eq:AD-cd}.
	On the other hand,
 it follows from the conclusion of the previous paragraph (i.e., applying $z=(\sq^2,0^{n-2})$ to \eqref{eq:iden}),
 that
\[   \fa(\sq,\sq,x_{3},\ldots, x_k) \. \fb(\sq,\sq,y_{3},\ldots, y_k) \ = \ \fc(\sq,\sq,1,\ldots,1)\. \fd(\sq,\sq,0,\ldots, 0),
\]
which is equivalent to \. $\fa',\fb',\fc',\fd'$ \. satisfying \eqref{eq:AD-eq}.
By the same reasoning, we also have
\begin{align*}
  \fa(1,\sq,x_{3},\ldots, x_k) \. \fb(0,\sq,y_{3},\ldots, y_k) \ &= \ \fc(1,\sq,1,\ldots,1)\. \fd(0,\sq,0,\ldots, 0),\\
    \fa(\sq,1,x_{3},\ldots, x_k) \. \fb(\sq,0,y_{3},\ldots, y_k) \ &= \ \fc(\sq,1,1,\ldots,1)\. \fd(\sq,0,0,\ldots, 0).
\end{align*}
This is equivalent to $i=1$ and $j=2$ having type \eqref{eq:B1} with respect to \. $\fa',\fb',\fc',\fd'$.
Applying Lemma~\ref{lem:n2-B1B1} to \. $\fa',\fb',\fc',\fd'$, gives
\[ \fa'(1,1) \. \fb'(0,0)  \ = \  \fc'(1,1) \. \fd'(0,0).
\]
By definition, this is equivalent to
\[ 		\fa\big(x\big) \. \fb\big(y\big) \ = \   \fc\big(1,\ldots,1 \big) \. \fd\big(0,\ldots,0 \big).  \]
This completes the proof of Lemma~\ref{lem:iden}. \qed
\smallskip

\subsection{Proof Theorem~\ref{thm:AD-boolean}}\label{ss:AD-boolean-proof}
The $\ts\Leftarrow\ts$ direction is clear, so it suffices to prove the $\ts\Rightarrow\ts$ direction.
Let
\[ A \ := \ \{ \.  i \in [n] \. : \. i \ \text{ has type } \eqref{eq:B1}   \}.   \]
By permuting indices if necessary, without loss of generality we assume that \ts $A=\emp$ \ts
or \ts $A=[k]$ for some \ts $k \in [n]$.
It then follows from the Identification Lemma~\ref{lem:iden}, that \. $\fa,\fb,\fc,\fd$ \.
satisfy the assumptions in the Consistency Lemma~\ref{lem:cst}.
Now the claim \eqref{eq:AD-equal-boolean} follows from the Consistency Lemma~\ref{lem:cst}.
This completes the proof of the theorem.
 \qed

\medskip


\section{Equality conditions for the AD inequality}\label{sec:AD-proof}

In this section, we prove equality conditions for the AD inequality (Theorem~\ref{thm:AD-eq}),
followed by the version for the product of chains (Theorem~\ref{thm:AD-chains}).
Theorem~\ref{thm:AD-eq} is  derived from the version for the Boolean lattice
(Theorem~\ref{thm:AD-boolean}), while Theorem~\ref{thm:AD-chains} is derived directly from Theorem~\ref{thm:AD-eq}.  We need to prove two lemmas in lattice
theory along the way (Lemmas~\ref{lem:structure-partition} and~\ref{lem:lattice-category}).

\subsection{Cross-factoring version}\label{sec:product}
In this section, we present the cross-factoring version of AD equality
conditions for Boolean lattices.  This formulation proves to be more practical
in applications and is a step towards the proof of Theorem~\ref{thm:AD-eq}.

\smallskip

	Let   \. $\fa,\fb,\fc,\fd: \{0,1\}^n \to \rr_{\geq 0}$ \. be
functions satisfying condition \eqref{eq:AD-cd} and \eqref{eq:AD-eq}.
Note that if
\[  \sum_{x \in \{0,1\}^n} \fa(x) \.  \sum_{x \in \{0,1\}^n} \fb(x)  \ = \  \sum_{x \in \{0,1\}^n} \fc(x)  \sum_{x \in \{0,1\}^n} \fd(x)  \  = \  0,\]
then it readily follows that, one of the function $\fa,\fb$ is identically $0$, and one of the functions  $\fc,\fd$ is identically $0$.
This is  a (trivial) sufficient condition for \eqref{eq:AD-eq}.
The next theorem addresses the remaining sufficient and necessary
conditions in terms of auxiliary functions \ts $\fa,\fb,\fc,\fd$.

\smallskip

\begin{thm}[{\rm Cross-factoring version of Theorem~\ref{thm:AD-boolean}}{}]\label{thm:AD-boolean-2}
Let   \. $\fa,\fb,\fc,\fd: \{0,1\}^n \to \rr_{\geq 0}$ \. be  functions satisfying
condition \eqref{eq:AD-cd}, and let $A\subseteq [n]$ be a subset satisfying \eqref{eq:AD-equal-boolean}.
	Then there exist  functions
\. $\ff_1, \fg_1: \{0,1\}^{A}\to \rr_{\geq 0}\ts$,
\. $\ff_2, \fg_2: \{0,1\}^{[n]-A}\to \rr_{\geq 0}\ts$,
and nonnegative constants \. $\ala, \beb, \gac, \ded\geq 0$ \.
such that  \. $\ala\ts\beb=\gac\ts\ded$, \ts and
	\begin{equation}\label{eq:AD-boolean-cross}
	 \left\{\aligned
		\fa(x_1,x_2) \ &= \ \ala \.   {\color{blue}\ff_1}(x_1) \. {\color{blue}\ff_2}(x_2),\\
		\fb(x_1,x_2) \ &= \  \beb \.  {\color{red}\fg_1}(x_1) \. {\color{red}\fg_2}(x_2),\\
		\fc(x_1,x_2) \ &= \  \gac \.  {\color{blue}\ff_1}(x_1) \. {\color{red}\fg_2}(x_2),\\
		\fd(x_1,x_2) \ &= \  \ded \.  {\color{red}\fg_1}(x_1) \. {\color{blue}\ff_2}(x_2).
	\endaligned \right.  \qquad
\text{for all \, $(x_1,x_2) \in \{0,1\}^{A} \times \{0,1\}^{[n]-A}$.}
	\end{equation}
\end{thm}

\begin{proof}
Without loss of generality, assume that functions $\fa, \fb, \fc, \fd$ are not identically zero.
For example, if \ts $\fa=\fc= \zero$, we can set \ts $A = \varnothing$, \ts $\ff_1=\fg_1=\one$,
\ts $\ff_2 = \fd$, \ts $\fg_2 = \fb$, \ts $\al = \ga = 0$,\ts $\be = \de = 1$, and
note that \eqref{eq:AD-boolean-cross} holds.  Other cases are analogous.
	
	Let \ts $A$ \ts as given in Theorem~\ref{thm:AD-boolean}.
	Let \. $(u_1,u_2), (v_1,v_2) \in \{0,1\}^{A} \times \{0,1\}^{[n]-A}$ \. be such that
	\[ \fa(u_1,u_2) \, > \, 0 \qquad \text{and} \qquad \fb(v_1,v_2) \, > \, 0,   \]
	which exist by assumption that \ts $\fa,\fb\ne \zero$.
	It then follows from \eqref{eq:AD-equal-boolean} that
	\[   \fa(u_1,u_2) \. \fb(v_1,v_2) \ = \  \fc(u_1,v_2) \. \fd(v_1,u_2) \ > \ 0. \]
	By rescaling the functions if necessary, without loss of generality, we can assume that
	\[  \fa(u_1,u_2) \, = \,  \fb(v_1,v_2) \, = \,  \fc(u_1,v_2) \, = \, \fd(v_1,u_2) \, = \, 1.  \]
	
	It now follows from Theorem~\ref{thm:AD-boolean}, that
	\begin{align}\label{eq:boolean-2-1}
\fa({\color{magenta}x_1},u_2) \ = \  \fa({\color{magenta}x_1},u_2)  \fb(v_1,v_2) \ = \ \fc({\color{magenta}x_1},v_2)  \fd(v_1,u_2) \ = \ \fc({\color{magenta}x_1},v_2),
	\end{align}
for all \ts $x_1 \in \{0,1\}^{A}$.
	We then define \. $\ff_1:  \{0,1\}^{A} \to \rr_{\geq 0}$ \. by
	\begin{align*}
\ff_1({\color{magenta}x_1}) \ := \  \fa({\color{magenta}x_1},u_2) \ = \ \fc({\color{magenta}x_1},v_2),
	\end{align*}
	and note that the second equality is due to \eqref{eq:boolean-2-1}.
	
	By a similar argument, for all
\.$({\color{magenta}x},{\color{magenta}y}) \in \{0,1\}^{A} \times \{0,1\}^{[n]-A}$\., we have:
	\begin{align*}
				\fb({\color{magenta}x_1}, v_2) \ &= \  \fd({\color{magenta}x_1},u_2), \quad
		\fa(u_1, {\color{magenta}x_2}) \, = \,  \fd(v_1,{\color{magenta}x_2}), \quad
		\fb(v_1, {\color{magenta}x_2}) \, = \,  \fc(u_1,{\color{magenta}x_2}).
	\end{align*}
	We then define functions \. $\fg_1: \{0,1\}^{A}\to \rr_{\geq 0}$ \. and
\. $\ff_2, \fg_2: \{0,1\}^{[n]-A}\to \rr_{\geq 0}$ \. as follows:
$$\fg_1({\color{magenta}x_1})  := 	\fb({\color{magenta}x_1}, v_2)  =   \fd({\color{magenta}x_1},u_2), \ \
\ff_2({\color{magenta}x_2})  :=  \fa(u_1, {\color{magenta}x_2})  =  \fd(v_1,{\color{magenta}x_2}), \ \
\fg_2({\color{magenta}x_2})  :=  \fb(v_1, {\color{magenta}x_2}) =   \fc(u_1,{\color{magenta}x_2}).
$$	
	
Now note that, for all \.$(x_1,x_2) \in \{0,1\}^{A} \times \{0,1\}^{[n]-A} $\., we have:
\begin{align*}
\fa( {\color{magenta}x_1}, {\color{magenta}x_2}) \ = \  \fa( {\color{magenta}x_1}, {\color{magenta}x_2}) \fb(v_1,v_2) \ = \ \fc( {\color{magenta}x_1},v_2) \fd(v_1, {\color{magenta}x_2})   \ = \ \ff_1({\color{magenta}x_2}) \ff_2({\color{magenta}x_2}),
\end{align*}
	where the second equality follows from Theorem~\ref{thm:AD-boolean}, and the third equality is by the definition
of functions \. $\ff_1,\ff_2\ts$. Similarly, we have:
	\begin{alignat*}{3}
		\fb( {\color{magenta}x_1}, {\color{magenta}x_2}) \ &= \  \fa(u_1,u_2) \fb( {\color{magenta}x_1}, {\color{magenta}x_2})  \ &&= \ \fc(u_1,{\color{magenta}x_2}) \fd({\color{magenta}x_1},u_2)   \ &&= \ \fg_1({\color{magenta}x_1}) \fg_2({\color{magenta}x_2}),\\
\fc( {\color{magenta}x_1}, {\color{magenta}x_2}) \ &= \  \fc( {\color{magenta}x_1}, {\color{magenta}x_2}) \fd(v_1,u_2) \ &&= \ \fa( {\color{magenta}x_1},u_2) \fb(v_1, {\color{magenta}x_2})   \ &&= \ \ff_1({\color{magenta}x_1}) \fg_2({\color{magenta}x_2}),\\
		\fd( {\color{magenta}x_1}, {\color{magenta}x_2}) \ &= \  \fc(u_1,v_2) \fd( {\color{magenta}x_1}, {\color{magenta}x_2})  \ &&= \ \fa(u_1,{\color{magenta}x_2}) \fb({\color{magenta}x_1},v_2)   \ &&= \ \fg_1({\color{magenta}x_1}) \ff_2({\color{magenta}x_2}).
	\end{alignat*}
	This completes the proof.
\end{proof}


\smallskip

\subsection{Support product lemma}\label{sec:structural}
The following lemma will be used to generalize Theorem~\ref{thm:AD-boolean-2}
to apply to other types of lattices that are not Boolean lattices.
Recall the notation of Birkhoff's Theorem~\ref{thm:Birkhoff}.

\smallskip

\begin{lemma}[{\rm \defn{Support product lemma}\ts}{}] \label{lem:structure-partition}
Let  \ts $\cL=(\aL,\vee,\wedge)$ \ts  be a distributive lattice.  Let
\. $\ff: \aLr \to \rr_{\geq 0}$ \. be a function, which can be uniquely extended  to \. $\fF:\{0,1\}^{\aJr(\cL)} \to \rr_{\geq 0}$ \. by
\[\fF(y) \ := \
\begin{cases}
	\ff(x) & \text{ if } \  y = \phi(x) \  \text{ for some } \ x \in \aLr,\\
	0 & \text{ otherwise.}
	\end{cases}
\]
Let \. $X_1,X_2\subseteq \aJr(\cL)$ \. such that \ts $X_1 \cap X_2= \emp$ \ts and \ts $X_1 \cup X_2=\aJr(\cL)$.
Suppose there exist functions \. $\fF_1:\{0,1\}^{X_1} \to \rr_{\geq 0}$ \. and \. $\fF_2:\{0,1\}^{X_2} \to \rr_{\geq 0}$ \.
such that
\begin{equation}\label{eq:f-dec}
	\fF(x_1,x_2) \ = \  \fF_1(x_1) \fF_2(x_2) \qquad \text{for all \quad }  (x_1,x_2) \in \{0,1\}^{X_1} \times \{0,1\}^{X_2}.
\end{equation}
Then:
$$\overline{\supp \ff} \ \simeq  \ \overline{\supp \fF_1} \. \times \. \overline{\supp \fF_2}\,.
$$
\end{lemma}

\begin{proof}
Let \. $\aL_1,\aL_2 \subseteq \{0,1\}^{\aJ(\cL)}$ \. given by
	\begin{align*}
		\aL_1 \ &:= \ \big \{ \. (x_1, {\color{magenta}{z_2}}) \in \{0,1\}^{X_1} \times \{0,1\}^{X_2} \, : \,  x_1 \in \overline{\supp \fF_1} \.  \big \},\\
		\aL_2 \ &:= \ \big \{ \. ({\color{magenta}{z_1}}, x_2) \in \{0,1\}^{X_1} \times \{0,1\}^{X_2} \, : \,  x_2 \in \overline{\supp \fF_2} \.  \big \},
	\end{align*}
	where \ts $({\color{magenta}{z_1}}, {\color{magenta}{z_2}})$ \ts is the minimum element of \. $\phi(\overline{\supp \ff})$\..
Denote \ts $\cL_1 = (\aL_1,\vee,\wedge)$, \.  $\cL_2 = (\aL_2,\vee,\wedge)$, \. and note that these are sublattices of the Boolean lattice.
	It suffices to show that
	\[ \phi(\overline{\supp \ff}) \ = \  \aL_1 \vee \aL_2.   \]

	We first show that \. $\phi(\overline{\supp \ff}) \. \subseteq \aL_1 \vee \aL_2$\..
	Let $(x_1,x_2)$ be an arbitrary element of $\phi(\supp \ff)$.
	It follows from \eqref{eq:f-dec} that \ts $x_1 \in \supp \fF_1$ \ts and \ts $x_2 \in \supp \fF_2$\ts.
	Hence we have
	\[   \phi(\supp \ff) \ \subseteq  \ \aL_1 \vee \aL_2.    \]
	Since \ts $\phi$ \ts is a lattice embedding and \ts $\cL_1, \cL_2$ \ts are lattices, it then follows that
	\[ \phi(\overline{\supp \ff})  \ = \  \overline{\phi(\supp \ff)} \ \subseteq
\ \overline{\aL_1 \vee \aL_2} \ = \  \aL_1 \vee \aL_2\.,  \]
	as desired.
	
We now show that \. $\aL_1 \vee \aL_2 \subseteq \phi(\overline{\supp \ff}).$
By the symmetry, it suffices to show that \. $\aL_1\subseteq \phi(\overline{\supp \ff})$.
Let \. $x$ \. be an arbitrary element of $\overline{\supp \fF_1}$.
It  follows that there exists
\[ y_1, \. y_2, \. \ldots \., \. y_k  \ \in \supp \fF_1 \]
such that $x$ can be obtained by applying a sequence of  join/wedge operations to those elements.
Now $w$ be an arbitrary element of $\supp \fF_2$. It then follows from \eqref{eq:f-dec} that
\[  (y_{1}, w), \   (y_{2}, w), \. \ldots \. , \  (y_{k}, w)  \ \in  \ \supp \fF_1  \times \supp  \fF_2 \ = \  \phi(\supp \ff).\]
By applying the same sequence of join/wedge operations to these new elements, we conclude that
\begin{equation}\label{eq:f-dec-2}
 (x,w)  \ \in \ \phi(\overline{\supp \ff})  \quad \text{ for any } \ w \in \supp \fF_2.
\end{equation}

Now, since  \ts $({\color{magenta}{z_1}}, {\color{magenta}{z_2}})$ \ts is the minimum element of \. $\phi(\overline{\supp \ff})$\.,  it follows that there exists
\[   w_{1}, \. w_{2}, \. \ldots \. , \. w_{\ell} \  \in \  \supp \fF_2 \]
such that $ {\color{magenta}{z_2}}$ can be obtained by applying a sequence of  join/wedge operations to those elements.
Then, it follows from \eqref{eq:f-dec-2} that
\[ (x, w_{1}), \   (x, w_{2}), \. \ldots \. , \  (x, w_{\ell})  \ \in \phi(\overline{\supp \ff}). \]
By applying the same sequence of join/wedge operations to these new elements, we  conclude that
\[  (x,{\color{magenta}{z_2}})  \ \in \ \phi(\overline{\supp \ff}).  \]
Since the choice of \. $x \in \overline{\supp \fF_1}$ \. is arbitrary, we conclude that
\. $\aL_1 \subseteq \phi(\overline{\supp \ff})$. This completes the proof.
\end{proof}

\smallskip

\subsection{Proof of Theorem~\ref{thm:AD-eq}}\label{sec:proof-AD-abstract}
The proof of the $\ts\Leftarrow\ts$ is clear, so we present only the proof of the $\ts\Rightarrow\ts$ direction.
By the same argument as in the beginning of the proof of Theorem~\ref{thm:AD-boolean-2},
without loss of generality, assume that the functions \ts $\fa,\fb,\fc,\fd$ \ts are not identically equal to~$0$.

Let \. $\phi:\aL \to \{0,1\}^{\aJ(\cL)}$ \. be the lattice embedding in Birkhoff's Theorem~\ref{thm:Birkhoff}.
Let \. $\fA,\fB,\fC,\fD: \{0,1\}^{\aJ(\cL)} \to \rr_{\geq 0}$ \. be the unique extension of \ts $\fa,\fb,\fc,\fd$.
It then follows from Theorem~\ref{thm:AD-boolean-2}, that
there exist disjoint subsets \ts $X_1, X_2$ \ts satisfying \ts $X_1\cup X_2=\aJ(\cL)$,  functions
\. $\fF_1, \fG_1: \{0,1\}^{X_1}\to \rr_{\geq 0}$,  \.   $\fF_2, \fG_2: \{0,1\}^{X_2}\to \rr_{\geq 0}$ \. and positive constants \. $\ala, \beb, \gac, \ded>0$ \. satisfying $\ala\beb=\gac\ded$, such that
\begin{align*}
	\fA(x_1,x_2) \ &= \ \ala \.  {\color{blue}\fF_1}(x_1) \. {\color{blue}\fG_2}(x_2)\ts,\qquad
	\ts\fB(x_1,x_2) \ = \  \beb  \.  {\color{red}\fG_1}(x_1) \. {\color{red}\fG_2}(x_2)\ts,\\
	\fC(x_1,x_2) \ &= \  \gac \.  {\color{blue}\fG_1}(x_1) \. {\color{red}\fG_2}(x_2)\ts,\qquad
	\fD(x_1,x_2) \ = \  \ded \.  {\color{red}\fG_1}(x_1) \. {\color{blue}\fF_2}(x_2)\ts,
\end{align*}
for all \.$(x_1,x_2) \in \{0,1\}^{X_1} \times \{0,1\}^{X_2} $.
By rescaling the functions \. $\fa,\fb,\fc,\fd$ \. if necessary, we can without loss of generality assume that
\. $\ala= \beb= \gac = \ded=1$.

Let \. $\fh: \aL \to \rr_{\geq 0}$ \. be defined by \. $\fh := \fa + \fb + \fc + \fd$,
and let \. $\fH:\ \{0,1\}^{\aJ(\cL)} \to \rr_{\geq 0}$ \. be the corresponding unique extension.
Note that
\[ \fH(x_1,x_2) \ =  \ [\fF_1+\fG_1](x_1)\. \cdot \. [\fF_2+\fG_2](x_2)\ts,
 \]
for all \.$(x_1,x_2) \in \{0,1\}^{X_1} \times \{0,1\}^{X_2}$.
Let  \. $\aL_1 \subseteq \{0,1\}^{X_1}$ \. and  \. $\aL_2 \subseteq \{0,1\}^{X_2}$ \. be given by
\[ \aL_1 \ := \  \overline{\supp (\fF_1+\fG_1)}, \qquad   \aL_2 \ := \  \overline{\supp (\fF_1+\fG_1)}.  \]
We then have
\[ \aL \ = \   \overline{\supp \fh} \ \simeq \ \aL_1 \times \aL_2\.,    \]
where the first equality follows from the assumption that \ts $\aL$ \ts is the lattice closure of the support of \.
$\fa+\fb+\fc+\fd$, and the second equality follows from  Lemma~\ref{lem:structure-partition}.

Finally, let \. $\ff_1,\fg_1$ \. be the functions \. $\fF_1,\fG_1$ \. restricted to  \ts $\aL_1$\ts,
and let  \. $\ff_2,\fg_2$ \. be the functions \. $\fF_2,\fG_2$ \. restricted to  \ts $\aL_2$\ts.
It follows from the setup above that these functions satisfy the conclusion of Theorem~\ref{thm:AD-eq}.
This completes the proof.
\qed

\smallskip



\subsection{AD equality for products of chains}\label{ss:special-AD-weaker}
A \defnb{direct product of chains} is a lattice
\begin{equation}\label{eq:chains-product}
\cL \, = \, \cL(N_1,\ldots,N_n) \, := \, \cC_{N_1} \. \times \. \cdots \. \times \. \cC_{N_n}\,,
\end{equation}
where $i$-th chain \ts $\cC_{N_i}$ \ts has elements in
\ts $[N_i]=\{1,\ldots, N_i\}$, and the join and meet operations are given
by pointwise maximum and minimum, respectively.

\smallskip

\begin{thm}[{\rm \defn{AD equality for direct products of chains}\ts}{}]\label{thm:AD-chains}
	Let \ts  $\cL=(\aL,\vee,\wedge)$ \ts be a direct product of $n$ chains \. as in \eqref{eq:chains-product}.
	Let   \. $\fa,\fb,\fc,\fd: \aL \to \rr_{\geq 0}$ \. be  functions satisfying condition \eqref{eq:AD-cd}.
	Then \eqref{eq:AD-eq} holds \ \underline{\em if and only if}  \
there exist a subset \ts $A \subseteq [n]$, lattices
	\[   \cL_1 \, = \, (\aL_1,\vee,\wedge) \ := \  \bigotimes_{i \in A} \cC_{N_i}\,,
\qquad  \cL_2  \, = \, (\aL_2,\vee,\wedge) \ := \  \bigotimes_{j \notin A} \cC_{N_j}\,,  \]
	 functions \. $\ff_1, \fg_1: \aL_1\to \rr_{\geq 0}\ts$,
\. $\ff_2, \fg_2: \aL_2\to \rr_{\geq 0}\ts$,
and nonnegative constants \. $\ala, \beb, \gac, \ded\geq 0$ \.
such that  \. $\ala\ts\beb=\gac\ts\ded$, \ts and \eqref{eq:AD-abstract-2} holds.
\end{thm}

\smallskip

In the setting of Theorem~\ref{thm:AD-eq}, we are claiming that one only has to
consider sublattices which are themselves direct product of chains.  This
\defng{unique decomposability} \ts is well known phenomenon in lattice theory,
see e.g. Corollary~4 in \cite[$\S$III.4]{Gra98}.  For our purposes we need
the following variation:

\smallskip

\begin{lemma}\label{lem:lattice-category}
Let \ts $\cL_1=(\aL_1,\vee',\wedge')$ \ts and \ts
$\cL_2=(\aL_2,\vee'',\wedge'')$\ts  be  finite distributive lattices,
and let \ts $\cL=(\cL,\vee,\wedge)$ \ts be a product of $n$ chains as in \eqref{eq:chains-product}.
Suppose \. $\iota:\aL_1 \times \aL_2 \to \aL$ \. is an injective lattice homomorphism.
Let \ts $z_1\in \aL_1$ \ts and \ts $z_2\in \aL_2$ \ts denote the minimum elements
of \ts $\cL_1$ \ts and \ts $\cL_2$, respectively, and let \.
$(c_1, \ldots, c_n) := \iota(z_1, z_2)$.
Then there exists a subset \ts $A\subseteq [n]$, such that
$$
\iota\big(\aL_1 \times \{z_2\}\big) \, = \, \aL_1^\ast \times \{ (c_j)_{j \in [n]-A} \}, \qquad
\iota\big(\{z_1\} \times \aL_2\big) = \{ (c_i)_{i \in A} \} \times \aL_2^\ast, 
$$
where \ts $\cL_1 ^\ast= (\aL_1^\ast,\vee,\wedge)$ \ts is a sublattice of \.
$\bigotimes_{i \in A} \cC_{N_i}$ \.  and \. $\cL_2^\ast = (\aL_2^\ast,\vee,\wedge)$ \ts
is a sublattice of \. $\bigotimes_{j \notin A} \cC_{N_j}\ts.$
\end{lemma}

\begin{proof}
Let \ts $Z_1\in \aL_1$ \ts and \ts $Z_2\in \aL_2$ \ts denote the maximum elements
of \ts $\cL_1$ \ts and \ts $\cL_2\ts$, respectively, and let \.
$(C_1,\ldots, C_n) := \iota(Z_1,Z_2)$.
Similarly, let \. $(a_1,\ldots, a_n):= \iota(Z_1,z_2)$\., \. $(b_1,\ldots, b_n):=\iota(z_1,Z_2)$.
Note that
\begin{align*}
(a_1,\ldots, a_n) \. \vee \. (b_1,\ldots, b_n) \ &= \ \iota(Z_1,z_2) \. \vee \. \iota(z_1,Z_2) \ = \  \iota(Z_1, Z_2) \
= \ (C_1,\ldots, C_n), \\
(a_1,\ldots, a_n)\. \wedge \. (b_1,\ldots, b_n) \ &= \ \iota(Z_1,z_2) \. \wedge \. \iota(z_1,Z_2) \ = \  \iota(z_1, z_2) \ \
\ts = \ (c_1,\ldots, c_n).
\end{align*}
Hence we have:
\[  a_i \vee b_i \ = \ C_i\,,  \quad a_i \wedge b_i \ = \ c_i \quad \ \text{for all \ \ $i \in [n]$}.  \]
Define \. $ A := \  \{i \in [n] \, : \, a_i= C_i \}$, and let
$$
\pi_1: \cL \to \bigotimes_{i \in A} \cC_{N_i}\,,  \qquad
\pi_2: \cL \to \bigotimes_{j \notin A} \cC_{N_j}
$$
be projections onto the coordinates indexed by \ts $A$ \ts and \ts $[n]-A$, respectively.
It  suffices to show that, for all \ts $x_1\in \aL_1$ \ts and \ts $x_2 \in \aL_2$, we have:
\begin{align}\label{eq:regrator}
	\pi_2 \circ \iota (x_1,z_2) \ = \  (c_j)_{j \notin A} \quad \text{and} \quad
	\pi_1 \circ \iota (z_1,x_2) \ = \  (c_i)_{i \in A}\..
\end{align}
We prove only the first claim, as the proof of the second claim is analogous.

First, note that
\[  \pi_2 \circ \iota (x_1,z_2)  \ \geqslant \   \pi_2 \circ \iota (z_1,z_2)
\ = \ \pi_2(c_1,\ldots, c_n) \ = \   (c_j)_{j \notin A}\., \]
where $\geqslant$ is the natural partial order on \ts $\bigotimes_{i \in A} \cC_{N_i}\ts$.
On the other hand, we similarly have
\[  \pi_2 \circ \iota (x_1,z_2)  \ \leqslant \ \pi_2 \circ \iota (Z_1,z_2) \ = \
\pi_2 (a_1,\ldots, a_n) \ = \ (c_j)_{j \notin A} \., \]
where the last equality is by the definition of~$A$.
This implies the first claim in \eqref{eq:regrator}, and completes the proof.
\end{proof}

\smallskip

\begin{proof}[Proof of Theorem~\ref{thm:AD-chains}]
	It follows from Theorem~\ref{thm:AD-eq} that the lattice closure of \. $\supp \ts (\fa+\fb+\fc+\fd)$ \.
is isomorphic \ts $\cL_1 \times \cL_2$ \ts for some finite distributive lattices \ts $\cL_1=(\aL_1,\vee',\wedge')$ \ts
and \ts $\cL_2=(\aL_2,\vee'',\wedge'')$. 	Furthermore, there exist \. $\al, \be, \ga, \de>0$,
such that \. $\al\ts\be=\ga \ts \de$, and \eqref{eq:AD-abstract-2} holds.
	
Let $\iota: \aL_1 \times \aL_2 \to \aL$ be the given embedding.
Apply Lemma~\ref{lem:lattice-category}  to $\iota$, to obtain a subset \ts $A\subseteq [n]$.
By the lemma, we can  identify \ts $\cL_1$ \ts and \ts $\cL_2$ \ts with sublattices of \ts
$\bigotimes_{i \in A} \cC_{N_i}$ \ts  and \ts $\bigotimes_{j \notin A} \cC_{N_j}\ts$, respectively.
	
	  We now formally embed $\cL_1$ into the product of chains \ts $\bigotimes_{i \in A} \cC_{N_i}$\ts,
and then  extend the functions \. $\ff_1, \fg_1: \aL_1 \to \rr_{\geq 0}$ \.
to this entire space by setting them to be zero outside the original~$\aL_1$.
	  We apply the same extension to $\cL_2,\ff_2,\fg_2$.
	  The theorem now follows immediately from \eqref{eq:AD-abstract-2}.
\end{proof}

\medskip

\section{FKG equality conditions}\label{s:FKG}

In this section, we prove equality conditions of the FKG inequality (Theorem~\ref{thm:FKG-eq}).
We then state the version for the product of chains (Theorem~\ref{thm:FKG-chain}).

\subsection{Proof of Theorem~\ref{thm:FKG-eq}}\label{sec:proof-FKG-abstract}

The proof of the \ts $\Leftarrow$ \ts direction is clear,
so we only prove the \ts $\Rightarrow$ \ts direction.
By adding positive constants to the functions if necessary,
without loss of generality we can assume that functions \ts
$\ff,\fg$ \ts are strictly positive.
We can then apply Theorem~\ref{thm:AD-eq} to functions
\[ \fa \, := \,  \ff \cdot \ts\mu\ts,  \quad
\fb \, := \,  \fg \cdot  \ts\mu\ts,  \quad
\fc \, := \,  \ff \cdot \fg \cdot  \ts\mu\ts,  \quad
\fd \, := \,  \mu\ts.
\]
Note that these four functions satisfy \eqref{eq:AD-cd} and \eqref{eq:AD-eq}.
Hence there exists finite distributive lattices \. $\cL=(\aL_1,\vee',\wedge')$, \.
$\cL_2=(\aL_2,\vee'',\wedge'')$, such that \. $\cL\simeq \cL_1\times \cL_2\ts,$
 functions \. $\fp_1,\fq_1:\aL_1 \to \rr_{\geq 0}\ts,$
\. $\fp_2,\fq_2:\aL_2 \to \rr_{\geq 0}$\.,
and positive constants \. $\ala, \beb, \gac, \ded>0$, such that \.
$\ala\ts \beb=\gac\ts \ded$, \ts and
\begin{align}
	\label{eq:FKG-1}\ff(x_1,x_2)\.  \mu(x_1,x_2) \ &= \ \ala \.  {\color{blue}\fp_1}(x_1) \. {\color{blue}\fp_2}(x_2), \hskip1.cm\\
	\label{eq:FKG-2}\fg(x_1,x_2) \. \mu(x_1,x_2) \ &= \  \beb \. {\color{red}\fq_1}(x_1) \. {\color{red}\fq_2}(x_2), \hskip1.cm\\
	\label{eq:FKG-no} \ff(x_1,x_2) \.\fg(x_1,x_2) \. \mu(x_1,x_2) \ &= \  \gac \. {\color{blue}\fp_1}(x_1) \. {\color{red}\fq_2}(x_2), \hskip1.cm\\
	\label{eq:FKG-3}\mu(x_1,x_2) \ &= \  \ded \. {\color{red}\fq_1}(x_1) \. {\color{blue}\fp_2}(x_2), \hskip1.cm
\end{align}
for all \. $(x_1,x_2) \in \aL_1\times \aL_2\ts$.
By rescaling  \. $\ff,\fg,\mu$ \. if necessary, without loss of generality, we can assume that
\. $\ala= \beb= \gac = \ded=1$.
Since \ts $f,g,\mu$ \ts are strictly positive, so are  \. $\fp_1,\fp_2,\fq_1,\fq_2\ts$.

We now define  \. $\mu_1, \ff' :\aL_1 \to \rr_{\geq 0}$ \.  and \. $\mu_2,\fg':\aL_2 \to \rr_{\geq 0}$ \. by
\begin{align*}
\mu_1 \ := \  {\color{red}\fq_1}\,, \qquad \mu_2 \ := \ {\color{blue}\fp_2}\,, \qquad
\ff' \ := \  \frac{{\color{blue}\fp_1}}{{\color{red}\fq_1}}\,, \qquad  \fg' \ := \  \frac{{\color{red}\fq_2}}{{\color{blue}\fp_2}}\..
\end{align*}
It then follows from \eqref{eq:FKG-3} that
\begin{align}\label{eq:FKG-4}
	\mu(x_1,x_2) \ = \  \mu_1(x_1) \. \mu_2(x_2) \qquad \text{ for all }  \ (x_1,x_2) \in \aL.
\end{align}
Next, it follows from \eqref{eq:FKG-1}, \eqref{eq:FKG-2} and \eqref{eq:FKG-4} that
\begin{align*}
	\ff({\color{blue}x_1},{\color{red}x_2}) \, = \,  \ff'({\color{blue}x_1})\,, \quad
\fg({\color{blue}x_1},{\color{red}x_2}) \, = \,  \fg'({\color{red}x_2})\,,
 \quad  \text{ for all } \ ({\color{blue}x_1},{\color{red}x_2}) \in \aL_1 \times \aL_2\.,
\end{align*}
as desired. \qed

\smallskip

\subsection{FKG inequality for product of chains}\label{ss:special-FKG}
In applications, the FKG inequality~\eqref{eq:FKG} is often applied to lattices that are direct products of chains.
In  the next theorem we present a version of the equality conditions for the FKG inequality tailored to these specific types of lattices.

\smallskip

 \begin{thm}\label{thm:FKG-chain}
	Let \ts  $\cL=(\aL,\vee,\wedge)$ \ts be a direct product of $n$ chains \. as in \eqref{eq:chains-product}.
	Let \. $\mu,\ff,\fg$ \. be as in Theorem~\ref{thm:FKG}.
	Then \eqref{eq:FKG} is an equality, i.e.\
\ \underline{\em if and only if}  \
there exist a subset \ts $A \subseteq [n]$, lattices
	\[   \cL_1 \, = \, (\aL_1,\vee,\wedge) \ := \  \bigotimes_{i \in A} \cC_{N_i}\,,
\qquad  \cL_2  \, = \, (\aL_1,\vee,\wedge) \ := \  \bigotimes_{j \notin A} \cC_{N_j}\,,
\]	
functions \. $\ff':\aLr_1 \to \rr$, \ $\fg':\aLr_2 \to \rr$, such that
	 \begin{align}
	\ff({\color{blue}x_1},{\color{red}x_2}) \ = \  \ff'({\color{blue}x_1})\ts,
\quad \fg({\color{blue}x_1},{\color{red}x_2}) \ = \  \fg'({\color{red}x_2})
\quad \text{ for all } \quad ({\color{blue}x_1},{\color{red}x_2}) \in \supp \mu\ts,
\end{align}
	and  measures \. $\mu_1:\aLr_1 \to \rr_{\geq 0}$\ts, \ $\mu_2:\aLr_2 \to \rr_{\geq 0}$\ts,  such that
	\begin{align}
		&\mu(x_1,x_2) \ = \  \mu_1(x_1) \. \mu_2(x_2) \qquad \forall \ (x_1,x_2) \in \aLr.
	\end{align}
\end{thm}

\smallskip

The theorem is a direct consequence of Theorem~\ref{thm:AD-chains}.  Alternatively,
it can be derived from Theorem~\ref{thm:FKG-eq} and Lemma~\ref{lem:lattice-category}
in the same way Theorem~\ref{thm:AD-chains} is derived from Theorem~\ref{thm:AD-eq}.
We omit both  proofs.

\medskip

 \section{LPP equality conditions}\label{sec:LPP}


\subsection{Setup} \label{ss:LPP-setup}
The proof of the \ts $\Leftarrow$ \ts direction is clear, so we present
only the proof of the \ts $\Rightarrow$ \ts direction.  In fact, we
prove the following much stronger result.

\smallskip

\begin{thm}\label{thm:LPP-eq-eval}
		Let \ts ${\color{blue}\la}/{\color{blue}\mu}$ \ts and \ts ${\color{red}\nu}/{\color{red}\rho}$ \ts be skew shapes.
Assume that \. $(\blue{\la}\vee \red{\nu})/(\blue{\mu} \wedge \red{\rho})$ \.  is  connected.
Let \. $N\geq \max\{\ell(\lambda),\ell(\nu)\}+1$ \.
and let \. $\bz=(z_1,\ldots, z_N) \in \nn_{> 0}^{N}$ \. be a positive vector.
Then:
\begin{equation}\label{eq:LPP-z}
	\ s_{{\color{blue}\la}/{\color{blue}\mu}}(\bz)  \. \cdot \. s_{{\color{red}\nu}/{\color{red}\rho}}(\bz)
	\ = \ s_{{\color{blue}\la}/{\color{blue}\mu} \. \vee \. {\color{red}\nu}/{\color{red}\rho}}(\bz)
\. \cdot \. s_{{\color{blue}\la}/{\color{blue}\mu} \. \wedge \. {\color{red}\nu}/{\color{red}\rho}}(\bz)
\end{equation}
\underline{\em if and only if} \  \eqref{eq:LPP-eq-comb} holds.
\end{thm}

\smallskip

Theorem~\ref{thm:LPP-eq-eval} immediately implies Theorem~\ref{thm:LPP-eq},
since \eqref{eq:LPP-z} is an evaluation of the equality case of \eqref{eq:LPP}.
In particular, for \ts $\bz = (1^n)$, Theorem~\ref{thm:LPP-eq-eval} gives
Theorem~\ref{thm:SSYT-eq}.  The proof of Theorem~\ref{thm:LPP-eq-eval}
occupies the rest of the section.

 \smallskip

\subsection{Applying AD equality conditions} \label{ss:LPP-AD-eq}
Let \. $N \geq \max\{\ell({\la}),\ell({\nu})\}+1$.
Let \ts $\cL=(\aL,\vee,\wedge)$ \ts be the direct product of chains $\cC_{ij}$
on the elements \ts $[N] \cup \{-\infty,\infty\}$, where the product is over \.
$(i,j) \in (\la \vee {\nu})/({\mu} \wedge {\rho})$. In particular, we have \.
$|\aL| = (N+2)^K$, where \ts $K = |\la\vee\nu|-|\mu\wedge\rho|$.
Here we assume that elements in \ts $\cC_{(ij)}$ have the natural total order: \ts $-\infty \prec 1 \prec  \ldots \prec  N \prec  \infty$.
The elements \ts $T\in \aL$ \ts are functions \. $T: (\la \vee \nu)/(\mu \wedge \rho) \to [N] \cup \{-\infty,\infty\}$,
and can be viewed as fillings of the shape \ts $(\la \vee \nu)/(\mu \wedge \rho)$.

For a skew shape \ts $\pi/\tau$, where \. ${\mu} \wedge {\rho} \subseteq \tau \subseteq \pi \subseteq {\lambda} \vee {\nu}$\.,
let \. $\varphi^{\pi/\tau}:\aL \to \Rb_{\geq 0}$ \.  be the indicator function on elements  \. $T:({\la} \vee {\nu})/({\mu} \wedge {\rho}) \to [N]$ \.  satisfying all these properties:
\begin{equation*}
\left\{
\aligned \,
	T(i,j) \, =  \, -\infty \qquad &\text{if } \   (i,j) \in \tau\ts,\\
	T(i,j) \, = \, \infty \qquad &\text{if } \ (i,j) \in {(\la \vee \nu)/\pi}\ts,\\
		T(i,j) \,  \leq  \, T(i,j+1) \qquad &\text{if } \ (i,j), \. (i,j+1) \in {\la} \vee {\nu}\ts, \\
T(i,j)+1 \, \leq \, T(i+1,j) \qquad &\text{if } \ (i,j), \. (i+1,j) \in {\la} \vee {\nu}\ts,
\endaligned\right.
\end{equation*}
where we assume that \ts $-\infty+1=-\infty$.  
In other words, we have \. $\varphi^{\pi/\tau}(T)=1$ \. for all functions \. $T^\ast\in \SSYT(\pi/\tau,N)$,
where $T^\ast$ is obtained from~$T$ by removing all boxes with entries in \ts $\{-\infty,\infty\}$,
and we have \. $\varphi^{\pi/\tau}=0$ \. otherwise.

Let \. $\fa,\fb,\fc,\fd:\aL \to \rr_{\geq 0}$ \. be the functions given by
\begin{alignat*}{3}
	\fa(T) \ &:= \ \varphi^{\la/\mu}(T) \prod_{i\in [N]}z_{T(i)}\ts, \qquad
&&	\fb(T) \ &&:= \ \varphi^{\nu/\rho}(T) \prod_{i\in [N]}z_{T(i)}\ts, \\
	\fc(T) \ &:= \ \varphi^{(\la/\mu) \wedge (\nu/\rho)}(T) \prod_{i\in [N]}z_{T(i)}\ts, \qquad
&&	\fd(T) \ &&:= \ \varphi^{(\la/\mu) \vee (\nu/\rho)}(T) \prod_{i\in [N]}z_{T(i)}\ts.
\end{alignat*}
In was shown in \cite[\S8]{CP-multi}, that these four functions satisfy \eqref{eq:AD-cd}.

Now, it follows from \eqref{eq:LPP-z} that
\[ 	
s_{{\la}/{\mu}}(\bz) \. \cdot \. s_{\nu/{\rho}}(\bz)
\ = \ s_{{\la}/{\mu} \. \wedge  \. {\nu}/{\rho}}(\bz)  \. \cdot \.
s_{{\la}/{\mu} \. \vee \. {\nu}/{\rho}}(\bz)\ts.
\]
On the other hand,
it follows from the choice of \. $\fa,\fb,\fc,\fd$ \. that
\begin{alignat*}{4}
 	s_{{\la}/{\mu}}(\bz) \ &= \  	\sum_{T \in \aL} \. \fa(T), \qquad \quad \ \
 	s_{{\nu}/{\rho}}(\bz) \ &= \  	\sum_{T \in \aL} \. \fb(T),\\
 	s_{{\la}/{\mu} \wedge \nu/\rho}(\bz) \ &= \  	\sum_{T \in \aL} \. \fc(T), \qquad
 	s_{{\la}/{\mu} \vee \nu/\rho}(\bz) \ &= \  	\sum_{T \in \aL} \. \fd(T).
\end{alignat*}
We conclude that  \. $\fa,\fb,\fc,\fd$ \.  satisfy \eqref{eq:AD-eq}.
It then follows from Theorem~\ref{thm:AD-chains} that
there exist a subset \. $A \subseteq (\la \vee {\nu})/({\mu} \wedge {\rho})$, the
corresponding sublattices \. $\cL_1=(\aL_1, \vee,\wedge)$, \. $\cL_2=(\aL_2, \vee,\wedge)$,
given by
\[
\cL_1 \ := \  \bigotimes_{(i,j) \in A} \cC_{(i,j)}, \qquad
\cL_2 \ := \  \bigotimes_{(i,j) \notin A} \cC_{(i,j)},
\]
 functions
\. $\ff_1, \fg_1: \aL_1 \to \rr_{\geq 0}$\.,   	\. $\ff_2, \fg_2: \aL_2 \to \rr_{\geq 0}$ \. and positive constants \. $\ala, \beb, \gac, \ded>0$, such that \. $\ala\ts\beb=\gac\ts\ded$ \. and
\begin{align}
\label{eq:umi-1}	\varphi^{\la/\mu}(x_1,x_2) \ &= \ \ala \.  {\color{blue}\ff_1}(x_1) \. {\color{blue}\ff_2}(x_2)\ts,\\
\label{eq:umi-2}	\varphi^{\nu/\rho}(x_1,x_2) \ &= \  \beb \. {\color{red}\fg_1}(x_1) \. {\color{red}\fg_2}(x_2)\ts,\\
\label{eq:umi-3}	\varphi^{\la/\mu \wedge \nu/\rho}(x_1,x_2) \ &= \  \gac \. {\color{blue}\ff_1}(x_1) \. {\color{red}\fg_2}(x_2)\ts,\\
\label{eq:umi-4}	\varphi^{\la/\mu \vee \nu/\rho}(x_1,x_2) \ &= \  \ded \. {\color{red}\fg_1}(x_1) \. {\color{blue}\ff_2}(x_2)\ts,
\end{align}
for all \.$(x_1,x_2) \in \aL_1 \times \aL_2$\..  These equality conditions are
quite far from the desired two possibilities in \eqref{eq:LPP-eq-comb}.  For the
rest of this section we give a detailed analysis bridging the gap.

\smallskip

\subsection{Analysis of equality cases} \label{ss:LPP-under}
The following lemmas that will be used repeatedly throughout the proof.

\smallskip

\begin{lemma}\label{lem:LPP-adjacent}
	Let $(i,j)$ and $(i',j')$ be adjacent squares such that
	\begin{alignat*}{3}
	  &\text{ $(i,j), (i',j') \. \in \. \la/\mu$  \qquad } &&\text{or} \qquad  &&\text{$(i,j), (i',j') \. \in \.  \nu/\rho$ \qquad or }\\
	  &\text{ $(i,j), (i',j') \. \in \.  \la/\mu \vee \nu/\rho$}  \qquad && \text{or} \qquad      &&(i,j), (i',j') \. \in \.  \la/\mu \wedge \nu/\rho\ts.
	\end{alignat*}	
	Then:
	\[   \text{either \quad $(i,j), (i',j') \. \in \.  A$  \quad or \quad  $(i,j), (i',j') \. \notin \. A$}\ts.     \]
\end{lemma}

\begin{proof}
	Without loss of generality, we assume that \. $(i',j')=(i,j+1)$, as the  case \. $(i',j')=(i+1,j)$ \. is analogous.
	Similarly, we assume that both \ts $(i,j)$ \ts and \ts $(i,j+1)$ \ts are in \ts $\la/\mu$,
as the proof of the other cases are also analogous.
	Suppose the claim in the lemma is false.  Without loss of generality,
we assume that \. $(i,j)\in A$, \ts $(i,j+1)\notin A$, as the reverse case is analogous.
	
	Let \. $T:({\la} \vee {\nu})/({\mu} \wedge {\rho}) \to [N] \cup \{-\infty,\infty\}$ \.  be the function given by
	\begin{align}\label{eq:T-1}
		T(k,\ell) \ := \
		\begin{cases}
			k &\text{ if } \ (k,\ell) \in \la/\mu\ts,\\
			\infty  & \text{ if } \ (k,\ell) \in (\la \vee \nu)/\la\ts,\\
			-\infty & \text{ if } \ (k,\ell) \in \mu\ts.
		\end{cases}
	\end{align}	Note that $T$ corresponds to a SSYT of shape $\la/\mu$, and hence $\varphi^{\la/\mu}(T) >0$ by construction.

		Let \. $T':({\la} \vee {\nu})/({\mu} \wedge {\rho}) \to [N] \cup \{-\infty,\infty\}$ \.  be the function given by
	\begin{align}\label{eq:T'-1}
		T'(k,\ell) \ := \
		\begin{cases}
			k+1 &\text{ if } \ (k,\ell) \in \la/\mu\ts,\\
			\infty & \text{ if } \ (k,\ell) \in (\la \vee \nu)/\la\ts,\\
			-\infty & \text{ if } \ (k,\ell) \in \mu\ts,
		\end{cases}
	\end{align}
	where  the image of $T$ is contained in $[N] \cup \{0,\infty\}$ because $k+1 \leq \ell(\lambda)+1 \leq N$.
	Note  that \ts $\varphi^{\la/\mu}(T) >0$, by construction.
	Finally, let
			 \. $T'':({\la} \vee {\nu})/({\mu} \wedge {\rho}) \to [N] \cup \{-\infty,\infty\}$ \.  be the function given by
	\begin{align*}
		T''(k,\ell) \ := \
		\begin{cases}
			k+1 &\text{ if } \ (k,\ell) \in (\la/\mu) \cap A\ts,\\
			k &\text{ if } \ (k,\ell) \in (\la/\mu)\ts, \. (k,\ell) \notin A\ts,\\
			\infty & \text{ if } \ (k,\ell) \in (\la \vee \nu)/\la\ts,\\
			-\infty  & \text{ if } \ (k,\ell) \in \mu\ts.
		\end{cases}
	\end{align*}
	On the one hand, since \. $\varphi^{\la/\mu}(T), \varphi^{\la/\mu}(T')>0$,
it follows from the decomposition of \ts $\varphi^{\la/\mu}$ \ts
in \eqref{eq:umi-1} that \. $\varphi^{\la/\mu}(T'')> 0$.  On the other hand, by the definition of $T''$, we have
	\[  T''(i,j) \ = \ i+1  \ > \ i \ = \  T''(i,j+1)\ts.   \]
This inequality gives \. $\varphi^{\la/\mu}(T'')= 0$ \. by the definition of~$\varphi^{\la/\mu}$, a contradiction.
\end{proof}

\smallskip

\begin{lemma}\label{lem:LPP-adjacent-2}
	Suppose that \. $(i,j) \in A$.
	Then:
$$
(i,j) \in\la/\mu  \ \. \Leftrightarrow \ \. (i,j) \in  \la/\mu \ts \wedge \ts\nu/\rho\., \qquad
(i,j) \in \nu/\rho  \ \.  \Leftrightarrow \ \. (i,j)\in   \la/\mu\ts \vee \ts\nu/\rho\ts.
$$
Similarly, suppose that \. $(i,j) \notin A$.  Then:
$$
(i,j) \in  \la/\mu  \ \. \Leftrightarrow \ \.  (i,j) \in  \la/\mu \ts \vee \ts \nu/\rho\., \qquad
(i,j) \in \nu/\rho  \ \. \Leftrightarrow \ \. (i,j) \in  \la/\mu \ts \wedge \ts \nu/\rho\ts.
$$
\end{lemma}


\begin{proof}
Without loss of generality, we can assume that  \. $(i,j) \in A$.
We will only show the proof of  the implication
\[  (i,j) \in  \la/\mu  \ \.   \Rightarrow \ \.   (i,j) \in \la/\mu  \ts\wedge \ts \nu/\rho\ts,
\]
as the proof of the other cases are analogous.

 Let \. $T:({\la} \vee {\nu})/({\mu} \wedge {\rho}) \to [N] \cup \{-\infty,\infty\}$ \. 
 be the function in \eqref{eq:T-1}. Note that  we have
\[   \varphi^{\la/\mu}(T) \ > \ 0 \quad \text{ and } \quad  T(i,j) \ = \  i  \in \{1,\ldots N\}. \]
Since \. $(i,j) \in A$, it follows from \eqref{eq:umi-1} and \eqref{eq:umi-3} that
that there exists \. $S:({\la} \vee {\nu})/({\mu} \wedge {\rho}) \to [N] \cup \{-\infty,\infty\}$ \.
such that
\[   \varphi^{\la/\mu \wedge \nu/\rho}(S) \, > \, 0 \quad \text{ and } \quad   S(i,j) \, = \,  i \. \in \. \{1,\ldots N\}. \]
It then follows from the definition of \. $\varphi^{\la/\mu \ts\wedge \ts\nu/\rho}$, that \.
$ (i,j) \in \la/\mu  \ts \wedge \ts \nu/\rho$, as desired.
\end{proof}

\smallskip

\begin{lemma}\label{lem:R1R2}
For any $i\geq 1$,  at least one of the following conditions hold:
\begin{align}
\tag{R1} \label{eq:R1}	&\la_i \.\leq \. \nu_i \quad \text{ and } \quad \mu_i \. \leq \. \rho_i\., \quad \ \underline{\text{or}}\\
\tag{R2} \label{eq:R2}	&\nu_i \. \leq \. \la_i \quad \text{ and } \quad \rho_i \. \leq \. \mu_i\..
\end{align}
\end{lemma}

\begin{proof}
	Suppose to the contrary, that the claim is false.
	By the symmetry, without loss of generality, we can assume that
	\. $\la_i < \nu_i$ \. and \. $\rho_i < \. \mu_i\ts$.
	This implies that \. $\rho_i < \nu_i$.
	It then follows that
	\. $(i,\nu_i)\in \nu/\rho$ \. and  \. $(i,\nu_i) \notin \la/\mu \ts\wedge \ts \nu/\rho$.
	Then Lemma~\ref{lem:LPP-adjacent-2} implies that \. $(i,\nu_i)\in A$.
	
Similarly,  it follows that
	\. $(i,\rho_i+1)\in\nu/\rho$ \. and  \. $(i,\rho_i+1)\notin \la/\mu \ts \vee \ts \nu/\rho$.
	Then Lemma~\ref{lem:LPP-adjacent-2} implies that   	\. $(i,\rho_i+1)\notin A$.
	On the other hand, we have
	\. $(i,\rho_i+1), (i,\rho_i+2),\ldots, (i,\nu_i)\in \nu/\rho$,
	so it follows from Lemma~\ref{lem:LPP-adjacent} that they all are in $A$, or they all are not in~$A$.
			This gives us a contradiction.
\end{proof}

\smallskip

\begin{lemma}\label{lem:R3}
	If \. $i$ \. satisfies \ts \eqref{eq:R1} \ts but not \ts \eqref{eq:R2}, then
\begin{equation*}
\left\{\,
\aligned
			 &\text{every square in $i$-th row of \,  $\la/\mu$    \, is in \,  $A$\ts,}   \\
			 &\text{every square in $i$-th row of   \,  $\nu/\rho$   \, is in \,   $A$\ts.  }
\endaligned\right.
\end{equation*}
Similarly, if \. $i$ \. satisfies \ts \eqref{eq:R2} \ts but not \ts \eqref{eq:R1}, then
\begin{equation*}
\left\{\,
\aligned	&\text{every square in $i$-th row of \,  $\la/\mu$  \, is not in  \, $A$\ts,  }  \\
	&\text{every square in $i$-th row of   \,  $\nu/\rho$   \ is not in \, $A$\ts.  }
\endaligned\right.
\end{equation*}
\end{lemma}

\begin{proof}
	We prove only the first part, as the proof of the second part is analogous.
	If \. $i$ \. is as  in the first part of the lemma, we have either \. $\la_i < \nu_i$ \.
    or  \. $\mu_i < \rho_i\ts$.
	We will further assume that \. $\la_i < \nu_i\ts,$ as the proof for the other case is analogous.
	
	Let us first show that every square in the $i$-th row of   \. $\nu/\rho$ \. is in~$A$.
	The statement is vacuously true if \. $\rho_i=\nu_i\ts,$ so assume
	that \. $\rho_i<\nu_i\ts.$  This implies that \. $(i,\nu_i)\in \nu/\rho$.
	Since  \. $\la_i < \nu_i\ts,$ we have 	\. $(i,\nu_i)\notin \la/\mu \ts \wedge \ts \nu/\rho$.
	Lemma~\ref{lem:LPP-adjacent-2} then implies that \. $(i,\nu_i)\in A$.
	Thus, by Lemma~\ref{lem:LPP-adjacent},  every square in the $i$-th row of \ts
$\nu/\rho$ \ts is in~$A$.

	Next, let us show that every square in $i$-th row of  \. $\la/\mu$  \. is in~$A$.
	The statement is vacuously true if \. $\mu_i=\lambda_i\ts,$ so assume that \. $\mu_i <\la_i\ts.$
	There are now two cases.  First, suppose that \. $\rho_i >\mu_i\ts$.
	Then  \. $(i,\mu_i+1)\in\la/\mu$, but \. $(i,\mu_i+1)\notin \la/\mu \ts \vee \ts \nu/\rho$.
	Lemma~\ref{lem:LPP-adjacent-2} implies that  \. $(i,\mu_i+1)\in A$.
	Lemma~\ref{lem:LPP-adjacent} then  implies that every square in $i$-th row of  \. $\la/\mu$  \.
    is in~$A$.
	
    Second suppose that \. $\rho_i=\mu_i$.
	In this case we then have every square in \. $\la/\mu$ \. is also in \. $\nu/\rho$.
		Since we have previously shown that every square in the $i$-th row of \. $\nu/\rho$  is in~$A$,
		it then follows that every square in the $i$-th row of \. $\la/\mu$ \. is also in~$A$.
		This completes the proof.
\end{proof}

\smallskip

\subsection{Proof of Theorem~\ref{thm:LPP-eq-eval}}\label{ss:LPP-eval-proof}
The proof of the \ts $\Leftarrow$ \ts direction is straightforward.
For the \ts $\Rightarrow$ \ts direction, suppose to the contrary
that the claim is false.
By the symmetry, and by making a careful choice the rows,
we can assume that there exists \. $i<j$,  such that: \. $i$ \ts
satisfies \eqref{eq:R1} but not \eqref{eq:R2}, \. $j$ \ts satisfies
\eqref{eq:R2} but not \eqref{eq:R1}, and \. $i+1,\ldots, j-1$ \. satisfy both \eqref{eq:R1} and \eqref{eq:R2}.

Because the skew shape \. $(\la \vee \nu)/(\mu \wedge \rho)$ \. is connected, it follows that \.
$\min\{\mu_i,\rho_i\} \ts < \ts \max\{\la_{i+1},\nu_{i+1}\}$.
Because \ts $i$ \ts satisfies \eqref{eq:R1} and \ts $i+1$ \ts satisfies \eqref{eq:R2}, we then have
\. $\mu_i < \la_{i+1}$.
This implies that \. $(i,\mu_i+1)$ \. and \. $(i+1,\mu_{i}+1)$ \. are adjacent squares in $\la/\mu$.

Since every square in the $i$-th row of $\la/\mu$ \ts is in~$A$ by Lemma~\ref{lem:R3},
it follows from Lemma~\ref{lem:LPP-adjacent}  that \. $(i+1,\mu_{i}+1)$ \. is a square
in the $(i+1)$-th row of  \ts  $\la/\mu$  \ts that is also in~$A$.
Iterating this argument, we conclude that  \. $(j,\mu_{j-1}+1)$ \. is a square in the $j$-th row of  $\la/\mu$ is also in~$A$.
On the other hand, it follows from Lemma~\ref{lem:R3} that \emph{every} \ts square in the $j$-th row of \ts
$\la/\mu$ \ts is not in~$A$, a contradiction.  This completes the proof.
\qed

\medskip



\section{Implications of LPP equality conditions} \label{sec:Oko-Bjo-proof}

In this section, we prove equality conditions for both the Okounkov
and the generalized Bj\"orner inequality.  Both proofs use
equality conditions for the LPP inequality (Theorem~\ref{thm:LPP-eq}).

\subsection{Proof of Theorem~\ref{thm:Oko-eq}}\label{sec:proof-Oko-eq}
The following result is an extension of Theorem~\ref{thm:Oko-eq}
to a large set of skew shapes, but with the same equality conditions.
\smallskip

\begin{thm}\label{thm:Oko-eq-gen}
Let \ts $\lambda/\mu$, \ts $\nu/\rho$ \ts be skew shapes such that \ts $\ell(\la), \ell(\nu)\le \ell$.
Assume that for all \ts $i \in [\ell-1]$, we have:
\begin{equation}\label{eq:Oko-Irr}
  \mu_i < \lambda_{i+1} \quad \text{ or } \quad \rho_i < \nu_{i+1}\..
\end{equation}
Then \eqref{eq:Oko} is an equality:
\begin{equation}\label{eq:Oko-eq-again}
	  s_{\lambda/\mu} \. \cdot \. s_{\nu/\rho} \ = \ s_{\lceil (\la+\nu)/2\rceil /\lceil(\mu+\rho)/2\rceil}
\. \cdot \.
s_{\lfloor (\la+\nu)/2\rfloor/\lfloor (\mu+\rho)/2\rfloor}
\end{equation}
	\underline{\em if and only if}
\begin{align}\label{eq:Oko-eq-comb-again}
\la \ts - \ts \nu \, = \, a\ts (\ve_1,\ve_2,\ldots) \. + \. (b,b,\ldots), \quad
\mu \ts - \ts \rho \, = \, a\ts (\ve'_1,\ve'_2,\ldots) \. + \. (b,b,\ldots),
%
\end{align}
for some \. $a \in \{\pm 1\}$, \. $b \in \zz$, \. and \.
$\ve_1,\ve_1',\ve_2,\ve_2',\ldots \in \{0,1\}$.
%
\end{thm}

\smallskip

 Note that the condition \eqref{eq:Oko-Irr} does not in fact make Theorem~\ref{thm:Oko-eq-gen}  less general. Indeed,
	given any $\lambda,\mu,\nu,\rho$, we can partition $[\ell]$ into disjoint intervals $[a_1, a_2)$, $[a_2,a_3)$, $\ldots$, $[a_{k-1},a_k)$ for some
	$1=a_1<\cdots < a_k=\ell+1$ , such that restricting to the rows indexed by each interval $[a_i, a_{i+1})$
	satisfies \eqref{eq:Oko-Irr}.

\smallskip

\begin{proof}[Proof of Theorem~\ref{thm:Oko-eq-gen}]
  The direction \ts $\Leftarrow$ \ts is straightforward, so we present only
  the proof of the $\Longrightarrow$ direction.  We start with the following
  observations.
Let \ts $\lambda,\mu,\nu,\rho \in \nn^\ell$ \ts be partitions satisfying \ts
$\mu \subseteq \lambda$, \ts $\rho \subseteq \nu$.
Denote \. $\lambda^\ast:=\lambda+1^\ell$ \. and
  \. $\mu^\ast:=\mu+1^\ell$\..
  Then we have \. $s_{\lambda^\ast/\mu^\ast} \. = \. s_{\lambda/\mu}$ \. and
  \begin{equation*}
  \begin{split}
s_{\lceil (\lambda^\ast+\nu)/2\rceil/\lceil (\mu^\ast+\rho)/2\rceil}
 \, & = \, s_{\lfloor (\lambda+\nu)/2\rfloor/  \lfloor (\mu+\rho)/2\rfloor}, \quad \
  s_{\left \lfloor (\lambda^\ast+\nu)/2 \right\rfloor /  \left \lfloor (\mu^\ast+\rho)/2 \right\rfloor}
  \, = \, s_{\left \lceil (\lambda+\nu)/2 \right\rceil/  \left \lceil (\mu+\rho)/2 \right\rceil},
 \end{split}
  \end{equation*}
by the identities \. $\lceil (x+1)/2\rceil=\lfloor x/2\rfloor+1$ \. and
\. $\lfloor (x+1)/2\rfloor =\lceil x/2\rceil$\..

Let \. $\lambda,\mu,\nu,\rho \in \zz^\ell$ \. be as given in the theorem.
From the equations above, the equality \eqref{eq:Oko-eq-again} is preserved under the
substitution \. $\lambda \gets \lambda+c\cdot 1^\ell$, $\mu\gets \mu +c\cdot 1^\ell$,
$\nu\gets \nu +d\cdot 1^\ell$, $\rho\gets \rho+d\cdot 1^\ell$, for all \ts $c,d \geq 0$.
In particular, we can assume without loss of generality, that all parts of \.
$\lambda,\mu,\nu,\rho$ \. are arbitrarily large.

Suppose to the contrary that partitions \. $\lambda,\mu,\nu,\rho$ \.
give an equality in \eqref{eq:Oko-eq-again}, but
do not satisfy \eqref{eq:Oko-eq-comb-again}. Then, vector \.
$(\lambda,\mu)-(\nu,\rho)\in \zz^{2\ell}$ \. contains at least three distinct entries.
It then follows that we can choose \ts $c \in \zz$ \ts such that the vector
$({\lambda}^\circ, {\mu}^\circ) - (\nu, \rho)$ \. contains at least one strictly
positive and at least one strictly negative entry, where \.
${\lambda}^\circ := \lambda + c \cdot 1^\ell$ \. and \. ${\mu}^\circ := \mu + c\cdot 1^\ell$.
In particular, this implies that
partitions \. ${\lambda}^\circ, {\mu}^\circ, \nu, \rho$ \. do not satisfy \eqref{eq:LPP-eq-comb}.

Now, it follows from the LPP inequality (Theorem~\ref{thm:LPP}), that
\begin{equation}\label{eq:Oko-1}
s_{{\lambda}/{\mu}} \.\cdot \. s_{\nu/\rho}  \ = \ s_{{\lambda}^\circ/{\mu}^\circ} \.\cdot \. s_{\nu,\rho}  \  \leqslant_{s} \
s_{({\lambda}^\circ/ {\mu}^\circ) \vee (\nu/\rho)} \.\cdot \.
s_{({\lambda}^\circ/{\mu}^\circ) \wedge (\nu/\rho)}\..
\end{equation}
Let  \. $\al/\be:= ({\la}^\circ/ {\mu}^\circ) \vee (\nu/\rho)$ \.
 and  \. $\ga/\de:= ({\la}^\circ, {\mu}^\circ) \wedge (\nu,\rho)$ \..
 It follows from \eqref{eq:Oko}
 that
 \begin{equation}\label{eq:Oko-2}
 \begin{split}
s_{\al/\be} \. \cdot \. s_{\ga/\de} \ &\leqslant_s \
s_{\left \lceil (\al+\ga)/2 \right\rceil /  \left \lceil (\be+\de)/2 \right\rceil}
 \. \cdot \.
s_{\left \lfloor (\al+\ga)/2 \right\rfloor /\left \lfloor (\be+\de)/2 \right\rfloor} \\
&= \
s_{\left \lceil ({\la}^\circ+\nu)/2 \right\rceil /  \left \lceil ({\mu}^\circ+\rho)/2 \right\rceil}
\. \cdot \.
s_{\left \lfloor ({\la}^\circ+\nu)/2 \right\rfloor /\left \lfloor ({\mu}^\circ+\rho)/2 \right\rfloor} \\
& = \
s_{\left \lfloor ({\lambda}+\nu)/2 \right\rfloor /  \left \lfloor ({\mu}+\rho)/2 \right\rfloor}
 \. \cdot \. s_{\left \lceil ({\lambda}+\nu)/2 \right\rceil /
	\left \lceil {\mu}+\rho)/2 \right\rceil}.
\end{split}
 \end{equation}
Since \. $\lambda,\mu,\nu,\rho$ \. gives an equality in \eqref{eq:Oko},
we conclude both \eqref{eq:Oko-1} and \eqref{eq:Oko-2} are equalities.

Since partitions \. ${\la}^\circ, {\mu}^\circ, \nu, \rho$ \. do not satisfy \eqref{eq:LPP-eq-comb},
it then follows from Theorem~\ref{thm:LPP-eq}
that skew shape \. $({\la}^\circ \vee \nu) / ({\mu}^\circ \wedge \rho)$ \. is not connected.
This implies that there exists \. $i \in [\ell-1]$ such that
\begin{equation*}
	\min\{{\mu}^\circ_i, \rho_i\}  \ \geq \  \max\{{\lambda^\circ_{i+1}}, \nu_{i+1} \}.
\end{equation*}
This in turn implies that
\[
  {\mu}^\circ_i \geq  {\la}^\circ_{i+1} \quad \text{ and } \quad \rho_i \geq  \nu_{i+1}\.,
      \]
which contradicts \eqref{eq:Oko-Irr}.  This completes the proof.
\end{proof}

\smallskip

\subsection{Proof of Theorem~\ref{thm:Bjo-eq}}\label{sec:proof-Bjo-eq}
The \. $\Leftarrow$ \. direction is straightforward.  Indeed, since \eqref{eq:LPP}
is an equality, for the principal evaluation \.
$(z_1,z_2,z_3,\ldots) \gets (1,q,q^2,\ldots)$ \.  gives:
\begin{equation*}
	\ s_{\lambda/\mu}(1,q,q^2,\ldots) \. \cdot \. s_{\nu/\rho}(1,q,q^2,\ldots)
	\ = \ s_{\la/\mu \. \vee  \. \nu/\rho}(1,q,q^2,\ldots) \. \cdot \.  s_{\la/\mu \. \wedge \. \nu/\rho}(1,q,q^2,\ldots),
\end{equation*}
for all $q \in [0,1)$.
Apply \eqref{eq:P-part} for each of the four principal evaluations.
Multiply both sides of the equation above by \. $(1-q)^m$ \. and take the limit \ts  $q \to 1^-$,
where \. $m:=|\lambda/\mu| +|\nu/\rho|$.  This gives
\begin{equation*}
	\ff(\la/\mu) \. \cdot \.  \ff(\nu/\rho) \ = \ \ff(\la/\mu \ts \wedge \ts \nu/\rho) \. \cdot \. \ff(\la/\mu \ts \vee \ts \nu/\rho),
\end{equation*}
as desired.

The \. $\Rightarrow$ \. direction is proved in a similar manner, but in reverse.
Write \eqref{eq:LPP} as follows:
\begin{equation}\label{eq:Bjo-LPP}
	s_{\la/\mu \ts \wedge  \ts \nu/\rho} \. \cdot \.  s_{\la/\mu \ts \vee \ts \nu/\rho} \ - \  	\ s_{\lambda/\mu} \. \cdot \. s_{\nu/\rho}  \ = \   \sum_{\tau \. \vdash \. m} \. a_{\tau}  \. s_{\tau}\.,
\end{equation}
where all \. $a_{\tau}\geq 0$.
As above, take principal evaluations and apply \eqref{eq:P-part} to each Schur function,
multiply both sides by \. $(1-q)^m$ \. and take the limit \ts  $q \to 1^-$.  We obtain:
\[
\ff(\la/\mu \ts \vee \ts \nu/\rho) \. \ff(\la/\mu \ts \wedge \ts  \nu/\rho)
\, - \, 	\ff(\la/\mu) \. \cdot \.  \ff(\nu/\rho) \, = \,
\sum_{\nu \. \vdash \. m}  \. a_{\tau}  \. \ff({\tau})\ts.
\]
By assumption, the LHS is equal to zero, so we have \. $a_{\tau}=0$ \. for all \. $\tau\vdash m$.
Thus, both sides of \eqref{eq:Bjo-LPP} are also equal to zero, as desired.
\qed

\smallskip
{
\begin{rem}\label{rem:LPP-Bjo}
In fact, the monomial positive version of \eqref{eq:LPP}
would also suffice for the argument above.  As we mentioned
earlier, monomial positivity is much easier to prove, see \cite{LP07,CP-multi}.
\end{rem}}

\medskip

\section{Proof of Theorem~\ref{thm:Fis-eq}}\label{sec:proof-Fis}

\subsection{RLS inequality}\label{ss:Fis-RLS}
We need the following curious generalization of \eqref{eq:AD}.
For a lattice \ts $\cL=(\aL,\vee,\wedge)$ \ts, a function $f:\aL \to \rr$, and a subset $S \subseteq \aL$, we write
\[  f(S) \  := \  \sum_{x \in S} f(x). \]

\smallskip

\begin{thm}[{\rm \defn{RLS inequality}\ts}{}]\label{thm:RLS}\label{t:RLS}
Let \ts $\cL=(\aL,\vee,\wedge)$ \ts be a product of $M$-chains,
let \ts $C,D\subseteq \aL$ \ts be fixed subsets, and let \ts
$\Omega(C,D)$ \ts denote the set of pairs \ts $(A,B)\in \aL^2$,
such that \ts $A\vee B=C$, \ts $A\wedge B=D$. Finally, let \.
$\fa,\fb,\fc,\fd: \aL \to \rr_{\geq 0}$ \. be four
functions satisfying \eqref{eq:AD-cd}.  Then:
	\begin{equation}\label{eq:RLS}\tag{\defn{\em RLS}}
\sum_{(A,B)\in \Omega(C,D)} \, \fa(A) \. \fb(B) \ \leq \
		\sum_{(A,B)\in \Omega(C,D)} \, \fc(A) \. \fd(B)\ts.
	\end{equation}
\end{thm}

\smallskip

This inequality was proved by the authors in \cite[Claim~6.3]{CP-multi}.
For the Boolean lattice (the case $M=2$), the result was proved by Lov\'asz--Saks
\cite[Thm~8]{LS06}, generalizing the case \ts $C=\cL$ \ts and \ts $D=\emp$
\ts by  Reuter \cite{Reu87}.  In fact, these were independent developments;
the connection was discovered in \cite{CCPS26}.

\smallskip

\subsection{Proof of the \. $\Rightarrow$ \. direction}
Let \. $n:=|X|$ \. and \. $N:=2n$.
Let $Y$ be the subset of $[N]$ given by
\[
    Y \ := \ \big\{1,2,\ldots, |X-A-B|\big\} \, \cup \,   \big\{n+1,n+2,\ldots, n+|X-C-D|\big\}.
\]
Notice that the two sets in the right hand side of the equation above are disjoint, and
\[ |Y| \, = \,  2\ts |X| \. - \. |A|\. - \.|B|\. - \.|C|\. - \.|D|\ts.
\]
Let \ts $\cL=(\aL,\vee,\wedge)$ \ts be a direct product of $|X|$ copies of
$(N+2)$-chains: \ts $\aL:=([N]\cup \{0,\infty\})^{X}$.  In other words, every element
\ts $T\in \aL$ \ts corresponds to  a function \ts $T: X \to [N]\cup \{0,\infty\}$.

Take variables \. $\bz=(\zq,z_1,\ldots, z_N, \zt)$.
For a function \. $T:X \to [N]\cup \{0,\infty\}$\., we write
\[ \bz^T  \ := \   \zq^{m_{0}(T)} \. z_1^{m_1(T)} \. \cdots \. z_N^{m_N(T)} \.  \zt^{m_{\infty}(T)}\.,   \]
where \ts $m_i(T):= |T^{-1}(i)|$. 
Let \ts $\aK \subseteq \aL \times \aL$ \ts be a subset given by
\[ 
    \aK \ := \ \big\{ \ts (S,T) \in \aL \times \aL  \, : \,  \bz^S \. \bz^T  \, 
           = \,  \zq^{|A|+|B|} \ts \zt^{|C|+|D|} \ts \bz^Y  \ts \big \}.   
\]

Let $W \subseteq X$ be a subset satisfying $X-W \subseteq  A \cup B \cup C \cup D$.
We denote by \. $\phi^{W}: \aL \to \rr_{\geq 0}$ \.
 the indicator function of the set of \. $T:X \to [N]\cup \{0,\infty\}$     \. satisfying
 \begin{alignat*}{2}
 	& T(x) \. < \. T(y) \quad &&\text{ for all }  \quad x \prec y, \ \text{ where \ $x,y\in W$,}\\
 	&T(x) \. = \.  0 \quad &&\text{ for all } \quad x \in (A\cup B) -W, \  \text{and} \\
 	&T(x) \. = \.  \infty \quad &&\text{ for all } \quad x \in (C\cup D) -W.
 \end{alignat*}
 Note that every such  \. $T:X \to  \{0,\ldots,N\}$   \. corresponds to a
 linear extension \ts $\lE\in \Ec(\cPo|_W)$, and that several such functions
 may correspond to the same~$\ts\lE$.

   We define functions \. $\fa,\fb,\fc,\fd:\aL \to \rr_{\geq }0$ \.
   by
   \begin{align*}
   	\fa \, = \,  \phi^{X-A-C}, \qquad  \ \ \fb \, := \, \phi^{X-B-D}, \quad \ \
   	\fc \, := \, \phi^{X-C-D},
    \qquad \fd \ := \   \phi^{X-A-B}.
   \end{align*}
It is straightforward to check that these four functions satisfy \eqref{eq:AD-cd}.
We have by construction:
   \begin{align*}
   	 \ff(X-A-C) \. \cdot \.\ff(X-B-D) \ &= \  \frac{1}{|Y|!}\. \sum_{(S,T) \in \aK} \. \fa(S) \. \fb(T),\\
     	 \ff(X-C-D) \.\cdot \. \ff(X-A-B) \ &= \  \frac{1}{|Y|!} \. \sum_{(S,T) \in \aK} \. \fc(S)\. \fd(T).
   \end{align*}
  It then follows from \eqref{eq:Fis-eq} that
  \begin{equation}\label{eq:Fis-1}
  	\sum_{(S,T) \in \aK} \,  \fc(S) \ts \fd(T) \. - \. \fa(S) \ts \fb(T) \ = \  0.
  \end{equation}

  Let \ts $\aK'\subseteq \aK$ \ts be the subset of elements \ts $(U,V)$, such that \. $U(x) \geq  V(x)$ \.
   for all \ts $x \in X$.  For every \. $(U,V) \in \aK'$, we write
  \[
    \psi(U,V) \ := \   \sum_{\substack{(S,T) \in \aK \ \text{s.t.} \\
    S\vee T = U,  \. S \wedge T = V}}  \, \fc(S) \ts \fd(T) \. - \. \fa(S) \ts \fb(T)\ts.
  \]
   By Theorem~\ref{thm:RLS}, we have \. $\psi(U,V) \geq 0$. 
  On the other hand, we also have
  \[  	\sum_{(S,T) \in \aK} \, \fc(S) \ts \fd(T) \. - \. \fa(S) \ts \fb(T) \ = \     \sum_{U,V \in \aK'} \. \psi(U,V).
 \]
  It then follows from \eqref{eq:Fis-1} that
  \begin{equation}\label{eq:Fis-2}
     \psi(U,V) \ = \ 0  \quad \text{ for all } \ (U,V) \in \aK'.
  \end{equation}

  For the rest of the proof, take \. $(U, V) \in \aK'$ \.
  that satisfies the following properties:  \begin{alignat*}{2}
  	& U(x) \  = \  \infty \qquad && \text{for all} \ \quad  x \in C \cup D\ts,\\
  	& V(x) \  = \  0 \qquad && \text{for all} \ \quad  x \in A \cup B\ts,\\   	
  	& U(x) \in \{n+1,n+2,\ldots, n+|X-C-D|\} \qquad && \text{for all} \ \quad  x \in X-C - D\ts,\\
  	& V(x) \in \{1,2,\ldots, |X-A-B|\} \qquad && \text{for all} \ \quad   x \in X-A - B\ts.
  \end{alignat*}
  Here the function \. $U:X\to [N] \cup \{0,\infty\}$ \. can be constructed
  by selecting a linear extension \ts $\lE\in \Ec(\cPo|_{X-C-D})$,
  and then specifying the value of \. $U(x)$ \. for all \. $x \in X-C-D$ \.
  according to the order established by~$\lE$.
The function \. $V:X\to [N] \cup \{0,\infty\}$ \. can be constructed similarly.
Note that
  \begin{equation}\label{eq:Fis-3}
  	U(x) \, > \, n \, \geq \, V(y) \quad \text{ for all } \ \  x,y \in X\ts.
  \end{equation}

Observe  that
  \. $\fc(U) \ts \fd(V) =1$ \. by construction.
Therefore, by  \eqref{eq:Fis-2},  there exists \. $(S,T)\in \aK$ \. satisfying
\. $S\vee T = U$ \. and   \. $S \wedge T = V$, and  such that \. $\fa(S)\ts\fb(T)=1$.
   Let \. $X_1, X_2$ \. be subsets of $X$ given by
   \begin{align*}
   	X_1 \ &:= \  \{ \. x \in X \. : \. S(x)=U(x) \ \text{ and } \ T(x)=V(x)    \. \}\.,\\
      	X_2 \ &:= \  \{ \. x \in X \. : \. S(x)=V(x) \ \text{ and } \ T(x)=U(x)    \. \}\..
   \end{align*}

  \smallskip

  \begin{lemma}\label{lem:Fis-adjacent}
  	Let \ts $x,y\in X$ \ts be comparable in $\cPo$. Then either \ts $x,y \in X_1$ \ts or \ts  $x,y \in X_2$.
  \end{lemma}

   \begin{proof}
   	Without loss generality, assume that \ts $x \prec y$ \ts and that \ts $x \in X_1$.
   	By \eqref{eq:Fis-3}, we have:
   	\[  S(y) \,  \geq \, S(x) \, = \,  U(x) \, \geq  \, n+1\ts.   \]
   	 By \eqref{eq:Fis-3} again, this gives \ts $S(y)=U(y)$, which in turn implies that \ts $y \in X_1\ts$,
   	 as desired.
   \end{proof}

   \smallskip

   \begin{lemma}\label{lem:Fis-adjacent-2}
	$A \cup D \subseteq X_2$ \. and \.
$B \cup C \subseteq X_1\ts$.
   \end{lemma}

\begin{proof}
	We will only prove that \. $x \in A \. \Rightarrow \. x \in X_2$\., as the proof of the other cases are analogous.
Since \ts $x \in A$ \ts and \ts $\fa(S)=\phi^{X-A-C}(S)>0$,
we have $S(x) =0$. Thus, in particular,we have \ts $S(x)\leq n$.
	It then follows from \eqref{eq:Fis-3} that \. $S(x)=V(x)$.
	This implies that \ts $x \in X_2\ts$, as desired.
\end{proof}

  \smallskip

\nin
{\em Proof of connectivity conditions}.
For the part \. $\blue{A} \nlra \red{B}$ \. in \eqref{eq:it-Fis}, it follows from
Lemma~\ref{lem:Fis-adjacent-2} that \. $A\subseteq X_2$, and \. $B\subseteq X_1$.
Suppose to the contrary, that some elements of \ts $A$ \ts and \ts
$B$ \ts are connected in the comparability graph \ts $\Ga(\cPo)$.
     It then follows from Lemma~\ref{lem:Fis-adjacent} that these elements
     must both belong to $X_1$ or both belong to $X_2$, a contradiction.
     Other conditions are proved analogously. \qed

  \medskip

\subsection{Proof of the \. $\Leftarrow$ \. direction}
It suffices to prove that, for every connected component \ts $X'\subseteq X$ \ts
in the comparability graph \ts $\Ga(\cPo)$, we have:
\begin{equation}\label{eq:Fis-4}
	\ff(X'-{A}-{C}) \.\cdot \. \ff(X'-{B}-{D}) \ = \ \ff(X'-{C}-{D}) \.\cdot\. \ff(X'-{A}-{B}).
\end{equation}
From~\eqref{eq:it-Fis}, it follows that \ts $X'$ \ts intersects at most two subsets of \ts $A,B,C,D$.

First, suppose that \ts $X'$ \ts intersects none of the subsets \ts $A,B,C,D$.
Then we have:
\[  X'-{A}-{C} \, = \, X'-{B}-{D} \, = \, X'-{C}-{D} \, = \, X'-{A}-{B} \, = \, X'\ts,  \]
and \eqref{eq:Fis-4} follows immediately.
Second suppose that $X'$ intersects exactly one subset of \ts $A,B,C,D$, say~$A$.
Then we have:
\[  X'-{A}-{C} \, = \, X'-{A}-{B} \, =  \, X'-A\., \qquad \ X'-{B}-{D} \, = \, X'-{C}-{D} \, = \, X'\ts,  \]
and \eqref{eq:Fis-4} follows immediately.

Finally, suppose that $X'$ intersects exactly two subsets of $A,B,C,D$.
From  \eqref{eq:it-Fis}, it follows that either $X'$ intersects \ts $A$ \ts and \ts $D$,
or \ts $X'$ \ts intersects \ts $B$ \ts and \ts $C$.  Without loss of generality, assume the former possibility.
Then we have:
\begin{align*}
	X'-{A}-{B} \, = \,  X'-A-C \, = \,   X'-A\., \qquad X'-B-D \, = \,  X'-C-D \, = \,  X'-D\.,
\end{align*}
and \eqref{eq:Fis-4} follows immediately. This completes the proof. \qed

\medskip



\section{Equality conditions for the ADS inequality}\label{s:ADS}

\subsection{Finite support case}\label{ss:ADS-finite}
We now prove Theorem~\ref{thm:ADS-eq} in the case when all four functions
have \emph{finite support}.  Formally, fix \ts $\ell,w \geq 1$.
Denote by \ts $\cY^{\ell,w}$ \ts the set of partitions \ts $\lambda = (\la_1,\ldots,\la_\ell)$,
such that \ts $\lambda_1 \leq w$.  Observe that \ts $\cY^{\ell,w}$ \ts
forms a finite sublattice of~$\cY$.

From this point on, we assume that \.
$$\supp(\fa), \ \supp(\fb), \ \supp(\fc), \ \supp(\fd) \ \subseteq \ \cY^{\ell,w}.
$$
and that none of the functions \ts $\fa, \fb, \fc, \fd$ \ts is identically zero
(otherwise the theorem holds trivially).
Let \ts $\cL^{\ell,w} = (\aL,\vee,\wedge)$ \ts
be a direct product \ts $\cL^{\ell,w} \ts = \ts \cC^1 \times \ldots \times \cC^\ell$ \ts
of \ts $\ell$ \ts copies of chains on \ts $\{0,1\ldots, w\}$.
Observe that \ts $\cY^{\ell,w}$ \ts is a sublattice of \ts $\cL^{\ell,w}$.
Extend  the functions \ts $\fa, \fb, \fc, \fd$ \ts to \ts $\aL$ \ts as
$0$ outside of \ts $\cY^{\ell,w}$.
	
Let \ts $N \geq \ell$, and let \ts $\bz=(z_1,\ldots,z_N) \in \rr^N_{>0}$ \ts
be an arbitrary positive vector.  Define:
$$
	\fa'(\lambda) \, :=  \,  \fa(\lambda) \. s_{\lambda}(\bz)\., \quad
		\fb'(\lambda) \, :=  \,   \fb(\lambda) \. s_{\lambda}(\bz)\., \quad
	\fc'(\lambda) \, :=  \,  \fc(\lambda) \. s_{\lambda}(\bz)\., \quad
	\fd'(\lambda) \, :=  \,  \fd(\lambda) \. s_{\lambda}(\bz)\ts.
$$
It follows from \eqref{eq:LPP},
that four functions \. $\fa',\fb',\fc',\fd':\aL \to \rr_{\geq 0}$ \.
satisfy \eqref{eq:AD-cd}.
It then follows from the equality in \eqref{eq:ADS}, that \ts
$\fa',\fb',\fc',\fd'$ \ts satisfy \eqref{eq:AD-eq}.
By Theorem~\ref{thm:AD-chains}, there exist a subset \ts $A \subseteq [\ell]$, lattices
\[
\cL_1 \, = \, (\aL_1,\vee,\wedge) \ := \  \bigotimes_{i \in A} \cC^i\ts,
\qquad  \cL_2  \, = \, (\aL_1,\vee,\wedge) \ := \  \bigotimes_{j \notin A} \cC^j\ts,
\]
functions \. $\ff_1', \fg_1': \aL_1 \to \rr_{\geq 0}\ts$,   	
\. $\ff_2', \fg_2': \aL_2 \to \rr_{\geq 0}$ \ts, and constants \. $\ala, \beb, \gac, \ded>0$, \. such that \.
$\ala\ts\beb \, = \, \gac\ts \ded$ \. and
\begin{equation}\label{eq:ADS-four}
 \begin{aligned}
			\fa'(x_1,x_2) \ &= \ \ala \.  {\color{blue}\ff'_1}(x_1) \. {\color{blue}\ff'_2}(x_2), \qquad
			& \fb'(x_1,x_2) \ &= \  \beb \. {\color{red}\fg'_1}(x_1) \. {\color{red}\fg'_2}(x_2),\\
			\fc'(x_1,x_2) \ &= \  \gac \. {\color{blue}\ff'_1}(x_1) \. {\color{red}\fg'_2}(x_2), \qquad
			&\fd'(x_1,x_2) \ &= \  \ded \. {\color{red}\fg'_1}(x_1) \. {\color{blue}\ff'_2}(x_2),
\end{aligned}
	\end{equation}
for all \. $(x_1,x_2) \in \aL_1\times \aL_2\ts$.
Let us emphasize that, \emph{a priori}, the subset \ts $A$ \ts can
depend on the underlying choice of \ts $\bz = (z_1,\ldots,z_N)$.

Note that the index sets \ts $A$ \ts and \ts $\ov A := [\ell]-A$ \ts are not necessarily contiguous.
Thus, in order to preserve membership in \ts $\cY^{\ell,w}$ \ts
the isomorphism \ts $\cL \simeq \cL_1 \times \cL_2$ \ts maps
weakly decreasing coordinates in \ts $\cL$ \ts to weakly
decreasing coordinates in \ts $\cL_1\ts, \cL_2\ts.$
We write \ts $\la=(\laone,\latwo)$ \ts to emphasize the decomposition,
which is mapped onto the union of partitions \ts $\la=\laone\cup \latwo$.

Let $\lambda = (\laone, \latwo)$ and $\mu = (\muone, \mutwo)$ be
elements of  $\aL =\aL_1 \times \aL_2$.
It follows from \eqref{eq:ADS-four} that
\begin{align*}
	& \fa(\laone,\latwo) \. \fb(\muone,\mutwo) \.\cdot \. s_{(\laone,\ts\latwo)}(\bz) \.
	s_{(\muone,\ts\mutwo)}(\bz) \  =  \ \fa'(\laone,\latwo) \. \fb'(\muone,\mutwo) \\
    & =  \ \fc'(\laone,\mutwo) \. \fd'(\muone,\latwo) \ = \ \fc(\laone,\mutwo) \. \fd(\muone,\latwo) \. \cdot \. s_{(\laone,\ts\mutwo)}(\bz) \.
s_{(\muone,\ts\latwo)}(\bz).
\end{align*}
Since the choice of $\bz$ varies over all $\rr_{>0}^N\ts$,
this implies the polynomial identity
\begin{equation}\label{eq:stable}
\aligned
	 & \fa(\laone,\latwo) \. \fb(\muone,\mutwo) \. \cdot \. s_{(\laone,\ts\latwo)} \.
	s_{(\muone,\ts\mutwo)} \\
	 & \quad \ = \ \fc(\laone,\mutwo) \. \fd(\muone,\latwo) \.\cdot \. s_{(\laone,\ts\mutwo)} \.
	s_{(\muone,\ts\latwo)}\.,
\endaligned
\end{equation}
where the subset \. $A\subseteq [\ell]$ \. is globally determined.

\smallskip

\begin{lemma}\label{lem:rosewood}
Let \. $(\lambda^{(1)},\lambda^{(2)}) \in \supp(\fa)$ \. and \. $(\mu^{(1)},\mu^{(2)}) \in \supp(\fb)$.
Then we have:
\[ \text{ either } \ \lambda^{(1)} \ = \ \mu^{(1)} \quad \text{ or } \quad \lambda^{(2)} \ = \  \mu^{(2)}.  \]
\end{lemma}

\begin{proof}
 Taking the
leading terms of \eqref{eq:stable} w.r.t.\ the dominance order gives an equality of monomial symmetric
polynomials:
\begin{equation*}
	  \fa(\laone,\latwo) \. \fb(\muone,\mutwo)  \,  m_{(\laone+\muone,\ts\latwo+\mutwo)} \.
	 \, = \, \fc(\laone,\mutwo) \. \fd(\muone,\latwo) \,  m_{(\laone+\muone,\ts\latwo+\mutwo)}\.,
\end{equation*}
which implies that
\begin{equation}\label{eq:stable-3}
	 \fa(\laone,\latwo) \. \fb(\muone,\mutwo)
\ = \ \fc(\laone,\mutwo) \. \fd(\muone,\latwo).
\end{equation}
We thus have the equality of products of Schur polynomials
\begin{equation}\label{eq:stable-2}
s_{(\laone,\ts\latwo)} \. \cdot \. s_{(\muone,\ts\mutwo)}
	 \ = \  s_{(\laone,\ts\mutwo)} \.\cdot \.	s_{(\muone,\ts\latwo)}\..
\end{equation}
By Rajan's Theorem~\ref{t:Rajan}, this implies that either \. $\laone  = \muone$ \. or \.
$\latwo =   \mutwo$,  as desired.
\end{proof}

\smallskip

A priori, which of the two conditions in Lemma~\ref{lem:rosewood} holds could depend on the choice of $\lambda$ and $\mu$.
In the following lemma, we show that one of these cases must hold globally.

\smallskip

\begin{lemma}\label{lem:Rajan-2}
	One of the following holds:
	either
\begin{equation*}
\left[\,
\aligned
\laone \ &= \ \muone  \qquad \text{for all } \ \
	(\laone, \latwo) \in \supp(\fa) \ \, \text{ and } \ \
	(\muone, \mutwo) \in \supp(\fb), \\
\latwo \ &= \ \mutwo  \qquad \text{for all } \ \
	(\laone, \latwo) \in \supp(\fa) \ \, \text{ and } \ \
	(\muone, \mutwo) \in \supp(\fb).
\endaligned\right.
\end{equation*}
\end{lemma}

\smallskip

\begin{proof}
	it  follows from \eqref{eq:stable} and Lemma~\ref{lem:rosewood} that, for all   $(\lambda^1, \lambda^2), (\mu^1, \mu^2) \in \aL$,
	\begin{equation}\label{eq:stable-go}
		\fa(\lambda^1,\lambda^2) \. \fb(\mu^1,\mu^2)
		\ = \ \fc(\lambda^1,\mu^2) \. \fd(\mu^1,\lambda^2).
	\end{equation}
By the argument analogous to the proof of Theorem~\ref{thm:AD-boolean-2}, it then follows that there exist
 functions
\. $\ff_1, \fg_1: \aL_1 \to \rr_{\geq 0}$\.,   	\. $\ff_2, \fg_2: \aL_2 \to \rr_{\geq 0}$ \.    such that,
for all \.$(\laone,\latwo) \in \aL_1 \times \aL_2$\.,

\begin{equation}\label{eq:split-app}
	\left\{ \ \begin{split}
			\fa(\laone,\latwo) \ &= \ \ala \.  {\color{blue} \ff_1}(\laone) \. {\color{blue} \ff_2}(\latwo), \qquad
			&& \fb(\laone,\latwo) \ = \  \beb \. {\color{red}\fg_1}(\laone) \. {\color{red}\fg_2}(\latwo),\\
			\fc(\laone,\latwo) \ &= \  \gac \. {\color{blue}\ff_1}(\laone) \. {\color{red}\fg_2}(\latwo), \qquad
			&&\fd(\laone,\latwo) \ = \  \ded \. {\color{red}\fg_1}(\laone) \. {\color{blue}\ff_2}(\latwo),
\end{split}\right.
	\end{equation}

Now, suppose to the contrary that the claim in the lemma is false.
Then it follows from \eqref{eq:split-app}, that there exist \ts
$\laone \in \supp(\ff_1)$ \ts and \ts $\muone \in \supp(\fg_1),$ such that \ts $\laone\neq \muone$.
Similarly there exist \ts $\latwo \in \supp(\ff_2)$ \ts and \ts $\mutwo \in \supp(\fg_2)$,
such that \ts $\latwo\neq \mutwo$.  By \eqref{eq:split-app}, this implies that
\. $(\laone, \latwo) \in \supp(\fa)$ \.
and \. $(\muone, \mutwo) \in \supp(\fb)$, which contradicts Lemma~\ref{lem:rosewood}.
\end{proof}

\smallskip

By Lemma~\ref{lem:Rajan-2}, without loss
of generality,  we may assume that there exists a fixed element \ts $w_1 \in \aL_1$ \ts
 such that
\[ \laone \ = \ \muone \ = \ w_1 \qquad \text{for all } \ \
(\laone, \latwo) \in \supp(\fa) \ \, \text{ and } \ \
(\muone, \mutwo) \in \supp(\fb). \]
By \eqref{eq:split-app}, it follows that
\[ \laone \ = \ \muone \ = \ w_1 \qquad \text{for all } \ \
(\laone, \latwo) \in \supp(\fc) \ \  \text{ and } \ \
(\muone, \mutwo) \in \supp(\fd). \]
Our non-degeneracy assumption ensures these supports are non-empty,
allowing us to find elements \ts $x_2,y_2 \in \aL_2\ts,$ such that
\ts $(w_1, x_2) \in \supp(\fa)$ \ts and \ts $(w_1,y_2)\in \supp(\fb)$.
By \eqref{eq:stable-go}, this gives \ts $(w_1, y_2) \in \supp \fc$ \ts
and \ts $(w_1, x_2) \in \supp \fd$.

We now show that, for all $\lambda :=(\laone,\latwo) \in \aL_1 \times \aL_2\ts$, we have:
\begin{equation}\label{eq:ad-collinear}
 \fd(\lambda) \ = \ \frac{\fb(w_1,y_2)}{\fc(w_1,y_2)} \. \fa(\lambda).
\end{equation}
Indeed,
if $\laone \neq w_1$, then we have \. $\fd(\lambda) = \fa(\lambda) = 0$,
so \eqref{eq:ad-collinear} clearly holds. If \ts $\laone=w_1$, it follows  that
\[\fd(\lambda) \, = \,  \fd(\laone,\latwo) \, = \, \frac{\fb(\laone, y_2)}{\fc(w_1,y_2)} \.\fa(w_1,\latwo) \, = \, \frac{\fb(w_1, y_2)}{\fc(w_1,y_2)} \.\fa(\laone,\latwo)  \, = \, \frac{\fb(w_1,y_2)}{\fc(w_1,y_2)}\. \fa(\lambda),  \]
where the second equality is by \eqref{eq:stable-3}.
This completes the proof of \eqref{eq:ad-collinear}.

By an analogous argument, we have,
for all $\lambda \in \cY^{\ell,w}$,
\begin{equation*}
	\fc(\lambda) \ = \ \frac{\fa(w_1,x_2)}{\fd(w_1,x_2)} \fb(\lambda).
\end{equation*}
This implies that  \eqref{eq:ADS-eq-2} holds, and the proof
of Theorem~\ref{thm:ADS-eq} in case of a finite support is complete. \qed

\smallskip

\subsection{General case} \label{ss:ADS-general}
	Expanding the products in the \eqref{eq:ADS} and grouping the summands, we have:
	\begin{equation}\label{eq:expand-1}
	\begin{split}
	&	\sum_{\lambda \ts \in \ts\cY} \. \fc({\lambda}) \. \fs_{\lambda} \, \cdot \,
        \sum_{\lambda \ts \in \ts\cY} \. \fd({\lambda}) \. \fs_{\lambda} \  - \
        \sum_{\lambda \ts \in \ts\cY} \. \fa({\lambda}) \. \fs_{\lambda}  \, \cdot \,
        \sum_{\lambda \ts \in \ts\cY} \. \fb({\lambda}) \. \fs_{\lambda}  \\
	&	\qquad = \ \sum_{\lambda,\ts\mu \ts \in \ts \cY} \. \big(\fc({\lambda}) \. \fd(\mu) \ - \
                    \fa({\lambda}) \. \fb(\mu)\big) \. \fs_{\lambda} \. \fs_{\mu}  \\
	& \qquad = \ \sum_{\substack{(\om,\ts\pi) }} \ \ \sum_{\substack{(\lambda,\ts\mu) \ts \in \ts \cU(\om,\pi)}} \.
            \big(\fc(\lambda) \.  \fd(\mu)  \. - \.  \fa(\lambda) \. \fb(\mu)\big) \. \fs_{\la} \. \fs_{\mu}\ts.
\end{split}
\end{equation}
 where the outer sum is over all \. $\om,\pi \in \zz^\infty_{\geq 0}$ \. with finite support, and  the inner sum is over the set
$$
\cU(\om,\pi) \ := \
\big\{(\lambda,\mu)\in \cY \times \cY \. : \.
\lambda \vee \mu=\om+\pi, \.  \lambda \wedge \mu =\om \big\}.
$$

Now, define \.
$\Phi_{\fs}(\om,\pi) \. := \. \Phi_{\fs}(\om,\pi;\fa,\fb,\fc,\fd)$ \. to be the inner sum on the RHS of \eqref{eq:expand-1}:
\begin{equation*}
	\Phi_{\fs}(\om,\pi) \ := \ \sum_{(\lambda,\ts\mu) \ts  \in\ts\cU(\om,\pi)}  \big(\fc(\lambda)  \fd(\mu)  \. - \.  \fa(\lambda)  \fb(\mu)\big)  \. \fs_{\la/\mu} \. \fs_{\nu/\rho}.
\end{equation*}

We need the following Schur-positive version of Theorem~\ref{thm:RLS}.

\smallskip

\begin{thm}[{\rm \cite[Thm~8.1]{CCPS26}}{}]\label{thm:SAAD}
	Let  \ts $\om,\pi\in \zz_{\geq 0}^{\ell}$,
	let \. $\fa,\fb,\fc,\fd: \cY \times \cY \to \rr_{\geq 0}$ \.
	be four functions satisfying \eqref{eq:AD-cd}.
	Then we have:
	\begin{align}\label{eq:LSS} 
		\Phi_{\fs}(\om,\pi) \ \geqslant_{\fs} \ 0.
	\end{align}
\end{thm}

\smallskip

Now, since \eqref{eq:ADS} is an equality by assumption, it follows that the LHS
of \eqref{eq:expand-1} is~$0$.
Since \eqref{eq:expand-1} is a sum  of nonnegative Schur functions by Theorem~\ref{thm:SAAD}, it then follows that
\begin{equation}\label{eq:comp-zero}
		\Phi_{\fs}(\om,\pi) \ = \ 0,
\end{equation}
for all $\om,\pi \in \zz_{\geq 0}^\infty$ with finite support.

Now, fix \. $\ell, w\ge 1$.
By the same calculation as in \eqref{eq:expand-1}, we obtain
\begin{equation}\label{eq:expand-2}
		\begin{split}
		& \sum_{\lambda \in \ts\cY^{\ell,w}} \fc({\lambda}) \ts \fs_{\lambda} \. \cdot
\sum_{\lambda  \ts\in \ts\cY^{\ell,w}} \fd({\lambda}) \ts \fs_{\lambda} \, -
\sum_{\lambda \in \cY^{\ell,w}} \fa({\lambda}) \ts \fs_{\lambda} \. \cdot
\sum_{\lambda \in \ts\cY^{\ell,w}} \fb({\lambda}) \ts \fs_{\lambda}
		\ = \ \sum_{\substack{(\om,\ts\pi) }}  \Phi_{\fs}(\om,\pi),
	\end{split}
\end{equation}
where the sum is over all pairs $\om,\pi \in \zz_{\geq 0}^\ell$ such that $\om +\pi \in \{0,1,\ldots,w\}^\ell$.
It then follows from \eqref{eq:comp-zero} that the right hand side of \eqref{eq:expand-2} is equal to $0$.
Applying the finite support case above,
we conclude that the four functions \ts $\fa, \fb, \fc, \fd$ \ts restricted to
\ts $\cY^{\ell,w}$ \ts satisfy either \eqref{eq:ADS-eq-1} or \eqref{eq:ADS-eq-2}.
Taking the limit \ts $\ell, w \to \infty$, it follows
that the functions satisfy whichever condition holds for infinitely many
pairs $(\ell, w)$.  This completes the proof of Theorem~\ref{thm:ADS-eq}. \qed

\medskip

\medskip


{\small
\section{Final remarks}\label{sec:finrem}

\subsection{}\label{ss:finrem-why-eq}
Throughout the paper, we took it for granted that equality conditions are
important and worth the effort.  But are they really?
The honest answer is: \defna{that depends entirely on the area of mathematics}.
For \emph{basic analytic inequalities}, such as the Chebyshev inequality \eqref{eq:Cheb},
the equality conditions are usually quite simple and follow easily from the proofs.
Including them along with the statement/proof of the inequality as it is
done in \cite{HLP52}, is a well-established and widely used practice.  This is largely
because strict versions of these inequalities are often needed in more involved
inequalities.

In the opposite extreme, for \emph{geometric inequalities} \ts
discussed in~$\S$\ref{ss:intro-why-other}, the equality conditions
are fundamental and remain a major direction of study, see e.g.\
\cite{Gar02,Sch94}.  In \emph{calculus of variations}, the study of equality cases of
geometric inequalities, such as the \defng{Brunn--Minkowski inequality}, is really
the starting point, as the game is over getting better bounds for the stability
problems, see e.g.\ \cite{Fig14}.  This is similar to \emph{additive combinatorics},
where equality cases are closely related to basic structures, such as arithmetic
progressions or group cosets, and a great deal of effort is made to obtain bounds
on the distance from these structures to the near-equality cases, see e.g.\ \cite{TV06}.

In \emph{Riemannian and differential geometry}, the equality conditions
are also crucial, as they describe rigidity properties of the (often unique)
equality cases,  see e.g., numerous examples in \cite{Pet16}.
A different picture emerges in  \emph{extremal combinatorics}, where the
equality conditions turn out to be interesting structures,
such as Kneser graphs or set designs, see e.g.\ \cite{Bol86}.

Curiously, in \emph{percolation theory} and \emph{interacting particle systems},
the playground of many \eqref{eq:FKG} applications, the equality conditions seem
to be of little importance, cf.~\cite{Gri99,Lig85}.  That is largely because
for estimates one often need explicit lower bounds on the defect \ts $\de>0$,
see \cite{DNS21,Tal96}.  Unfortunately, we are nowhere close to getting sharp
stability result for the FKG inequality in full generality.

While our own motivation lies in applications to \emph{order theory} and
\emph{algebraic combinatorics}, another major reason for the study of
equality cases is a connection to \emph{computational complexity} \ts that we
briefly alluded to in~$\S$\ref{ss:intro-why-other}.  More precisely,
the computational hardness of equality cases is closely related to
whether the defect of the inequality has a combinatorial interpretation,
see \cite{Pak-OPAC} for the introduction and overview.

In summary, understanding the equality conditions goes beyond simply giving
an additional information about the inequality.  Often, and depending on the
area, they represent a new body of knowledge that is complementary to the
inequality in question.

\subsection{}\label{ss:finrem-app}
As the reader may have noticed, the applications of \eqref{eq:AD} and \eqref{eq:FKG}
that we chose to present are interrelated: \eqref{eq:LPP} is a special case of \eqref{eq:ADS},
implies both \eqref{eq:Oko} and \eqref{eq:Bjo}, while the latter is a special case of
\eqref{eq:Fis}.  This is largely a reflection of our own interest in these inequalities
and their equality conditions, but it is also an attempt to keep the presentation coherent.

There are numerous other examples that we could have chosen, many of which can now be
handled using the techniques developed in this paper. As the proofs of equality conditions
of \eqref{eq:LPP} and \eqref{eq:ADS} show, this does not mean that all such
applications would be easy or straightforward.  If anything, the opposite is true
(cf.\ Remark~\ref{rem:ADS-LPP}).
Here is a way to think about this
phenomenon: obtaining equality conditions for the AD inequality is like climbing a mountain.
Reaching the summit is a major step, but the descent on the other side --- applying
these equality conditions --- can be  remarkably difficult.

\subsection{}\label{ss:finrem-inj}
In some cases, log-concave inequalities can be proved by a direct injective
argument, see e.g.\ \cite{CPP23,DDP84,Kra96}, or a near-injective argument \cite{BT02}.
Often, such arguments can be used to obtain the equality conditions,
see a long list in \cite[$\S$14.1]{CP-surv}.  Note  that the monomially positive
version of \eqref{eq:LPP} was proved by Lam--Pylyavskyy \cite{LP07} by a direct
injection (see also Remark~\ref{rem:intro-LPP}).  It would be very
interesting to see if our Theorems~\ref{thm:LPP-eq} and~\ref{thm:SSYT-eq}
can be derived from this injection.

\subsection{}\label{ss:finrem-more}
Recall that \eqref{eq:LPP} can be phrased
as the inequality for LR coefficients:
\begin{equation}\label{eq:LPP-LR}
c^\la_{\mu,\nu} \, \le \,  c^\la_{\mu \vee \nu, \ts\mu \wedge \nu}\quad
\text{for all} \quad \la \vdash |\mu| + |\nu|\ts,
\end{equation}
see Remark~\ref{rem:LPP-LR}.  Note that Theorem~\ref{thm:LPP-eq} does not give
the equality conditions for \eqref{eq:LPP-LR}, since \eqref{eq:LPP}
is an aggregate of inequalities \eqref{eq:LPP-LR} over all~$\la$.
Recently, Speyer~\cite{Spe26} used \eqref{eq:AD} to prove an
inequality for LR coefficients which implies \eqref{eq:LPP-LR}.
It would be interesting to see if our Theorem~\ref{thm:AD-eq}
can be combined with Speyer's approach, to obtain the equality
conditions for \eqref{eq:LPP-LR}.

Let us also briefly mention unusual log-concave and log-supermodular
inequalities for the numbers of linear extensions given in \cite{CP-corr},
see especially $\S$4.2 for applications to standard Young tableaux.
The proofs are based on the combinatorial atlas technology, and the
equality conditions remain open.  It would be especially interesting
to find injective proofs of these inequalities, or show that none exist.

\subsection{}\label{ss:finrem-many}
There are several notable generalizations of the AD and FKG inequalities
to products of multiple functions and events, see \cite{RS93,AK96,Sahi08,LS22,Gla-FKG},
and even to log-supermodular functions on the products \cite{AZ22}.
It would be interesting to generalize the equality conditions
in Theorem~\ref{thm:AD-eq} to these general settings.  It would
be even more interesting to find a natural generalization of
\eqref{eq:AD} where the equality conditions are intractable,
cf.\ \cite{CP-coinc,CP-AF,CP-SY}.

\subsection{}\label{ss:finrem-hist}
The definitive history of the Chebyshev inequality \eqref{eq:Cheb} is given
in \cite{MV74}.  It seems, the integral version is indeed due Chebyshev (1882),
much popularized by Hermite in his analysis course (1883)
with a different proof by Picard, while
the discrete version is due to Jensen (1888).  Note that Hardy--Littlewood--P\'olya
credited both inequalities to Chebyshev \cite[$\S$2.17]{HLP52}.  Curiously,
Seymour--Welsh derived from \eqref{eq:FKG} only a special case of \eqref{eq:Cheb}
for log-convex $\{p_i\}$, which they called a ``new inequality about
log convex sequences'' \cite[$\S$3]{SW75}.  They credited Chebyshev with the
uniform case \ts $p_i=1/n$, an assertion refuted in \cite{Gra83}.

The FKG inequality has its own interesting history which was thoroughly
investigated by Grimmett in \cite[App.]{Gri06}, including the
reason why the authors are listed in this order.  The backstory
of the AD inequality is also somewhat amusing, including a tidbit about
Ahlswede stating the inequality to Daykin while standing on a ladder
\cite[p.~X]{Ahl18}.

\subsection{}\label{ss:finrem-quote}
As we discussed above, we make a special effort to prove a number of
applications of a single theorem (Theorem~\ref{thm:AD-eq}).  This is
in part because we needed these applications for other purposes, but
also to justify the extensive study of the AD equality to begin with.
Compare this to the following sentiment by Bollob\'as, who wrote a
lengthy section of his monograph about the AD inequality (=~FFT)
and its applications:

\begin{center}\begin{minipage}{13.1cm}%
\emph{``At the first glance the FFT looks too general to be true and, if true,
it seems too vague to be of much use. In fact, exactly the opposite is
true: the Four Functions Theorem (FFT) of Ahlswede and Daykin is a
theorem from `the book'. It is beautifully simple and goes to the heart
of the matter. Having proved it, we can sit back and enjoy its power
enabling us to deduce a wealth of interesting results.'' \ts  \cite[$\S$19]{Bol86}.}
%
\end{minipage}\end{center}

\vskip.7cm

\subsection*{Acknowledgements}
We are grateful to Noga Alon, B\'ela Bollob\'as, 
Hong Chen, Nikita Gladkov, Tom Hutchcroft, Jeff Kahn, Thomas Lam, 
Yuval Peres, Pasha Pylyavskyy, Siddhartha Sahi, Daniel Soskin and 
Ramon van Handel for interesting conversations and helpful remarks.
The first author (SHC) was partially supported by the NSF grant DMS-2246845,
and would like to thank the National University of Singapore for their
warm hospitality during his sabbatical in the Spring of 2026.
The second author (IP) was partially supported by the NSF grant CCF-2302173.
}

\vskip1.1cm


\vskip.5cm


{\footnotesize
	
\begin{thebibliography}{abcdefgh}

\bibitem[AZ22]{AZ22}
Dimitris~Achlioptas and Kostas~Zampetakis,
A simpler proof of the four functions theorem and some new variants,
in \emph{Proc.\ 2022~ISIT}, IEEE, 2022, 714--717.

\bibitem[AH93]{AH93}
Ron~Aharoni and Ron~Holzman, Two and a half remarks on the Marica--Sch\"onheim
inequality, \emph{J.\ Lond.\ Math.\ Soc.} \textbf{48} (1993), 385--395.

\bibitem[AK96]{AK96}
Ron~Aharoni and Uri~Keich,
A generalization of the Ahlswede--Daykin inequality,
\emph{Discrete Math.} \textbf{152} (1996), 1--12.

\bibitem[Ahl18]{Ahl18}
Rudolf~Ahlswede, \emph{Combinatorial methods and models},
Springer, Cham, 2018, 385~pp.

\bibitem[AB08]{AB08}
Rudolf~Ahlswede and Vladimir~Blinovsky,
Lectures on advances in combinatorics,
Springer, Berlin, 2008, 314~pp.

\bibitem[AD78]{AD78}
Rudolf~Ahlswede and David~E.~Daykin,
An inequality for the weights of two families of sets, their unions and intersections,
\emph{Z.~Wahrsch.\ Verw.\ Gebiete}~\textbf{43} (1978), 183--185.

\bibitem[AK95]{AK95}
Rudolf~Ahlswede and Levon~H.~Khachatrian,
Towards characterizing equality in correlation inequalities,
\emph{European J.\ Combin.}~\textbf{16} (1995), 315--328.

\bibitem[Ale37]{Ale37}
Alexander~D.~Alexandrov, To the theory of mixed volumes of convex bodies~II.\ New inequalities between
mixed volumes and their applications (in Russian),  \emph{Mat.\ Sb.}~\textbf{2} (1937), 1205--1238;
English translation in A.~D.~Alexandrov, \emph{Selected works}, Part~I, Ch.~IV, CRC Press, 1996, 61--97.

\bibitem[AS16]{AS16}
Noga~Alon and Joel~H.~Spencer, \emph{The probabilistic method} (Fourth ed.),
John Wiley, Hoboken, NJ, 2016, 375~pp.

\bibitem[And87]{And87}
Ian~Anderson, \emph{Combinatorics of finite sets},
Oxford Univ.\ Press, New York, 1987, 250~pp.

\bibitem[BM92]{BM92}
Dominique~Bakry and Dominique~Michel,
Sur les in\'egalit\'es FKG (in French), in \emph{Lecture Notes in Math.}
\textbf{1526}, Springer, Berlin, 1992, 170--188.

\bibitem[Beck90]{Bec90}
Istv\'an~Beck, An inequality for partially ordered sets,
\emph{J.\ Combin.\ Theory, Ser.~A} \textbf{53} (1990), 123--138.


\bibitem[BTW06]{BTW06}
Louis~J.~Billera, Hugh~Thomas and Stephanie~van~Willigenburg,
Decomposable compositions, symmetric quasisymmetric functions and equality
of ribbon Schur functions, \emph{Adv.\ Math.}~\textbf{204} (2006), 204--240.

\bibitem[Bir33]{Bir}
Garrett~Birkhoff,
On the combination of subalgebras, \emph{Proc.\ Camb.\ Phil.\ Soc.}
\textbf{29} (1933), 441--464.
%

\bibitem[Bj\"o11]{Bjo11}
Anders~Bj\"{o}rner, A $q$-analogue of the FKG inequality and some applications,
\emph{Combinatorica}~\textbf{31} (2011), 151--164.

\bibitem[Bol86]{Bol86}
B\'ela~Bollob\'as, \emph{Combinatorics},
Cambridge Univ.\ Press, Cambridge, UK, 1986, 177~pp.

\bibitem[BR06]{BR06}
B\'{e}la~Bollob\'as  and Oliver~Riordan,
\emph{Percolation}, Cambridge Univ.\ Press, New York, 2006, 323~pp.

\bibitem[BL26]{BL26}
Petter~Br\"and\'en and Jonathan~Leake,
Lorentzian polynomials on cones,
\emph{Forum Math.\ Sigma} \textbf{14} (2026),
Paper No.~e16, 34~pp.

\bibitem[Bri88]{Bri88}
Graham~Brightwell,
Linear extensions of infinite posets,
\emph{Discrete Math.}~\textbf{70} (1988), 113--136.

\bibitem[Bri90]{Bri90}
Graham~R.~Brightwell,
Events correlated with respect to every subposet of a fixed poset,
\emph{Graphs Combin.} \textbf{6} (1990), 111--131.

\bibitem[BT02]{BT02}
Graham~R.~Brightwell and William~T.~Trotter,
A combinatorial approach to correlation inequalities,
\emph{Discrete Math.}~\textbf{257} (2002), 311--327.

\bibitem[BZ88]{BZ-book}
Yuri~D.~Burago and  Victor~A.~Zalgaller,
\emph{Geometric inequalities},
Springer, Berlin, 1988, 331~pp.

\bibitem[CCPS26]{CCPS26}
Swee~Hong~Chan, Hong~Chen, Igor~Pak and Daniel~Soskin,
Correlation inequalities for Schur positivity,
preprint (2026), 41~pp.; \ts {\tt arXiv:2606.06688}.

\bibitem[CP22]{CP-intro}
Swee~Hong~Chan and Igor~Pak, Introduction to the combinatorial atlas,
 \emph{Expo.\ Math.}~\textbf{40} (2022), 1014--1048.

\bibitem[CP23a]{CP-multi}
Swee~Hong~Chan and Igor~Pak,
Multivariate correlation inequalities for $P$-partitions,
\emph{Pacific J.\ Math.} \textbf{323} (2023), 223--252.

\bibitem[CP23b]{CP-surv}
Swee~Hong~Chan and Igor~Pak,
Linear extensions of finite posets, preprint (2023), 56~pp,
to appear in \emph{EMS Surv.\ Math.\ Sci.}; \ts {\tt arXiv:2311.02743}.

\bibitem[CP24a]{CP-atlas}
Swee~Hong~Chan and Igor~Pak, Log-concave poset inequalities,
\emph{Jour.\ Assoc.\ Math.\ Res.}~\textbf{2} (2024), 53--153.

\bibitem[CP24b]{CP-corr}
Swee~Hong~Chan and Igor~Pak,
Correlation inequalities for linear extensions,
\emph{Adv.\ Math.} \textbf{458} (2024), Paper~109954, 33~pp.

\bibitem[CP24c]{CP-coinc}
Swee~Hong~Chan and Igor~Pak,
Computational complexity of counting coincidences,
\emph{Theoret.\ Comput.\ Sci.} \textbf{1015} (2024), Paper No.~114776, 19~pp.

\bibitem[CP24d]{CP-AF}
Swee~Hong~Chan and Igor~Pak, Equality cases of the Alexandrov--Fenchel inequality
are not in the polynomial hierarchy, \emph{Forum Math.~Pi}~\textbf{12} (2024),
Paper No.~e21, 38~pp.

\bibitem[CP24e]{CP-SY}
Swee~Hong~Chan and Igor~Pak,
Equality cases of the Stanley--Yan log-concave matroid inequality,
preprint (2024), 36~pp.; \ts {\tt arXiv:2407.19608}.

\bibitem[CP26+]{CP-perm}
Swee~Hong~Chan and Igor~Pak,
Equality cases of the quadratic permanent inequality, in preparation (2026).

\bibitem[CPP23]{CPP23}
Swee~Hong~Chan, Igor~Pak and Greta~Panova, Effective poset inequalities,
\emph{SIAM J.\ Discrete Math.} \textbf{37} (2023), 1842--1880.

\bibitem[Day77]{Day77}
David~E.~Daykin,
A lattice is distributive if and only if \ts
$|A|\cdot |B| \leq |A\vee B|\cdot |A\wedge B|$,
\emph{Nanta Math.} \textbf{10} (1977), 58--60.

\bibitem[DDP84]{DDP84}
David~E.~Daykin,  Jacqueline~W.~Daykin  and Michael~S.~Paterson,
On log concavity for order-preserving maps of partial orders,
\emph{Discrete Math.}~\textbf{50} (1984), 221--226.

\bibitem[DKW79]{DKW79}
David~E.~Daykin, Daniel~J.~Kleitman and Douglas~B.~West,
The number of meets between two subsets of a lattice,
\emph{J.\ Combin.\ Theory, Ser.~A} \textbf{26} (1979), 135--156.

\bibitem[DL76]{DL76}
David~E.~Daykin and L\'aszl\'o~Lov\'asz,
The number of values of a Boolean function,
\emph{J.\ London Math. Soc.} \textbf{12} (1975), 225--230.

\bibitem[DNS21]{DNS21}
Anindya~De, Shivam~Nadimpalli and Rocco~A.~Servedio,
Quantitative correlation inequalities via semigroup interpolation,
in \emph{Proc.\ 12th ITCS} (2021), Art.\ No.~69, 20~pp.

\bibitem[Eng97]{Eng97}
Konrad~Engel, \emph{Sperner theory},
Cambridge Univ.\ Press, Cambridge, 1997, 417~pp.

\bibitem[F+17]{F+17}
Shaun~Fallat, Steffen~Lauritzen,  Kayvan~Sadeghi,
Caroline~Uhler, Nanny~Wermuth and Piotr Zwiernik,
Total positivity in Markov structures,
\emph{Ann.\ Statist.} \textbf{45} (2017), 1152--1184.

\bibitem[Fig14]{Fig14}
Alessio~Figalli,
Quantitative stability results for the Brunn--Minkowski inequality,
in \emph{Proc.\ ICM Seoul}, Vol.~III, Kyung Moon Sa, Seoul, 2014, 237--256.

\bibitem[Fis84]{Fish84}
Peter~C.~Fishburn,
A correlational inequality for linear extensions of a poset,
\emph{Order}~\textbf{1} (1984), 127--137.

\bibitem[FS00]{FS00}
Peter~C.~Fishburn and Lawrence~A.~Shepp,
The Ahlswede--Daykin theorem,
in \emph{Numbers, information and complexity},
Kluwer, Boston, MA, 2000, 501--516.

\bibitem[FKG71]{FKG71}
Cornelius~M.~Fortuin, Pieter~W.~Kasteleyn and Jean~Ginibre,
Correlation inequalities on some partially ordered sets,
\emph{Comm.\ Math.\ Phys.}~\textbf{22} (1971), 89--103.

\bibitem[Gar02]{Gar02}
Richard~J.~Gardner, The Brunn--Minkowski inequality,
\emph{Bull.\ AMS}~\textbf{39} (2002), 355--405.

\bibitem[Gla24a]{Gla-FKG}
Nikita~Gladkov,
A strong FKG inequality for multiple events,
\emph{Bull.\ LMS} \textbf{56} (2024), 2794--2801.

\bibitem[Gla25]{Gla-thesis}
Nikita~Gladkov,
Inequalities for connectivity events in Bernoulli percolation,
Ph.D.\ thesis, UCLA, 2025, 135~pp.

\bibitem[Gra83]{Gra83}
Ronald~L.~Graham,
Applications of the FKG inequality and its relatives, in
\emph{Mathematical programming: the state of the art},
Springer, Berlin, 1983, 115--131.

\bibitem[GYY80]{GYY80}
Ronald~L.~Graham, Andrew~C.~Yao and Frances~F.~Yao,
Some monotonicity properties of partial orders,
\emph{SIAM J. Algebraic Discrete Methods}~\textbf{1} (1980), 251--258.

\bibitem[Gr\"a98]{Gra98}
George~Gr\"atzer,
\emph{General lattice theory} (second ed.), Birkh\"auser, Basel, 1998, 663~pp.

\bibitem[Gri99]{Gri99}
Geoffrey~R.~Grimmett,
\emph{Percolation} (second ed.), Springer, Berlin, 1999, 444~pp.

\bibitem[Gri06]{Gri06}
Geoffrey~R.~Grimmett,
\emph{The random-cluster model}, Springer, Berlin, 2006, 377~pp.

\bibitem[Gri18]{Gri18}
Geoffrey~R.~Grimmett, \emph{Probability on graphs. Random processes on graphs and lattices}
(second ed.), Cambridge Univ. Press, Cambridge, UK, 2018. 265~pp.

\bibitem[HLP52]{HLP52}
Godfrey~H.~Hardy, John~E.~Littlewood and George~P\'olya,
\emph{Inequalities} (Second ed.),
Cambridge Univ.\ Press, 1952, 324~pp.

\bibitem[Har60]{Har60}
Theodore~E.~Harris,
A lower bound for the critical probability in a certain percolation process,
\emph{Math.\ Proc.\ Camb.\ Philos.\ Soc.}~\textbf{56} (1960), 13--20.

\bibitem[HK86]{HK86}
W.~Th.~Frank~den~Hollander and Michael~S.~Keane,
Inequalities of FKG type,
\emph{Phys.~A} \textbf{138} (1986), 167--182.

\bibitem[Hol74]{Hol74}
Richard~Holley,
Remarks on the FKG inequalities,
\emph{Comm.\ Math.\ Phys.}~\textbf{36} (1974), 227--231.

\bibitem[HW20]{HW20}
Daniel~Hug and Wolfgang~Weil, \emph{Lectures on convex geometry},
Springer, Cham, 2020, 287~pp.

\bibitem[Huh18]{Huh18}
June~Huh,
Combinatorial applications of the Hodge--Riemann relations, in
\emph{Proc.\ ICM Rio de Janeiro}, vol.~IV, World Sci.,
Hackensack, NJ, 2018, 3093--3111.


\bibitem[IP22]{IP22}
Christian~Ikenmeyer and Igor~Pak,
What is in~$\SP$ and what is not?,  preprint (2022),
82~pp.; extended abstract
in \emph{Proc.\ 63rd FOCS} (2022), 860--871; \ts {\tt arXiv:2204.13149}.




\bibitem[KKM16]{KKM16}
Gil~Kalai, Nathan~Keller and Elchanan~Mossel,
On the correlation of increasing families,
\emph{J.~Combin.\ Theory, Ser.~A} \textbf{144} (2016), 250--276.

\bibitem[KMS14]{KMS14}
Nathan~Keller, Elchanan~Mossel and Arnab~Sen,
Geometric influences II: correlation inequalities and noise sensitivity,
\emph{Ann.\ Inst.\ Henri Poincar\'e, Probab.\ Stat.}
\textbf{50} (2014), 1121--1139.

\bibitem[Kes82]{Kes82}
Harry~Kesten, \emph{Percolation theory for mathematicians},
Birkh\"auser, Boston, MA, 1982, 423~pp.

\bibitem[Kle66]{Kle66}
Daniel~J.~Kleitman, Families of non-disjoint subsets,
\emph{J.\ Combin.\ Theory}~\textbf{1} (1966), 153--155.

\bibitem[Kra96]{Kra96}
Christian~Krattenthaler,
Combinatorial proof of the log-concavity of the sequence
of matching numbers,
\emph{J.~Combin.\ Theory Ser.~A}~\textbf{74} (1996), 351--354.

\bibitem[LPP07]{LPP07}
Thomas~Lam, Alexander~Postnikov and Pavlo~Pylyavskyy,
Schur positivity and Schur log-concavity,
\emph{Amer.~J.\ Math.}~\textbf{129} (2007), 1611--1622.

\bibitem[LP07]{LP07}
Thomas~Lam and Pavlo~Pylyavskyy, Cell transfer and monomial positivity,
\emph{J.~Algebraic Combin.}~\textbf{26} (2007), 209--224.

\bibitem[LS22]{LS22}
Elliott~H.~Lieb and Siddhartha~Sahi,
On the extension of the FKG inequality to $n$ functions,
\emph{J.~Math.\ Phys.} \textbf{63} (2022), no.~4, Paper No.~043301, 11~pp.

\bibitem[Lig85]{Lig85}
Thomas~M.~Liggett,  \emph{Interacting particle systems},
Springer, New York, 1985, 488~pp.



\bibitem[LS06]{LS06}
L\'aszl\'o~Lov\'asz and Michael~Saks,
A localization inequality for set functions,
\emph{J.\ Combin.\ Theory, Ser.~A} \textbf{113} (2006), 726--735.


\bibitem[Mac95]{Mac}
Ian~G.~Macdonald, \emph{Symmetric functions and Hall polynomials}
(Second ed.), Oxford Univ.\ Press, New York, 1995, 475~pp.

\bibitem[MS69]{MS69}
John~G.~Marica and Johanan Sch\"onheim,
Differences of sets and a problem of Graham,
\emph{Canad.\ Math.\ Bull.} \textbf{12} (1969), 635--637.

\bibitem[MR94]{MR94}
Daniel~McQuillan and R.~Bruce~Richter,
Equality in a result of Kleitman,
\emph{J.~Combin.\ Theory, Ser.~A} \textbf{65} (1994), 330--333.

\bibitem[MV74]{MV74}
Dragoslav~S.~Mitrinovi\'{c} and Petar~M.~Vasi\'{c},
History, variations and generalisations of the \v{C}eby\v{s}ev inequality and the question of some priorities,
\emph{Univ.\ Beograd.\ Publ.\ El.\ Fak.} \textbf{461} (1974), 1--30.


\bibitem[Oko97]{Oko97}
Andrei~Okounkov,
Log-concavity of multiplicities with applications to characters of~$U(\infty)$,
\emph{Adv.\  Math.} \textbf{127} (1997), 258--282.

\bibitem[Oss79]{Oss79}
Robert~Osserman,
Bonnesen--style isoperimetric inequalities,
\emph{Amer.\ Math.\ Monthly}~\textbf{86} (1979), 1--29.

\bibitem[Pak22]{Pak-OPAC}
Igor~Pak, What is a combinatorial interpretation?, in
\emph{Open Problems in Algebraic Combinatorics}, AMS, Providence, RI, 2024, 191--260.

\bibitem[PPY19]{PPY19}
Igor~Pak, Greta~Panova and Damir~Yeliussizov,
On the largest Kronecker and Littlewood--Richardson coefficients,
\emph{J.\ Combin.\ Theory, Ser.~A}~\textbf{165} (2019), 44--77.

\bibitem[PS26+]{PS26}
Igor~Pak and Daniel~Soskin,
Equalities and inequalities for products of Schur functions,
in preparation (2026).

\bibitem[PO80]{PO80}
Michael~D.~Perlman and Ingram~Olkin,
Unbiasedness of invariant tests for MANOVA and other multivariate problems,
\emph{Ann.\ Statist.} \textbf{8} (1980), 1326--1341.

\bibitem[Pet16]{Pet16}
Peter~Petersen, \emph{Riemannian geometry} (third ed.),
Springer, Cham, 2016, 499~pp.

\bibitem[Raj04]{Raj}
Conjeeveram~S.~Rajan,
Unique decomposition of tensor products of irreducible representations of simple algebraic groups,
\emph{Annals of Math.} \textbf{160} (2004), 683--704.

\bibitem[RSW09]{RSW09}
Victor~Reiner, Kristin~M.~Shaw and Stephanie~van~Willigenburg,
Coincidences among skew Schur functions,
\emph{Adv.\ Math.}~\textbf{216} (2007), 118--152;
Corrigendum in \emph{Adv.\ Math.}~\textbf{220} (2009), 1655--1656.

\bibitem[Reu87]{Reu87}
Klaus~Reuter,
Note on the Ahlswede--Daykin inequality,
\emph{Discrete Math.} \textbf{65} (1987), 209--212.

\bibitem[RS93]{RS93}
Yosef~Rinott and Michael~Saks,
Correlation inequalities and a conjecture for permanents,
\emph{Combinatorica} \textbf{13} (1993), 269--277.

\bibitem[Sag01]{Sag01}
Bruce~E.~Sagan,
\emph{The symmetric group}, Springer, New York, 2001, 238~pp.

\bibitem[Sahi08]{Sahi08}
Siddhartha~Sahi,
Higher correlation inequalities, \emph{Combinatorica}~\textbf{28} (2008), 209--227.

\bibitem[Sch85]{Sch85}
Rolf~Schneider,
On the Aleksandrov--Fenchel inequality, in \emph{Discrete geometry and convexity},
New York Acad.\ Sci., New York, 1985, 132--141.

\bibitem[Sch94]{Sch94}
Rolf~Schneider,
Equality in the Aleksandrov--Fenchel inequality --- present state and new results,
in  \emph{Intuitive geometry}, North-Holland, Amsterdam, 1994, 425--438; available
 at \ts \href{https://tinyurl.com/5bwrvu9n}{tinyurl.com/5bwrvu9n}

\bibitem[Sch14]{Sch14}
Rolf~Schneider,
{\em Convex bodies: the Brunn--Minkowski theory} (second ed.),
Cambridge Univ.~Press, Cambridge, UK, 2014, 736~pp.

\bibitem[Sey73]{Sey73}
Paul~D.~Seymour, On incomparable collections of sets,
\emph{Mathematika}~\textbf{20} (1973), 208--209.

\bibitem[SW75]{SW75}
Paul~D.~Seymour and Dominic~J.~A.~Welsh,
Combinatorial applications of an inequality from statistical mechanics,
\emph{Math.\ Proc.\ Cambridge Philos.\ Soc.} \textbf{77} (1975), 485--495.

\bibitem[SvH19]{SvH-pams}
Yair~Shenfeld and Ramon~van~Handel,
Mixed volumes and the Bochner method,
\emph{Proc.\ AMS}~\textbf{147} (2019), 5385--5402.

\bibitem[SvH22]{SvH-duke}
Yair~Shenfeld and Ramon~van~Handel,
The extremals of Minkowski's quadratic inequality,
\emph{Duke Math.~J.} \textbf{171} (2022), 957--1027.

\bibitem[SvH23]{SvH-acta}
Yair~Shenfeld and Ramon~van~Handel,
The extremals of the Alexandrov--Fenchel inequality for convex polytopes,
\emph{Acta Math.}~\textbf{231} (2023), 89--204.

\bibitem[SvH24]{SvH-icbs}
Yair~Shenfeld and Ramon~van~Handel,
The Alexandrov--Fenchel inequality, in
\emph{Proc.\ ICBS} (Frontiers of Science Award), Int.\ Press,
Somerville, MA, 2024, 15~pp.;
available at \ts \href{https://web.math.princeton.edu/~rvan/icbs240806.pdf}{tinyurl.com/23e2pvht}

\bibitem[She80]{She80}
Lawrence~A.~Shepp,
The FKG inequality and some monotonicity properties of partial orders,
\emph{SIAM J.\ Algebraic Discrete Methods}~\textbf{1} (1980), 295--299.


\bibitem[Spe26]{Spe26}
David~E~Speyer,
$L$-log-concavity and a proof of the conjecture of Lam, Postnikov and Pylyavskyy,
preprint (2026), 26 pp.; \ts {\tt arXiv:2601.05007}.



\bibitem[Sta99]{Sta-EC}
Richard~P.~Stanley, {\em Enumerative Combinatorics}, vol.~1 (second edition)
and vol.~2, Cambridge Univ.~Press, 2012 and~1999.

\bibitem[Tal96]{Tal96}
Michel~Talagrand,
How much are increasing sets positively correlated?,
\emph{Combinatorica} \textbf{16} (1996), 243--258.


\bibitem[TV06]{TV06}
Terence~Tao and Van~Vu, \emph{Additive combinatorics},
Cambridge Univ.\ Press, Cambridge, 2006, 512~pp.


\bibitem[Wer09]{Wer09}
Wendelin~Werner, \emph{Percolation et mod\`ele d'Ising} (in French),
Soc.\ Math.\ de France, Paris, 2009, 161~pp.

\bibitem[West21]{West21}
Douglas~B.~West,
\emph{Combinatorial mathematics},
Cambridge Univ.\ Press, Cambridge, UK, 2021, 969~pp.

\bibitem[Win10]{Win10}
Peter Winkler,
A stronger form of the van den Berg--Kesten inequality,
preprint (2010), 10~pp.; available at \ts \href{https://math.dartmouth.edu/~pw/Math100/sample.pdf}{tinyurl.com/y8vmfhna}

\bibitem[YK00]{YK}
Wei-Shih~Yang and David~Klein,
An ``FKG equality'' with applications to random environments,
\emph{Statist.\ Probab.\ Lett.}~\textbf{46} (2000), 203--209.



\end{thebibliography}
}
\end{document}